\documentclass[preprint,12pt,compress]{elsarticle}
\usepackage{cancel}
\usepackage{amsfonts}
\usepackage{xcolor}
\usepackage{mathrsfs}
\usepackage{amsmath,amsfonts,amssymb,amscd,amsthm, mathrsfs}
\usepackage{extarrows}
\usepackage{lineno}
\usepackage{hyperref}					
\usepackage{verbatim}                   
\usepackage{booktabs}                   
\usepackage{float}
\usepackage{subcaption}

\usepackage{graphicx}
\usepackage{epsfig}
\usepackage{epstopdf}
\usepackage[normalem]{ulem}
\usepackage{tabularx}  

\hypersetup{colorlinks,
	linkcolor=blue,%
	citecolor=blue}

\newtheoremstyle{kai}
{3pt}{3pt}{}{}{\bfseries}{.}{.5em}{}
\makeatletter
\def\EquationsBySection{\def\theequation
	{\thesection.\arabic{equation}}%
	\@addtoreset{equation}{section}}

\newcommand\old[1]{}

\newtheorem{theorem}{Theorem}[section] 
\newtheorem{definition}[theorem]{Definition} 
\newtheorem{lemma}[theorem]{Lemma}
\newtheorem{assumption}[theorem]{Assumption}

\newtheorem{proposition}[theorem]{Proposition}
\newtheorem{remark}[theorem]{Remark}
\newtheorem{example}{Example}[section]

\journal{XXX} 
\EquationsBySection
\makeatother

\begin{document}

\begin{frontmatter}
	\title{
	Adaptive Time-Stepping Euler--Maruyama Scheme for SDEs with Non-Globally Lipschitz Coefficients: Uniform Convergence, Stability and Ergodicity
    }

\author[aff1]{Xueqi Wen}
\author[aff2]{Shan Huang}
\author[aff1]{Xiaoyue Li\corref{cor1}}

\address[aff1]{School of Mathematical Sciences,
	Tiangong University, Tianjin 300387, China}

\address[aff2]{School of Mathematics and Statistics,
	Northeast Normal University, Changchun, Jilin 130024, China}

\cortext[cor1]{Corresponding author.}

\ead{lixy@tiangong.edu.cn}

\fntext[funding]{  	
	Research of Xiaoyue Li was supported by the National Natural Science Foundation
	of China (No. 12371402) and the Tianjin Natural Science Foundation (24JCZDJC00830).}
    

\begin{abstract}
This paper develops an adaptive time-stepping Euler--Maruyama scheme for stochastic differential equations (SDEs) with non-globally Lipschitz drift and diffusion coefficients. By dynamically adjusting the timestep at each iteration, the proposed scheme effectively prevents numerical instability.
We prove the moment boundedness of the numerical solution and establish a $1/2$-order strong convergence rate both on finite-time intervals and uniformly in time. Furthermore, the scheme faithfully inherits the $p$th moment exponential stability of the underlying SDE. For long-time ergodic dynamics, we establish the polynomial ergodicity of the numerical invariant measure. Moreover, we show that the numerical invariant measure converges to the invariant measure of the underlying SDE at an optimal rate of $1/2$ in the $L^q$-Wasserstein distance. 
Numerical experiments confirm our theoretical results and indicate the superior accuracy and computational performance of the proposed scheme over several fixed-step and existing adaptive methods.

 \end{abstract}
	
	
\end{frontmatter}

\section{Introduction} \label{intro}
Sampling from a prescribed probability distribution plays an important role in
computational statistics, statistical physics and machine learning 
\cite{metropolis1953equation,geman1984stochastic,tierney1994markov,jones2022markov, heber2020posterior}.
A prevailing approach is to construct an ergodic diffusion process that admits the target distribution as its invariant measure.
A prototypical example is the overdamped Langevin equation
\begin{equation} \label{overdamped}
	dX(t) =  - \nabla V(X(t))dt + \sqrt {2{\beta ^{ - 1}}} dW(t), \quad t \ge 0,
\end{equation}
where $V:\mathbb{R}^d \to \mathbb{R}$ is a potential function, $W(t)$ is a $d$-dimensional Brownian motion, and $\beta >0$ denotes the inverse temperature parameter. Under suitable smoothness and confining assumptions on $V(\cdot)$, the Boltzmann--Gibbs distribution with density
\begin{equation}\label{Gibbs}
	\rho(x)=\frac{1}{Z}\exp\bigl(-\beta V(x)\bigr), \quad Z = \int_{{\mathbb{R}^d}} {\exp ( - \beta V(x))dx}
\end{equation}
is the unique invariant probability measure of stochastic differential equation (SDE) \eqref{overdamped} (see e.g., \cite[Proposition 4.2]{Pavliotis2014stochastic}). In many practical applications, however, the normalizing constant $Z$ is not available in closed form, as it involves a high-dimensional integral. 
This motivates the construction of numerical schemes that preserve the relevant long-time behaviour and whose invariant measures accurately approximate the target distribution. 
It is well-known that explicit approximation methods, such as Euler--Maruyama (EM), may fail to be ergodic for SDEs with local Lipschitz coefficients, even when the underlying SDE is exponentially ergodic (see e.g., \cite{roberts1996exponential, mattingly2002ergodicity}). In this paper, we aim to develop an adaptive time-stepping EM scheme, and use its numerical invariant to approximate that of the underlying SDE.

Adaptive time-stepping methods, which prevent numerical explosions by dynamically adjusting timestep based on the current state, have been extensively studied over the past two decades. 
Nonetheless, as noted by Lamba, Mattingly and Stuart \cite[p. 479]{lamba2007an}, ``the understanding of adaptive algorithms for SDEs is an open area, where many issues related to both convergence and stability (long-time behaviour) of algorithms are unresolved.''
In particular, results on uniform-in-time strong convergence and stability remain sparse.
Fang and Giles \cite{fang2020adaptive} proposed an adaptive EM scheme for SDEs with one-sided Lipschitz drift, and globally bounded and Lipschitz continuous diffusion coefficients. 
Without requiring a uniform positive minimum stepsize, they obtained a $1/2$-order uniform-in-time strong convergence rate, and a first-order for Langevin-type SDEs.  
Kieu et al. \cite{kieu2022strong} introduced a tamed-adaptive EM scheme for a L\'evy-driven SDE under locally Lipschitz drift and locally H\"older continuous coefficients. The $L^1$ uniform-in-time convergence rate was studied under some regularity of the coefficients.

Despite these advances, the uniform-in-time convergence analysis for SDEs with fully locally Lipschitz drift and diffusion coefficients remains largely open. 
Moreover, stability is another key component of long-time numerical analysis. To the best of our knowledege, moment exponential stability in the adaptive time-stepping setting has rarely been investigated.
To bridge these gaps, this paper proposes an adaptive time-stepping EM scheme for SDEs with fully locally Lipchitz coefficients, establishing its uniform-in-time strong convergence theory while simultaneously reproducing the moment exponential stability of exact solution of the SDE. 

In sampling, time discretizations of SDEs are widely used to construct Markov chains targeting a given distribution.
It is thus crucial that the numerical schemes preserve the ergodic properties of the continuous-time dynamics \cite{lamba2007an}. 
For SDEs with non-globally Lipschitz drift and bounded diffusion, Lamba et al. \cite{lamba2007an} developed an adaptive scheme that controls the drift contribution at each time step, and proved its exponential ergodicity.
Lemaire \cite{lemaire2007an} constructed an explicit adaptive EM scheme for dissipative SDE with locally Lipschitz drift, in which the adaptive timestep is chosen as the minimum of a deterministic decreasing stepsize and a state-dependent stability bound. Under suitable Lyapunov and growth conditions, the adaptive timesteps eventually coincide with the decreasing stepsize almost surely, and every weak limit of the associated weighted empirical measures is an invariant measure of the underlying SDE. 
For ergodic SDEs with one-sided Lipschitz drift and globally Lipschitz, bounded diffusion coefficients, Fang and Giles \cite{fang2020adaptive} proposed an adaptive multilevel Monte Carlo method to compute expectations with respect to the invariant measure. 
More recently, Leroy et al. \cite{leroy2024adaptive} introduced an invariant measure-preserving transformation for overdamped and underdamped Langevin dynamics based on a state-dependent time rescaling. Since direct time rescaling changes the invariant measure, they introduced a correction term to ensure that the transformed dynamics still preserve the target Gibbs distribution.
Despite these advances, the preservation of ergodicity by adaptive schemes for SDEs with both locally Lipschitz drift and diffusions remains an open problem. In this paper, we propose an adaptive time-stepping EM scheme for such SDEs and establish its ergodicity. Furthermore, we demonstrate that the corresponding numerical invariant measure converges to the invariant measure of the underlying SDE in $L^q$-Wasserstein distance.

Our main contributions are summarized as follows:
\begin{itemize}
	\item An adaptive time-stepping EM scheme is developed for SDEs with non-globally Lipschitz drift and diffusion coefficients. It is also shown that the terminal time $T>0$ is almost surely attainable, namely,
	$\mathbb{P}(\exists M(\omega ) < \infty \; \text{s.t.} \; {t_{M(\omega )}} \ge T) = 1$,
	which ensures that the adaptive numerical iteration reaches $T$ in finitely many steps.

	\item Strong convergence of the proposed scheme is established both on finite time intervals and uniformly in time. Under a slightly strengthened assumption, the optimal $1/2$-order of strong convergence rate is further obtained.
	
	\item It is shown that the adaptive time-stepping EM scheme preserves the $p$th moment exponential stability of the exact solution.
	
	\item The polynomial ergodicity of the proposed scheme is established. Moreover, a $1/2$-order of convergence rate in the $L^q$-Wasserstein distance is derived between the invariant measure of the exact solution and that of the numerical scheme.

\end{itemize}

The primary goal of this work is to construct an explicit adaptive time-stepping EM scheme for SDEs with locally Lipschitz drift and diffusion coefficients, and to establish its ergodicity along with the convergence of its invariant measure to that of the underlying SDE.
The adaptive timestep is dynamically adjusted at each iteration to prevent numerical explosions.
Because no uniform positive minimum stepsize is imposed, a primary challenge is to ensure that any terminal time $T > 0$ is almost surely attainable. Adopting the auxiliary $K$-scheme introduced by Fang and Giles \cite{fang2020adaptive}, we show that the numerical iteration reaches $T$ in finitely many steps almost surely, and establish pathwise moment bounds on finite time intervals. Building upon a continuous-time numerical scheme, we establish a $1/2$-order of strong convergence rate both on finite intervals and uniformly in time.
Furthermore, by employing Taylor expansions and the optimal stopping theorem, we prove that the continuous-time numerical solution inherits the $p$th moment exponential stability of the SDE, filling a theoretical gap for adaptive time-stepping schemes.

To establish ergodicity, we first prove that the discrete-time numerical solution forms a time-homogeneous, irreducible, and aperiodic Markov chain. Without a uniform positive minimum stepsize, we introduce a polynomial lower bound on the adaptive timestep function (see Assumption \ref{A-step}). Harris's theorem then yields the existence and uniqueness of the numerical invariant measure, while a refined timestep condition guarantees polynomial ergodicity. Finally, combining the uniform-in-time convergence rate with the ergodicity of the exact solution, we obtain a $1/2$-order convergence rate in the $L^q$-Wasserstein distance between the exact and numerical invariant measures. Numerical experiments on stiff, nonstiff, and overdamped Langevin systems validate that the proposed scheme achieves superior accuracy and better overall performance than the competing methods.

The rest of the paper is organized as follows. Section \ref{sec2} presents some preliminaries and results of the exact solution. Section \ref{sec3} constructs the adaptive time-stepping EM scheme and proves the boundedness of the numerical solutions. 
Section \ref{sec4} establishes the strong convergence result on finite time intervals. 
Section \ref{sec5} investigates the $p$th moment boundedness of the adaptive time-stepping EM solutions on infinite time intervals. Section \ref{sec6} further establishes the uniform-in-time convergence rate. 
Section \ref{sec7} shows that the adaptive time-stepping EM scheme is capable of reproducing both the $p$th moment exponential stability and almost sure exponential stability. 
Section \ref{sec8} shows the convergence rate of the invariant measure of the adaptive time-stepping EM scheme.
Section \ref{sec9} presents several numerical experiments to support our theoretical results. 
Section \ref{sec10} concludes the paper with some further remarks.

\section{Preliminaries} \label{sec2}
Let $(\Omega ,\mathcal{F},\mathbb{P})$ be a complete probability space with a filtration ${\{ {\mathcal{F}_t}\} _{t \ge 0}}$ satisfying the usual conditions, that is, it is increasing and right continuous while $\mathcal{F}_0$ contains all $\mathbb{P}$-null sets. For any vector or matrix $A$, we denote its transpose by $A^{\rm T}$. Let $W(t) = (W_1(t), W_2(t),...,W_m(t))^{\rm T} $ be an $m$-dimensional Brownian motion. Let $\left|  \cdot  \right|$ denote both the Euclidean norm in $\mathbb{R}^d$ and the Frobenius norm in $\mathbb{R}^{d \times m}$. Let $\langle { \cdot,  \cdot } \rangle $ denote the inner product of vectors in $\mathbb{R}^d$.  
Let $\mathbb{R}_+ = (0, +\infty)$. Let $\mathbb{N}^+$ denote the set of positive integers.
Let $\mathcal{P}(\mathbb{R}^d)$ denote the set of all probability measures on $\mathbb{R}^d$. For $q \ge 1$, define
\begin{equation*}
	{\mathcal{P}_q}({\mathbb{R}^d}) = \left\{ {\mu  \in \mathcal{P}({\mathbb{R}^d}):\mu \left( {{{\left|  \cdot  \right|}^q}} \right) = \int_{{\mathbb{R}^d}} {{{\left| x \right|}^q}\mu (dx)}  < \infty } \right\} .
\end{equation*}
For $\mu ,\nu  \in {\mathcal{P}_q}({\mathbb{R}^d})$, the $L^q$-Wasserstein distance is defined by
\begin{equation*}
	{{\cal W}_q}(\mu ,\nu ) = \mathop {\inf }\limits_{\pi  \in {\cal C}(\mu ,\nu )} {\left( {\int_{{\mathbb{R}^d} \times {\mathbb{R}^d}} {{{\left| {x - y} \right|}^q}\pi (dx,dy)} } \right)^{1/q}} ,
\end{equation*}
where $\mathcal{C}(\mu ,\nu )$ is the set of all couplings of $\mu ,\nu$, i.e., $\pi  \in \mathcal{C}(\mu ,\nu )$ if and only if $\pi ( \cdot ,{\mathbb{R}^d}) = \mu ( \cdot )$ and $\pi ({\mathbb{R}^d}, \cdot) = \nu ( \cdot )$. Clearly, $({\mathcal{P}_q}({\mathbb{R}^d}),{\mathcal{W}_q})$ is a complete metric space for any $q \ge 1$ (see e.g., \cite[Theorem 5.4, p. 175]{chen2004from}).

Let $\mu$ and $\nu$ be two probability measures on $(\mathbb{R}^d, \mathcal{B}(\mathbb{R}^d))$. The total variation distance between $\mu$ and $\nu$ is defined by
\begin{equation}
	{\left\| {\mu  - \nu } \right\|_{\rm TV}} = \mathop {\max }\limits_{A \in \mathcal{B}(\mathbb{R}^d)} \left| {\mu (A) - \nu (A)} \right|.
\end{equation}
Equivalently, it admits the dual representation (see e.g., \cite[p. 49]{levin2017markov})
\begin{equation}  \label{aa2}
	{\left\| {\mu  - \nu } \right\|_{\rm TV}} = \frac{1}{2}\sup \left\{ {\int_{{\mathbb{R}^d}} {f(x)\mu (dx)}  - \int_{{\mathbb{R}^d}} {f(x)\nu (dx)} :f \in \mathcal{B}_b (\mathbb{R}^d), \; \mathop {\max }\limits_{x \in {\mathbb{R}^d}} \left| {f(x)} \right| \le 1} \right\}.
\end{equation}
Here, $\mathcal{B}_b (\mathbb{R}^d)$ denotes the space of real-valued bounded Borel measurable functions on $\mathbb{R}^d$.
Moreover, if $|\mu-\nu|$ denotes the total variation measure of the signed measure $\mu-\nu$, then
\begin{equation}  \label{aa1}
	\left| {\mu  - \nu } \right|({\mathbb{R}^d}) = 2{\left\| {\mu  - \nu } \right\|_{\rm TV}} .
\end{equation}

Throughout this paper, we use $C$ to denote a generic positive constant whose value may change from place to place. Furthermore, we use $C_a$ to emphasize the dependence of the generic constant the parameter $a$.

In this paper, we consider a $d$-dimensional It\^o stochastic differential equation 

\begin{equation}  \label{equation}
	\left\{  
	\begin{array}{l}
		dX(t) = f(X(t))dt + g(X(t))dW(t), \quad t \ge 0,  \\
		X(0) = x_0,
	\end{array}
	\right.
\end{equation}
where $W(t)$ is an $m$-dimensional Brownian motion, and the drift and diffusion coefficients are $f:{\mathbb{R}^d} \to {\mathbb{R}^d}$ and $g:{\mathbb{R}^d} \to {\mathbb{R}^{d \times m}}$. For convenience, we impose the following assumptions.

\begin{assumption}  \label{A2.1}
	 For each $R > 0$, there exists a positive constant $C_R$ such that
	\begin{equation} \label{llc}
		\left| {f(x) - f(y)} \right| + \left| {g(x) - g(y)} \right| \le {C_R}\left| {x - y} \right|,
	\end{equation}
	for all $x, y \in \mathbb{R}^d$ with $\left| x \right| \vee \left| y \right| \le R$. Furthermore, there exist constants $p \ge 2$ and $K_1, K_2 \ge 0$ such that 
	\begin{equation}  \label{A2.1-1}
		\left\langle {x,f(x)} \right\rangle  + \frac{{p - 1}}{2}{\left| {g(x)} \right|^2} \le {K_1}{\left| x \right|^2} + {K_2}
	\end{equation} 
	for any $x \in \mathbb{R}^d$. 
\end{assumption}

\begin{assumption}  \label{Ag}
	Assume that there exists constants $l > 2$ and $K_3 > 0$ such that
	\begin{equation*}
		{\left| {g(x)} \right|^2} \le K_3\left( {1 + {{\left| x \right|}^{l}}} \right) 
	\end{equation*}
	 for any $x \in \mathbb{R}^d$.
\end{assumption}

Under Assumption \ref{A2.1}, the SDE \eqref{equation} has a unique strong solution on any interval $[0, T]$ with $T < \infty$ (see e.g., \cite{mao2008stochastic}). Moreover, the following moment bounds apply over any interval $[0, T]$. 

\begin{lemma} (\cite[Lemma 4.2]{mao2015the})  \label{EXt}
Let Assumptions \ref{A2.1} and \ref{Ag} hold, and assume that $ p > l-2$. Set $\bar p = p +2-l$. Then
\begin{equation*}
	\mathbb{E}\left( {\mathop {\sup }\limits_{0 \le t \le T} {{\left| {X(t)} \right|}^{\bar p}}} \right) \le C_T, \quad \forall T > 0.
\end{equation*}
\end{lemma}

\section{Explicit adaptive time-stepping EM scheme} \label{sec3}

In this section, we construct an explicit adaptive time-stepping EM scheme for SDE \eqref{equation} and show the boundedness of the numerical solutions.
We first specify the critical assumption about the adaptive timestep.

\begin{assumption}  \label{A-stepsize}
	The adaptive timestep function $\delta :{\mathbb{R}^d} \to \mathbb{R}_+$ is continuous. Moreover, there exists a constant $\gamma \in (0,1/3]$ such that 
	\begin{equation}  \label{A-step1}
		\left| {f(x)} \right|{\left( {\delta (x)} \right)^\gamma } \le {1 + \left| x \right|} , \quad {\left| {g(x)} \right|^2}{\left( {\delta (x)} \right)^\gamma } \le  {\left( {1 + \left| x \right|} \right)^2} ,
	\end{equation}	
	for any $x \in \mathbb{R}^d$.
\end{assumption}

Note that Assumption \ref{A-stepsize} remains valid if $\delta (\cdot)$ is replaced by any continuous and strictly positive timestep function $\delta^\Delta (\cdot)$ satisfying $\delta^\Delta (x) \le \delta(x)$ for any $x \in \mathbb{R}^d$.

\begin{remark}
We provide several examples of adaptive timestep functions satisfying Assumption \ref{A-stepsize} with $\gamma \in (0, 1/3]$:
\begin{itemize}
	\item[(i)] $\delta (x) = {\left( {\frac{1}{{1 \vee \left| {f(x)} \right|}} \wedge \frac{1}{{1 \vee {{\left| {g(x)} \right|}^2}}}} \right)^{1/{\gamma }}}$;
	
	\item[(ii)] $\delta (x) = {\left( {\frac{{1 \vee \left| x \right|}}{{1 \vee \left| {f(x)} \right|}} \wedge \frac{{1 \vee {{\left| x \right|}^2}}}{{1 \vee {{\left| {g(x)} \right|}^2}}}} \right)^{1/{\gamma }}}$;
	
	\item[(iii)] If the drift and diffusion coefficients satisfy the polynomial growth condition, that is, there exist constants $l_1, l_2 \ge 1$ and  $C_1, C_2 > 0$ such that $\left| {f(x)} \right| \le C_1( {1 + {{\left| x \right|}^{{l_1}}}} )$ and $\left| {g(x)} \right| \le C_2( {1 + {{\left| x \right|}^{{l_2}}}} )$
	for any $x \in \mathbb{R}^d$, then one may choose 
	\begin{equation*}
		\delta (x) = {\left( {{C_1} \vee {C_2^2}} \right)^{ - 1/\gamma }}{\left( {1 + \left| x \right|} \right)^{\frac{{(1 - {l_1}) \wedge (2(1 - {l_2}))}}{\gamma }}}.
	\end{equation*}.
\end{itemize}
\end{remark}

In the following, we propose an adaptive time-stepping EM scheme to approximate the exact solution of SDE \eqref{equation}. Set $t_0 = 0$ and $Y_0 = x_0$. Define
\begin{equation}  \label{scheme}
	\left\{  
	\begin{array}{l}
	    \delta _n := \delta (Y_{t_n}), \quad t_{n+1} = t_n + \delta _n , \quad n = 0,1,...,\\
		{Y_{{t_{n + 1}}}} =  {{Y_{{t_n}}} + f({Y_{{t_n}}}){\delta _n} + g({Y_{{t_n}}})\Delta {W_n}}, \quad n = 0,1,...,
	\end{array}
	\right.
\end{equation}
where $\Delta {W_n} = W({t_{n + 1}}) - W({t_n})$. For any $t \ge 0$, we define $ \underline{t}= \max \{ {t_n}:{t_n} \le t\} $, ${n_t} = \max \{ n:{t_n} \le t\} $ as the nearest time point
before time $t$, and its index. Based on \eqref{scheme}, we define two continuous-time versions of the adaptive time-stepping EM solutions by
\begin{equation} \label{continuous}
	Y(t) = Y_{\underline{t}}, \quad   t \ge 0,
\end{equation}
and
\begin{equation}   \label{continuous1}
\begin{split}
	\hat{Y}(t) = {Y_{\underline{t}}} + f({Y_{\underline{t}}})(t - \underline{t}) + g({Y_{\underline{t}}})(W(t) - W(\underline{t})) ,  \quad   t \ge 0.
\end{split}
\end{equation}
Together with \eqref{scheme}-\eqref{continuous1} we have
\begin{equation}  \label{continuous2}
\begin{split}
	\hat{Y}(t) = {x_0} + \int_0^t {f( Y(s))ds}  + \int_0^t {g(Y(s))dW(s)},  \quad   t \ge 0.
\end{split}
\end{equation}
Clearly, one has $\hat Y(\underline{t}) = Y_{\underline{t}} = {Y}(t)$ for any $t \ge 0$. 

\begin{remark}  \label{remark-filtration}
Under Assumption \ref{A-stepsize}, the random times $t_n$, $n = 0,1,...,$ defined by \eqref{scheme} are $\{\mathcal{F}_t\}$-stopping times \cite[Remark 3]{wen2026strong}.
Moreover, we define the filtration associated with the stopping time $t_n$ by
\begin{equation*} 
	\mathcal{F}_{t_n} 
	= \left\{ A \in \mathcal{F} : A \cap \{\omega : t_n(\omega) \le t\} \in \mathcal{F}_t \ \text{for all} \ t \ge 0 \right\}.
\end{equation*}
Hence, $Y_{t_n}$ is $\mathcal{F}_{t_n}$-measurable, and the adaptive timestep $\delta_n=\delta(Y_{t_n})$ is also $\mathcal{F}_{t_n}$-measurable, and $t_{n+1}$ is $\mathcal{F}_{t_{n}}$-measurable.
\end{remark}

\begin{remark}   \label{remark-steps}
	Notice that each Brownian increment $W({t_{n + 1}}) - W({t_n})$ depends on the adaptive timestep $\delta_n = {t_{n + 1}} - {t_n}$, and from \eqref{scheme} we see ${\delta _n} = {\delta }({Y_{t_n}})$ is determined by $Y_{t_n}$. Hence, in general, $W({t_{n + 1}}) - W({t_n})$ is not independent of $\mathcal{F}_{t_n}$. Nevertheless, $W(t_{n+1}) - W(t_n)$ is $\mathcal{F}_{t_n}$-conditionally normally distributed (see \cite[Remark 2.2]{kelly2018adaptive}). The Brownian increment $W(t_{n+1}) -W(t_n)$ has the same distribution as the random variable $\sqrt{\delta_n} \chi$ \cite{botija2024explicit}, where $\chi  \sim N(0,{I_{m}})$ is independent of $\mathcal{F}_{t_n}$, and $I_{m}$ is the identity matrix. Consequently, for every $r > 0$, 
	\begin{equation*}
		\mathbb{E}\left( {{{\left| {\Delta {W_n}} \right|}^r}|{\mathcal{F}_{{t_n}}}} \right) \le {C_{r,m}}\delta _n^{r/2} \quad \text{a.s.}
	\end{equation*}
	Let $A_n \in \mathbb{R}^{d}$ and $B_n \in \mathbb{R}^{d \times m}$ be $\mathcal{F}_{t_n}$-measurable random variables. For every integer $i \ge 1$, we have
	\begin{equation} \label{kl7}
		\begin{split}
			&\mathbb{E}\left[ {{{\left( {\left\langle {A_n,B_n \Delta {W_n}} \right\rangle } \right)}^{2i - 1}}|{{\cal F}_{{t_n}}}} \right] = 0 \quad \text{a.s.}, \\
			&\mathbb{E}\left[ {{{\left( {\left\langle {A_n,B_n \Delta {W_n}} \right\rangle } \right)}^{2i}}|{{\cal F}_{{t_n}}}} \right] \le {\left| A_n \right|^{2i}}{\left| B_n \right|^{2i}}\mathbb{E}\left[ {{{\left| {\Delta {W_n}} \right|}^{2i}}|{{\cal F}_{{t_n}}}} \right]  \le {\left| A_n \right|^{2i}}{\left| B_n \right|^{2i}}(2i - 1)!!{m^i}\delta _n^i \quad \text{a.s.}
		\end{split}
	\end{equation}	
\end{remark}

We then establish the moment boundedness of the adaptive time-stepping EM solutions \eqref{scheme}.
One of our aims is to show that the terminal time $T > 0$ is almost surely attainable, i.e., $\mathbb{P}(\exists M(\omega ) \in \mathbb{N}^+ \; \text{s.t.} \; {t_{M(\omega )}} \ge T) = 1$, which ensures that the numerical iteration reaches $T$ in finitely many steps. Without loss of generality, we always assume that $\delta(x) \le T$ for any $x \in \mathbb{R}^d$.

\begin{theorem}   \label{th2.1}
	Let Assumptions \ref{A2.1}, \ref{Ag} and \ref{A-stepsize} hold and assume that $p > l$. Then $T$ is almost surely attainable, and the adaptive time-stepping EM scheme \eqref{scheme} has the property
	\begin{equation}   \label{a7}
		\mathbb{E}\left( {\mathop {\sup }\limits_{0 \le t \le T} {{\left| {\hat Y(t)} \right|}^{\bar p}}} \right) \le C_T, \quad \forall T>0,
	\end{equation}
where $C_T$ is a positive constant depending on $x_0, p, l, T$.	
\end{theorem}

\begin{proof}
Since the proof is rather technical, we divide it into four steps.

\textbf{Step 1:} Definition of the auxiliary $K$-scheme.
For any $K > |x_0|$, we define the truncation mapping ${\pi _K}:{\mathbb{R}^d} \to {\mathbb{R}^d}$ by
\begin{equation}  \label{pi}
	{\pi _K}(x) = \left( {\left| x \right| \wedge K} \right)\frac{x}{{\left| x \right|}},
\end{equation}
we use the convention $x/|x|=\textbf{0}$ when $x = \mathbf{0}$.
Set $t_0= 0$ and ${Y_0^K = Z_0^K = {x_0}}$. We then define an auxiliary $K$-scheme by
\begin{equation}   \label{K-scheme}
	\left\{  
	\begin{array}{l}
		\delta _n^{K}: = \delta (Y_{{t_n}}^K), \quad t_{n+1} = t_n + \delta _n^{K }, \quad n = 0,1,...,  \\
		{Z_{{t_{n + 1}}}^K = Y_{{t_n}}^K + f(Y_{{t_n}}^K)\delta _n^{K } + g(Y_{{t_n}}^K)\Delta W_n}, \quad n = 0,1,...,  \\
		{Y_{{t_{n + 1}}}^K = {\pi _K}(Z_{{t_{n + 1}}}^K)}, \quad n = 0,1,...,  \\  
	\end{array}
	\right.
\end{equation}
where $\Delta W_n = W({t_{n + 1}}) - W({t_n})$. 
For notational simplicity, in Steps 1--3 we shall write $t_n$ instead of $t_n^K$ whenever no confusion can arise.
By the definition of $\pi_K(\cdot)$, we have
\begin{equation}  \label{a33}
	\begin{split}
		\left| {Y_{{t_n}}^K} \right| \le \left| {Z_{{t_n}}^K} \right|, \quad \left| {Y_{{t_n}}^K} \right| \le K, \quad a.s., \quad n =1,2,...
	\end{split}
\end{equation}
Moreover, we define two step processes  
\begin{equation}  \label{a22}
	{Y^K}(t) = Y_{\underline{t}}^K, \quad {Z^K}(t) = Z_{\underline{t}}^K, \quad \forall t \ge 0, 
\end{equation}
and define
\begin{equation}  \label{a11}
	{{\hat{Z}}^K}(t) = Y_{\underline{t}}^K + f(Y_{\underline{t}}^K)(t - \underline{t}) + g(Y_{\underline{t}}^K)(W(t) - W(\underline{t})), \quad  t \ge 0.
\end{equation}
Clearly, we have ${{\hat{Z}}^K}(t_n) = Y_{t_n}^K = {Y^K}(t_n)$ for all $n\ge 0$. Furthermore, from \eqref{K-scheme} and \eqref{a11}, for every $n \ge 1$ and almost every $\omega \in \Omega$, we have
\begin{equation} \label{a34}
	\mathop {\lim }\limits_{t \uparrow {t_n}(\omega)} {{\hat Z}^K}(t, \omega) = Z_{{t_n}}^K (\omega),\quad \mathop {\lim }\limits_{t \downarrow {t_n} (\omega)} {{\hat Z}^K}(t, \omega) = Y_{{t_n}}^K (\omega),\quad \text{a.s.}
\end{equation}
 
Since $\delta ( \cdot )$ is continuous and strictly positive, and $\{ {x \in {\mathbb{R}^d}:\left| x \right| \le K} \}$ is a bounded and closed set, we get $\delta _{\min }^{K}: = {\min _{|x| \le K}}{\delta  }(x) > 0$. Then, for any $\omega \in \Omega$,
\begin{equation*}
	 \delta _n^{K }(\omega ) \ge \delta _{\min }^{K } > 0, \quad  n \ge 0,
\end{equation*}
which implies
\begin{equation*}
	\mathop {\lim }\limits_{n \to \infty } {t_n}(\omega ) = \sum\limits_{n = 0}^\infty  {{\delta _n^{K }}(\omega )}  \ge \sum\limits_{n = 0}^\infty  {\delta _{\min }^{K}}  = \infty .
\end{equation*}
Hence, $T$ is attainable. 

\textbf{Step 2:} Estimation of ${\sup _{K > |{x_0}|}}{\sup _{0 \le t \le T}}\mathbb{E}{| {{{\hat Z}^K}(t)} |^p} \le {C_T}$.
For any integer $n \ge 0$, from \eqref{K-scheme} one obtains
	\begin{equation}  \label{a1}
		\begin{split}
			{\left| {Z_{{t_{n + 1}}}^K} \right|^2} =& {\left| {Y_{{t_n}}^K + f(Y_{{t_n}}^K){\delta _n^{K }} + g(Y_{{t_n}}^K)\Delta {W_n}} \right|^2} \\
			=& {\left| {{Y_{{t_n}}^K}} \right|^2} + 2\left\langle {{Y_{{t_n}}^K},f({Y_{{t_n}}^K})} \right\rangle \delta _n^{K }  + {\left| {f({Y_{{t_n}}^K})} \right|^2}{\left| {\delta _n^{K } } \right|^2} + {\left| {g(Y_{{t_n}}^K)\Delta {W_n}} \right|^2} \\
			&+ 2\left\langle {{Y_{{t_n}}^K} + f({Y_{{t_n}}^K})\delta _n^{K } ,g({Y_{{t_n}}^K})\Delta {W_n}} \right\rangle .
		\end{split}
	\end{equation}
	We therefore have
	\begin{equation}  \label{a6}
		\begin{split}
			{\left( {1 + {{\left| {Z_{{t_{n + 1}}}^K} \right|}^2}} \right)^{p/2}} = {\left( {1 + {{\left| {Y_{{t_n}}^K} \right|}^2}} \right)^{p/2}}{\left( {1 + \xi _n^{K }} \right)^{p/2}},
		\end{split}
	\end{equation}
	where
	\begin{equation}  \label{a26}
		\xi _n^{K } = \frac{{2\left\langle {Y_{{t_n}}^K,f(Y_{{t_n}}^K)} \right\rangle \delta _n^{K } + {{\left| {f(Y_{{t_n}}^K)} \right|}^2}{{\left( {\delta _n^{K }} \right)}^2} + {\left| {g(Y_{{t_n}}^K)\Delta {W_n}} \right|^2} + 2\left\langle {Y_{{t_n}}^K + f(Y_{{t_n}}^K)\delta _n^{K },g(Y_{{t_n}}^K)\Delta {W_n}} \right\rangle }}{{1 + {{\left| {Y_{{t_n}}^K} \right|}^2}}}.
	\end{equation}
	For the given constant $p \ge 2$, choose a non-negative integer $k$ such that $2k < p \le 2(k+1)$. It follows from \cite[Lemma 3.3]{yang2018explicit} that
	\begin{equation}  \label{sum1}
	\begin{split}
		&\mathbb{E}\left[ {{{\left( {1 + {{\left| {{Z_{{t_{n + 1}}}^K}} \right|}^2}} \right)}^{p/2}}|{{\cal F}_{{t_n}}}} \right] \\
		&\le {\left( {1 + {{\left| {{Y_{{t_n}}^K}} \right|}^2}} \right)^{p/2}}\left[ {1 + p/2\mathbb{E}(\xi _n^{K }|{{\cal F}_{{t_n}}}) + \frac{{p(p - 2)}}{8}\mathbb{E}({{(\xi _n^{K })}^2}|{{\cal F}_{{t_n}}}) + \mathbb{E}({{(\xi _n^{K })}^3}{P_k}(\xi _n^{K })|{{\cal F}_{{t_n}}})} \right], \quad \text{a.s.},
	\end{split}
	\end{equation}
	where $P_k(\cdot)$ represents a $k$th-order polynomial whose coefficients
	depend only on $p$. Using Assumption \ref{A-stepsize}, Remarks \ref{remark-filtration} and \ref{remark-steps} yields	
	\begin{equation*}
		\begin{split}  \label{xi1}
			&\mathbb{E}\left( {\xi _n^{K }|{{\cal F}_{{t_n}}}} \right) \\
			&= {\left( {1 + {{\left| {Y_{{t_n}}^K} \right|}^2}} \right)^{ - 1}}\mathbb{E}\Big[ {2\left\langle {Y_{{t_n}}^K,f(Y_{{t_n}}^K)} \right\rangle \delta _n^{K } + {{\left| {f(Y_{{t_n}}^K)} \right|}^2}{{\left( {\delta _n^{K }} \right)}^2} + {{\left| {g(Y_{{t_n}}^K)\Delta {W_n}} \right|}^2}}  \\
			&\quad + 2\left\langle {Y_{{t_n}}^K + f(Y_{{t_n}}^K){\delta _n^{K }},g(Y_{{t_n}}^K)\Delta {W_n}} \right\rangle |{\mathcal{F}_{{t_n}}} \Big] \\
			&= {\left( {1 + {{\left| {Y_{{t_n}}^K} \right|}^2}} \right)^{ - 1}}\left[ {2\left\langle {Y_{{t_n}}^K,f(Y_{{t_n}}^K)} \right\rangle  + {{\left| {g(Y_{{t_n}}^K)} \right|}^2}} \right]{\delta _n^{K }} + {\left( {1 + {{\left| {Y_{{t_n}}^K} \right|}^2}} \right)^{ - 1}}{\left| {f(Y_{{t_n}}^K)} \right|^2}\left( {\delta _n^{K }} \right)^2 \\
			&\le {\left( {1 + {{\left| {Y_{{t_n}}^K} \right|}^2}} \right)^{ - 1}}\left[ {2\left\langle {Y_{{t_n}}^K,f(Y_{{t_n}}^K)} \right\rangle  + {{\left| {g(Y_{{t_n}}^K)} \right|}^2}} \right]{\delta _n^{K }} + 2{\left( {\delta _n^{K }} \right)^{2(1 - \gamma )}} \quad \text{a.s.}
		\end{split}
	\end{equation*}
Similarly, using Assumption \ref{A-stepsize}, Remarks \ref{remark-filtration} and \ref{remark-steps}
we compute 
\begin{equation*}
\begin{split}
	&\mathbb{E}\left( {{{\left( {\xi _n^{K }} \right)}^2}|{\mathcal{F}_{{t_n}}}} \right) \le 4{\left( {1 + {{\left| {Y_{{t_n}}^K} \right|}^2}} \right)^{ - 2}}{\left| {{{\left( {Y_{{t_n}}^K} \right)}^{\rm T}}g(Y_{{t_n}}^K)} \right|^2}\delta _n^K + C{\left( {\delta _n^{K }} \right)^{2(1 - \gamma )}}  \quad \text{a.s.} \\
	&\mathbb{E}\left( {{{\left( {\xi _n^K} \right)}^3}{P_k}\left( {\xi _n^K} \right)|{\mathcal{F}_{{t_n}}}} \right) \le {C}{\left( {\delta _n^K} \right)^{2(1 - \gamma )}} \quad  \text{a.s.}
\end{split}
\end{equation*}
Substituting the above into \eqref{sum1}, using Assumption \ref{A2.1}, \eqref{a33}, the facts $\delta_n^K\le T$ and $\gamma \in (0, 1/3]$ yields
\begin{equation}  \label{a4} 
	\begin{split}
		&\mathbb{E}\left[ {{{\left( {1 + {{\left| {Z_{{t_{n + 1}}}^K} \right|}^2}} \right)}^{p/2}}|{\mathcal{F}_{{t_n}}}} \right] \\
		&\le {\left( {1 + {{\left| {Y_{{t_n}}^K} \right|}^2}} \right)^{p/2}}\Bigg[ {1 + C{{\left( {\delta _n^{K}} \right)}^{2(1 - \gamma )}}}  \\
		&\quad + p\delta _n^K\frac{{\left( {1 + {{\left| {Y_{{t_n}}^K} \right|}^2}} \right)\left( {2\left\langle {Y_{{t_n}}^K,f(Y_{{t_n}}^K)} \right\rangle  + {{\left| {g(Y_{{t_n}}^K)} \right|}^2}} \right) + (p - 2){{\left| {{{\left( {Y_{{t_n}}^K} \right)}^{\rm T}}g(Y_{{t_n}}^K)} \right|}^2}}}{{2{{\left( {1 + {{\left| {Y_{{t_n}}^K} \right|}^2}} \right)}^2}}} \Bigg] \\
		&\le \left( {1 + C{{\left( {\delta _n^{K}} \right)}^{2(1-\gamma)}}} \right){\left( {1 + {{\left| {Y_{{t_n}}^K} \right|}^2}} \right)^{p/2}} \\
		& \quad + \frac{{p{\delta _n^{K}}}}{2}{\left( {1 + {{\left| {Y_{{t_n}}^K} \right|}^2}} \right)^{p/2 - 1}}\left( {2\left\langle {Y_{{t_n}}^K,f(Y_{{t_n}}^K)} \right\rangle  + (p - 1){{\left| {g(Y_{{t_n}}^K)} \right|}^2}} \right) \\
		&\le (1 + C_T{\delta _n^{K}}){\left( {1 + {{\left| {Y_{{t_n}}^K} \right|}^2}} \right)^{p/2}} \\
		&\le (1 + C_T{\delta _n^{K}}){\left( {1 + {{\left| {Z_{{t_n}}^K} \right|}^2}} \right)^{p/2}} \quad \text{a.s.}
	\end{split}
\end{equation}
Using \eqref{a4} and \eqref{a22} we see 
\begin{equation}  \label{a28}
\begin{split}
	\mathbb{E}\left[ {{{\left( {1 + {{\left| {Z_{{t_{n + 1}}}^K} \right|}^2}} \right)}^{p/2}}|{{\cal F}_{{t_n}}}} \right] - {\left( {1 + {{\left| {Z_{{t_n}}^K} \right|}^2}} \right)^{p/2}} \le {C_T}\int_{{t_n}}^{{t_{n + 1}}} {{{\left( {1 + {{\left| {{Z^K}(s)} \right|}^2}} \right)}^{p/2}}ds}   \quad \text{a.s.}
\end{split}
\end{equation}
For any $t \in [0, T]$,
summing \eqref{a28} over multiple timesteps, taking expectations on both sides and using \eqref{a22} yields
\begin{equation*}
\begin{split}
	&\mathbb{E}\left[ {{{\left( {1 + {{\left| {Z^K (t)} \right|}^2}} \right)}^{p/2}}} \right] -{\left( {1 + {{\left| {{x_0}} \right|}^2}} \right)^{p/2}}  \\
	&\le  {C_T}\mathbb{E}\left( {\int_0^{\underline{t}} {{{\left( {1 + {{\left| {{Z^K}(s)} \right|}^2}} \right)}^{p/2}}ds} } \right) 
	\le  {C_T}\int_0^t {\mathbb{E}{{\left( {1 + {{\left| {{{Z}^K}(s)} \right|}^2}} \right)}^{p/2}}ds} \\
	&\le  {C_T}\int_0^t {\mathop {\sup }\limits_{0 \le u \le s} \mathbb{E}{{\left( {1 + {{\left| {{{Z}^K}(u)} \right|}^2}} \right)}^{p/2}}ds} .
\end{split}
\end{equation*} 
Taking the supremum on both sides and using Gronwall's inequality, we have
\begin{equation*}
	\mathop {\sup }\limits_{0 \le t \le T} \mathbb{E}\left[ {{{\left( {1 + {{\left| {{{ Z}^K}(t)} \right|}^2}} \right)}^{p/2}}} \right] \le {C_T} .
\end{equation*}
Therefore,
\begin{equation}  \label{a29}
	\mathop {\sup }\limits_{0 \le t \le T} \mathbb{E}{\left| {{{Z}^K}(t)} \right|^p} \le \mathop {\sup }\limits_{0 \le t \le T} \mathbb{E}\left[ {{{\left( {1 + {{\left| {{{Z}^K}(t)} \right|}^2}} \right)}^{p/2}}} \right] \le {C_T},
\end{equation}
where $C_T$ is a positive constant depending on $x_0, p, T$, but independent of $K$.

On the other hand, for the partial timestep from $\underline{t}$ to $t$, from \eqref{a11} we have
\begin{equation*}
\begin{split}
	{\left| {{\hat{Z}^K}(t)} \right|^2} =& {\left| {Y_{\underline{t}}^K} \right|^2} + 2\left\langle {Y_{\underline{t}}^K,f(Y_{\underline{t}}^K)} \right\rangle (t - \underline{t}) + {\left| {f(Y_{\underline{t}}^K)} \right|^2}{(t - \underline{t})^2} + {\left| {{{\left( {g(Y_{\underline{t}}^K)} \right)}}\left( {W(t) - W(\underline{t})} \right)} \right|^2} \\
	&+ 2\left\langle {Y_{\underline{t}}^K + f(Y_{\underline{t}}^K)(t - \underline{t}),g(Y_{\underline{t}}^K)(W(t) - W(\underline{t}))} \right\rangle .
\end{split}
\end{equation*}
Similar to \eqref{a6}-\eqref{a4} one has
\begin{equation}  \label{a3}
\begin{split}
	\mathbb{E}\left[ {{{\left( {1 + {{\left| {{\hat Z^K}(t)} \right|}^2}} \right)}^{p/2}}|{{\cal F}_{\underline{t}}}} \right] 
	&\le \left( {1 + C_T(t - \underline{t})} \right){\left( {1 + {{\left| {Y_{\underline{t}}^K} \right|}^2}} \right)^{p/2}} \\
	 &\le \left( {1 + C_T(t - \underline{t})} \right){\left( {1 + {{\left| {Z_{\underline{t}}^K} \right|}^2}} \right)^{p/2}} \\
	&\le {\left( {1 + {{\left| {Z_{\underline{t}}^K} \right|}^2}} \right)^{p/2}} + C_T\int_{\underline{t}}^t {{{\left( {1 + {{\left| {{{Z}^K}(s)} \right|}^2}} \right)}^{p/2}}ds}  \quad \text{a.s.}
\end{split}
\end{equation}
Summing \eqref{a28} over multiple timesteps, adding \eqref{a3} and then taking expectations on both sides gives
\begin{equation*}
\begin{split}
	\mathbb{E}\left[ {{{\left( {1 + {{\left| {{\hat Z^K}(t)} \right|}^2}} \right)}^{p/2}}} \right] \le {\left( {1 + {{\left| {{x_0}} \right|}^2}} \right)^{p/2}} + C_T\mathbb{E}\left[ {\int_0^t {{{\left( {1 + {{\left| {{{Z}^K}(s)} \right|}^2}} \right)}^{p/2}}ds} } \right].
\end{split}
\end{equation*}
Taking the supremum on both sides of the above inequality and using \eqref{a29} yields
\begin{equation*}  \label{a9}
\begin{split}
	\mathop {\sup }\limits_{0 \le t \le T} \mathbb{E}\left[ {{{\left( {1 + {{\left| {{\hat Z^K}(t)} \right|}^2}} \right)}^{p/2}}} \right] 
	\le& {\left( {1 + {{\left| {{x_0}} \right|}^2}} \right)^{p/2}} + C_T\int_0^T \mathbb{E}{{{\left( {1 + {{\left| {Z^K(t)} \right|}^2}} \right)}^{p/2}}dt} \\
	\le& {\left( {1 + {{\left| {{x_0}} \right|}^2}} \right)^{p/2}} + C_T\int_0^T {\mathop {\sup }\limits_{0 \le s \le t} \mathbb{E}{{\left( {1 + {{\left| {{Z}^K(s)} \right|}^2}} \right)}^{p/2}}dt} \\
	\le& C_T .
\end{split}
\end{equation*}
Therefore, 
\begin{equation}  \label{a23}
	\mathop {\sup }\limits_{K > |{x_0}|} \mathop {\sup }\limits_{0 \le t \le T} \mathbb{E}{\left| {{{\hat Z}^K}(t)} \right|^p} \le \mathop {\sup }\limits_{K > |{x_0}|} \mathop {\sup }\limits_{0 \le t \le T} \mathbb{E}\left[ {{{\left( {1 + {{\left| {{{\hat Z}^K}(t)} \right|}^2}} \right)}^{p/2}}} \right] \le {C_T} ,
\end{equation}
where $C_T$ is a positive constant depending on $x_0, p, T$, but independent of $K$.

\textbf{Step 3:} The estimation of ${\sup _{K > |{x_0}|}}\mathbb{E}({\sup _{0 \le t \le T}}|{{\hat Z}^K}(t){|^{\bar p}}) \le {C_T}$.
For any $t \in [0, T]$, from \eqref{a22} and \eqref{a11} we have
\begin{equation*}
	{{\hat{Z}}^K}(t) = Y_{\underline{t}}^K + \int_{\underline{t}}^t {f({{Y}^K}(s))ds}  + \int_{\underline{t}}^t {g({{Y}^K}(s))dW(s)} .
\end{equation*}
Using It\^o's formula, applying Assumption \ref{A2.1} and \eqref{a33} one has
\begin{equation}  \label{a31}
	\begin{split}
		{\left| {{{\hat{Z}}^K}(t)} \right|^{\bar p}} \le& {\left| {Y_{\underline{t}}^K} \right|^{\bar p}} + \int_{\underline{t}}^t {{\bar p}{{\left| {{{\hat{Z}}^K}(s)} \right|}^{{\bar p} - 2}}\left[ {\left\langle {{{\hat{Z}}^K}(s),f({{Y}^K}(s))} \right\rangle  + \frac{{{\bar p} - 1}}{2}{{\left| {g({{Y}^K}(s))} \right|}^2}} \right]ds} \\
		&+ \int_{\underline{t}}^t {\bar p{{\left| {{{\hat Z}^K}(s)} \right|}^{\bar p - 2}}{{\left( {{{\hat Z}^K}(s)} \right)}^{\rm T}}g({Y^K}(s))dW(s)}  \\
		\le& {\left| {Z_{\underline{t}}^K} \right|^{\bar p}} + \int_{\underline{t}}^t {{\bar p}{{\left| {{{\hat{Z}}^K}(s)} \right|}^{{\bar p} - 2}}\left[ {\left\langle {{{Y}^K}(s),f({{Y}^K}(s))} \right\rangle  + \frac{{{\bar p} - 1}}{2}{{\left| {g({{Y}^K}(s))} \right|}^2}} \right]ds} \\
		&+ \int_{\underline{t}}^t {{\bar p}{{\left| {{{\hat{Z}}^K}(s)} \right|}^{{\bar p} - 2}}\left\langle {{{\hat{Z}}^K}(s) - {{Y}^K}(s),f({{Y}^K}(s))} \right\rangle ds} \\
		&+ \int_{\underline{t}}^t {\bar p{{\left| {{{\hat Z}^K}(s)} \right|}^{\bar p - 2}}{{\left( {{{\hat Z}^K}(s)} \right)}^{\rm T}}g({Y^K}(s))dW(s)} \\
		\le& {\left| {Z_{\underline{t}}^K} \right|^{\bar p}} + \int_{\underline{t}}^t {\bar p}{{{\left| {{{\hat{Z}}^K}(s)} \right|}^{{\bar p} - 2}}\left( {K_1 {{\left| {{{Y}^K}(s)} \right|}^2} + K_2 } \right)ds}  \\
		&+ \int_{\underline{t}}^t {{\bar p}{{\left| {{{\hat{Z}}^K}(s)} \right|}^{{\bar p} - 2}}\left\langle {{{\hat{Z}}^K}(s) - {{Y}^K}(s),f({{Y}^K}(s))} \right\rangle ds} \\
		&+ \int_{\underline{t}}^t {\bar p{{\left| {{{\hat Z}^K}(s)} \right|}^{\bar p - 2}}{{\left( {{{\hat Z}^K}(s)} \right)}^{\rm T}}g({Y^K}(s))dW(s)}.
	\end{split}
\end{equation}
For each fixed $n \ge 0$ and almost every $\omega \in \Omega$, taking the limit ${t \uparrow {t_{n+1} (\omega)}}$ on both sides of \eqref{a31} and using \eqref{a34} yields
\begin{equation}  \label{a35}
	\begin{split}
		{\left| {Z_{{t_{n + 1}}}^K} \right|^{\bar p}} \le& {\left| {Z_{{t_{n}}}^K} \right|^{\bar p}} + \int_{{t_{n}}}^{{t_{n+1}}} {\bar p} {\left| {{{\hat Z}^K}(s)} \right|^{\bar p - 2}}\left( {{K_1}{{\left| {{Y^K}(s)} \right|}^2} + {K_2}} \right)ds \\
		&+ \int_{{t_{n}}}^{{t_{n+1}}} {\bar p{{\left| {{{\hat Z}^K}(s)} \right|}^{\bar p - 2}}\left\langle {{{\hat Z}^K}(s) - {Y^K}(s),f({Y^K}(s))} \right\rangle ds} \\
		&+ \int_{{t_{n}}}^{{t_{n+1}}} {\bar p{{\left| {{{\hat Z}^K}(s)} \right|}^{\bar p - 2}}{{\left( {{{\hat Z}^K}(s)} \right)}^{\rm T}}g({Y^K}(s))dW(s)} .
	\end{split}
\end{equation}
Summing \eqref{a35} over multiple timesteps and adding \eqref{a31} we obtain
\begin{equation*}
\begin{split}
	{\left| {{{\hat Z}^K}(t)} \right|^{\bar p}} \le& {\left| {{x_0}} \right|^{\bar p}} + \int_0^t {\bar p} {\left| {{{\hat Z}^K}(s)} \right|^{\bar p - 2}}\left( {{K_1}{{\left| {{Y^K}(s)} \right|}^2} + {K_2}} \right)ds \\
	&+ \int_0^t {\bar p{{\left| {{{\hat Z}^K}(s)} \right|}^{\bar p - 2}}\left\langle {{{\hat Z}^K}(s) - {Y^K}(s),f({Y^K}(s))} \right\rangle ds}  \\
	&+ \int_{{0}}^t {\bar p{{\left| {{{\hat Z}^K}(s)} \right|}^{\bar p - 2}}{{\left( {{{\hat Z}^K}(s)} \right)}^{\rm T}}g({Y^K}(s))dW(s)}. 
\end{split}
\end{equation*}
Taking the supremum on both sides and then taking expectations we have
\begin{equation}  \label{Jsum}
	\begin{split}
		\mathbb{E}\left( {\mathop {\sup }\limits_{0 \le t \le T} {{\left| {{{\hat{Z}}^K}(t)} \right|}^{\bar p}}} \right) \le {\left| {{x_0}} \right|^{\bar p}} + {J_1} + {J_2} + {J_3},
	\end{split}
\end{equation}
where 
\begin{equation*}
	\begin{split}
		{J_1} =& \bar p\mathbb{E}\left[ {\mathop {\sup }\limits_{0 \le t \le T} \int_0^{{t}} {{{\left| {{{\hat Z}^K}(s)} \right|}^{\bar p - 2}}\left( {{K_1}{{\left| {{Y^K}(s)} \right|}^2} + {K_2}} \right)ds} } \right] \\
		{J_2} =& \bar p\mathbb{E}\left[ {\mathop {\sup }\limits_{0 \le t \le T} \int_0^{{t}} {{{\left| {{{\hat Z}^K}(s)} \right|}^{\bar p - 2}}\left\langle {{{\hat Z}^K}(s) - {Y^K}(s),f({Y^K}(s))} \right\rangle ds} } \right] \\
		{J_3} =& \mathbb{E}\left[ {\mathop {\sup }\limits_{0 \le t \le T} \left| {\int_0^t {\bar p{{\left| {{{\hat Z}^K}(s)} \right|}^{\bar p - 2}}{{\left( {{{\hat Z}^K}(s)} \right)}^{\rm T}}g({Y^K}(s))dW(s)} } \right|} \right] .
	\end{split}
\end{equation*}
For $J_1$, using Young's inequality, \eqref{a33}, \eqref{a22}, \eqref{a29} and \eqref{a23} yields
\begin{equation}  \label{J1}
\begin{split} 
	J_1
	\le& \bar p\mathbb{E}\left[ {\int_0^T {{{\left| {{{\hat Z}^K}(s)} \right|}^{\bar p - 2}}\left( {{K_1}{{\left| {{Y^K}(s)} \right|}^2} + {K_2}} \right)ds} } \right] \\
	\le& (\bar p - 2)({K_1} + {K_2})\int_0^T {\mathbb{E}{{\left| {{{\hat Z}^K}(s)} \right|}^{\bar p}}ds}  + 2{K_1}\int_0^T {\mathbb{E}{{\left| {{Y^K}(s)} \right|}^{\bar p}}ds}  + 2{K_2}T \\
	\le& C\int_0^T {{{\left( {\mathop {\sup }\limits_{0 \le u \le s} \mathbb{E}{{\left| {{{\hat Z}^K}(u)} \right|}^p}} \right)}^{\frac{{\bar p}}{p}}}ds} + C_T \\
	\le& {C_T}.
\end{split}
\end{equation}
For $J_2$, for any $0 \le s \le T$, using \eqref{a22} and the fact $|s - \underline{s}| \le \delta_{n_s}$, Assumption \ref{A-stepsize} and Remark \ref{remark-steps} yields
\begin{equation}  \label{a16}
\begin{split}
	&\mathbb{E}\left( {{{\left| {{{\hat Z}^K}(s) - {Y^K}(s)} \right|}^{\bar p}}|{{\cal F}_{{\underline{s}}}}} \right)  \\
	&\le {2^{{\bar p} - 1}}{\left| {f(Y_{{\underline{s}}}^K)} \right|^{\bar p}}\mathbb{E}\left( {{{\left| {s - {\underline{s}}} \right|}^{\bar p}}|{\mathcal{F}_{{\underline{s}}}}} \right) + {2^{{\bar p} - 1}}{\left| {g(Y_{{\underline{s}}}^K)} \right|^{\bar p}}\mathbb{E}\left( {{{\left| {W(s) - W({\underline{s}})} \right|}^{\bar p}}|{\mathcal{F}_{{\underline{s}}}}} \right) \\
	&\le C{\left( {\delta _{{n_s}}^{K}} \right)^{\bar p(1 - \gamma )}}{\left( {1 + \left| {Y_{{\underline{s}}}^K} \right|} \right)^{\bar p}} + C{\left( {\delta _{{n_s}}^{K}} \right)^{\frac{{\bar p}}{2}(1 - \gamma )}}{\left( {1 + \left| {Y_{{\underline{s}}}^K} \right|} \right)^{\bar p}} \\
	&\le C_T{\left( {\delta _{{n_s}}^{K}} \right)^{\frac{{\bar p}}{2}(1 - \gamma )}}{\left( {1 + \left| {Y_{{\underline{s}}}^K} \right|} \right)^{\bar p}} \quad \text{a.s.}
\end{split}
\end{equation}
The last inequality holds because ${(\delta _{{n_s}}^K)^{\bar p(1 - \gamma )}} \le C_T{(\delta _{{n_s}}^K)^{\bar p(1 - \gamma )/2}}$ follows from the fact $\delta(x) \le T $ for any $x \in \mathbb{R}^d$.
Moreover, using Assumption \ref{A-stepsize} we have
\begin{equation}   \label{a17}
	\mathbb{E}\left( {{{\left| {f({{Y}^K}(s))} \right|}^{\bar p}}|{{\cal F}_{{\underline{s}}}}} \right) \le {\left( {\delta _{{n_s}}^{K}} \right)^{ - \bar p\gamma }}{\left( {1 + \left| {Y_{{\underline{s}}}^K} \right|} \right)^{\bar p}} \quad \text{a.s.}
\end{equation}
Since $\gamma \in (0, 1/3]$, applying H\"older's inequality for conditional expectations, from \eqref{a16} and \eqref{a17} we obtain 
\begin{equation}  \label{a27}
\begin{split}
	&\mathbb{E}\left( {{{\left| {{{\hat Z}^K}(s) - {Y^K}(s)} \right|}^{\frac{{\bar p}}{2}}}{{\left| {f({Y^K}(s))} \right|}^{\frac{{\bar p}}{2}}}|{\mathcal{F}_{{\underline{s}}}}} \right) \\
	&\le {\left[ {\mathbb{E}\left( {{{\left| {{{\hat Z}^K}(s) - {Y^K}(s)} \right|}^{\bar p}}|{\mathcal{F}_{{\underline{s}}}}} \right)} \right]^{\frac{1}{2}}}{\left[ {\mathbb{E}\left( {{{\left| {f({Y^K}(s))} \right|}^{\bar p}}|{\mathcal{F}_{{\underline{s}}}}} \right)} \right]^{\frac{1}{2}}} \\
	&\le {C_T}{\left( {\delta _{{n_s}}^K} \right)^{\bar p(1 - 3\gamma )/4}}{\left( {1 + \left| {Y_{{\underline{s}}}^K} \right|} \right)^{\bar p}} \\ 
	&\le {C_T}{\left( {1 + \left| Y_{{\underline{s}}}^K \right|} \right)^{\bar p}} \quad \text{a.s.}
\end{split}
\end{equation}
Then, taking expectations on both sides, using \eqref{a33} and \eqref{a22} one computes
\begin{equation}   \label{a12}
\begin{split}
	&\mathbb{E}\left( {{{\left| {{{\hat Z}^K}(s) - {Y^K}(s)} \right|}^{\frac{{\bar p}}{2}}}{{\left| {f({Y^K}(s))} \right|}^{\frac{{\bar p}}{2}}}} \right) \\
	&\le {C_T}\mathbb{E}{\left( {1 + \left| Y_{{\underline{s}}}^K \right|} \right)^{\bar p}} \\
	&\le {C_T}\left( {1 + \mathbb{E}{{\left| Z_{{\underline{s}}}^K \right|}^{\bar p}}} \right) 
	= {C_T}\left( {1 + \mathbb{E}{{\left| {{Z^K}(s)} \right|}^{\bar p}}} \right) .
\end{split}
\end{equation}
Hence, using Young's inequality, H\"older's inequality, \eqref{a12}, \eqref{a29} and \eqref{a23} we have
\begin{equation}   \label{J21}
	\begin{split}
		J_2
		\le& \bar p\mathbb{E}\left[ {\int_0^T {{{\left| {{{\hat Z}^K}(s)} \right|}^{\bar p - 2}}\left| {{{\hat Z}^K}(s) - {Y^K}(s)} \right|\left| {f({Y^K}(s))} \right|ds} } \right] \\
		\le& (\bar p - 2)\int_0^T {\mathbb{E}{{\left| {{{\hat Z}^K}(s)} \right|}^{\bar p}}ds}  + 2\int_0^T {\mathbb{E}\left( {{{\left| {{{\hat Z}^K}(s) - {Y^K}(s)} \right|}^{\frac{{\bar p}}{2}}}{{\left| {f({Y^K}(s))} \right|}^{\frac{{\bar p}}{2}}}} \right)ds}  \\
		\le& {C_T} + {C_T}\int_0^T {\mathop {\sup }\limits_{0 \le u \le s} {{\left( {\mathbb{E}{{\left| {{{\hat Z}^K}(u)} \right|}^p}} \right)}^{\frac{{\bar p}}{p}}}ds}  \le C_T.
	\end{split}
\end{equation}
For $J_3$, using the Burkholder--Davis--Gundy's inequality (BDG's inequality, see e.g., \cite[Theorem 1.7.3]{klebaner2012introduction}, p. 201), Young's inequality and Assumption \ref{Ag} yields
\begin{equation}  \label{J31}
\begin{split}
	J_3 
	\le& 4\sqrt 2 \bar p\mathbb{E}\left( {{{\left| {\int_0^T {{{\left| {{{\hat Z}^K}(t)} \right|}^{2\bar p - 2}}{{\left| {g({Y^K}(t))} \right|}^2}dt} } \right|}^{\frac{1}{2}}}} \right) \\
	\le& 4\sqrt 2 \bar p\mathbb{E}\left[ {{{\left( {\left( {\mathop {\sup }\limits_{0 \le t \le T} {{\left| {{{\hat Z}^K}(t)} \right|}^{\bar p}}} \right)\int_0^T {{{\left| {{{\hat Z}^K}(t)} \right|}^{\bar p - 2}}{{\left| {g({Y^K}(t))} \right|}^2}dt} } \right)}^{\frac{1}{2}}}} \right]  \\
	\le& \frac{1}{2}\mathbb{E}\left( {\mathop {\sup }\limits_{0 \le t \le T} {{\left| {{{\hat Z}^K}(t)} \right|}^{\bar p}}} \right) + 16{{\bar p}^2}{K_3}\mathbb{E}\left[ {\int_0^T {{{\left| {{{\hat Z}^K}(t)} \right|}^{\bar p - 2}}\left( {1 + {{\left| {{Y^K}(t)} \right|}^l}} \right)dt} } \right] \\
	\le& \frac{1}{2}\mathbb{E}\left( {\mathop {\sup }\limits_{0 \le t \le T} {{\left| {{{\hat Z}^K}(t)} \right|}^{\bar p}}} \right) + C\int_0^T {\mathbb{E}{{\left| {{{\hat Z}^K}(t)} \right|}^p}dt}  + C\int_0^T {\mathbb{E}{{\left| {{Y^K}(t)} \right|}^p}dt}  + {C_T} .
\end{split}
\end{equation}
Using \eqref{a33}, \eqref{a22}, \eqref{a29} and \eqref{a23} one has
\begin{equation}  \label{J3}
\begin{split}
	{J_3}  \le \frac{1}{2}\mathbb{E}\left( {\mathop {\sup }\limits_{0 \le t \le T} {{\left| {{{\hat Z}^K}(t)} \right|}^{\bar p}}} \right)  + {C_T} .
\end{split}
\end{equation}
Substituting the above into \eqref{Jsum} yields
\begin{equation*}
	\mathbb{E}\left( {\mathop {\sup }\limits_{0 \le t \le T} {{\left| {{{\hat Z}^K}(t)} \right|}^{\bar p}}} \right) \le {\left| {{x_0}} \right|^{\bar p}} +  \frac{1}{2}\mathbb{E}\left( {\mathop {\sup }\limits_{0 \le t \le T} {{\left| {{{\hat Z}^K}(t)} \right|}^{\bar p}}} \right) + {C_T},
\end{equation*}
which implies the desired result
\begin{equation}  \label{a24}
	\mathbb{E}\left( {\mathop {\sup }\limits_{0 \le t \le T} {{\left| {{{\hat Z}^K}(t)} \right|}^{\bar p}}} \right) \le {C_T},
\end{equation}
where $C_T$ is a positive constant independent of $K$. 

\textbf{Step 4:} The estimation of $\mathbb{E}( {{{\sup }_{0 \le t \le T}}{{| {\hat Y(t)} |}^{\bar p}}} ) \le {C_T}$.
Define 
\begin{equation}  \label{a19}
	{\hat Y^K}(t) = {\pi _K}({{\hat{Z}}^K}(t)), \quad t \ge 0.
\end{equation}
Since $|\pi_K(x)| \le |x|$ for any $x \in \mathbb{R}^d$, it follows from \eqref{a24} and \eqref{a19} that
\begin{equation}  \label{a25}
	\mathbb{E}\left( {\mathop {\sup }\limits_{0 \le t \le T} {{\left| {{{\hat Y}^K}(t)} \right|}^{\bar p}}} \right) \le \mathbb{E}\left( {\mathop {\sup }\limits_{0 \le t \le T} {{\left| {{{\hat Z}^K}(t)} \right|}^{\bar p}}} \right) \le {C_T} ,
\end{equation}
where $C_T$ is independent of $K$. 
By the construction of the schemes in \eqref{scheme} and \eqref{K-scheme}, the two numerical solutions coincide before leaving the ball of radius $K$. In particular,
\begin{equation*}
	\left\{ {\omega :\mathop {\sup }\limits_{0 \le t \le T} \left| {{{\hat Y}^K}(t,\omega )} \right| < K} \right\} = \left\{ {\omega :\mathop {\sup }\limits_{0 \le t \le T} \left| {\hat Y(t,\omega )} \right| < K} \right\},
\end{equation*}
and on this event, 
\begin{equation*}
	{{\hat Y}^K}(t) = \hat Y(t), \quad  0 \le t \le T. 
\end{equation*}
Consequently, Markov's inequality gives
\begin{equation*}
	\mathbb{P}\left( {\mathop {\sup }\limits_{0 \le t \le T} {{\left| {\hat Y(t)} \right|}} < K} \right) = \mathbb{P}\left( {\mathop {\sup }\limits_{0 \le t \le T} {{\left| {{{\hat Y}^K}(t)} \right|}} < K} \right) \ge 1 - \frac{{\mathbb{E}\left( {\mathop {\sup }\limits_{0 \le t \le T} {{\left| {{{\hat Y}^K}(t)} \right|}^{\bar p}}} \right)}}{{{K^{\bar p}}}}.
\end{equation*}
Letting $K \to \infty$ on both sides, and using continuity of probability from below, we have
\begin{equation*}
\begin{split}
	\mathbb{P}\left( {\mathop {\sup }\limits_{0 \le t \le T} {{\left| {\hat Y(t)} \right|}} < \infty } \right) =& \mathbb{P}\left( {\bigcup\limits_{K \in {\mathbb{N}^ + },K > |{x_0}|}^\infty  {\left\{ {\mathop {\sup }\limits_{0 \le t \le T} {{\left| {\hat Y(t)} \right|}} < K} \right\}} } \right) \\
	=& \mathop {\lim }\limits_{K \to \infty } \mathbb{P}\left( {\mathop {\sup }\limits_{0 \le t \le T} {{\left| {\hat Y(t)} \right|}} < K} \right) = 1,
\end{split}
\end{equation*}
which implies that ${\sup _{0 \le t \le T}}| {\hat Y(t)}| < \infty $ a.s., that is, for almost all $\omega \in \Omega$, there exists a finite random variable $M_T(\omega)$ such that
\begin{equation}  \label{a14}
	\mathop {\sup }\limits_{0 \le t \le T} \left| {\hat Y(t)} \right| \le M_T(\omega ) .
\end{equation}
Since $\delta ( \cdot )$ is continuous and strictly positive, we get ${\delta _{\min }} (\omega ) = {\inf _{|x| \le M_T(\omega )}}\delta (x) > 0$. Then, 
\begin{equation*}
	{\delta _n}(\omega ) \ge {\delta _{\min }}(\omega ) > 0 ,
\end{equation*}
which implies
\begin{equation*}
	\mathop {\lim }\limits_{n \to \infty } {t_n}(\omega ) = \sum\limits_{n = 0}^\infty  {{\delta _n}(\omega )}  \ge \sum\limits_{n = 0}^\infty  {{\delta _{\min }}(\omega )}  = \infty .
\end{equation*}
Thus, $T$ is almost surely attainable.

Moreover, for almost every $\omega$, whenever $K > M_T(\omega)$, the consistency of the two schemes implies
\begin{equation*}
	{{\hat Y}^K}(t,\omega ) = \hat Y(t,\omega ), \quad 0 \le t \le T.
\end{equation*}
Hence,
\begin{equation*}
	\mathop {\lim }\limits_{K \to \infty } \mathop {\sup }\limits_{0 \le t \le T} {\left| {{{\hat Y}^K}(t,\omega )} \right|^{\bar p}} = \mathop {\sup }\limits_{0 \le t \le T} {\left| {\hat Y(t,\omega )} \right|^{\bar p}}.
\end{equation*}
Finally, Fatou's lemma gives
\begin{equation*}
	\mathbb{E}\left( {\mathop {\sup }\limits_{0 \le t \le T} {{\left| {\hat Y(t)} \right|}^{\bar p}}} \right) \le \mathop {\lim \inf }\limits_{K \to \infty } \mathbb{E}\left( {\mathop {\sup }\limits_{0 \le t \le T} {{\left| {{{\hat Y}^K}(t)} \right|}^{\bar p}}} \right) \le {C_T} .
\end{equation*}
The proof is therefore complete.
\end{proof}

\begin{remark}  \label{rem-a1}
We highlight two important observations regarding the adaptive scheme:
\begin{itemize}
\item[(i)] We emphasize that the estimate established in \eqref{a29}
\begin{equation*}
	\mathop {\sup }\limits_{K > |{x_0}|} \sup_{0 \le t \le T} \mathbb{E}\bigl| Z_{\underline{t}}^K \bigr|^p
	\le C_T,
\end{equation*}
is essential. Unlike the boundedness analysis for the fixed-step numerical schemes, such as \cite{mao2015the, sabanis2016euler}, this estimate cannot be inferred directly from
\begin{equation*}
	\mathop {\sup }\limits_{K > |{x_0}|} \sup_{0 \le t \le T} \mathbb{E}\bigl| \hat Z^K(t) \bigr|^p \le C_T
\end{equation*}
Indeed, the discretization $\underline{t}$ is random and may take different values on different sample paths. Therefore, a moment bound at each fixed time does not automatically yield a moment bound at the randomly selected time $\underline{t}$. Hence, the $p$th moment bound for $Z_{\underline{t}}^K$ must be established separately.

\item[(ii)] If the diffusion coefficient $g(\cdot)$ satisfies a linear growth condition, i.e., $l = 2$ in Assumption \ref{Ag}, one derives from \eqref{a16}-\eqref{a27} that the admissible range of $\gamma$ can be extended to $(0, 1/2]$.
\end{itemize}
\end{remark}

\begin{remark}
	Fang and Giles \cite{fang2020adaptive} proposed an adaptive timestep function of the form 
	\begin{equation}  \label{fang-stepsize}
		\left\langle {x,f(x)} \right\rangle  + \frac{1}{2}\delta (x){\left| {f(x)} \right|^2} \le \alpha {\left| x \right|^2} + \beta,
	\end{equation}
	where $\delta (\cdot) : \mathbb{R}^d \to \mathbb{R}_+$ is continuous and strictly positive, and $\alpha, \beta $ are positive constants. This approach exploits the inner-product term (which, for instance, provides dissipativity when $f(x) = x - x^3$) together with the adaptive timestep $\delta(\cdot)$ to control the superlinear growth of the drift coefficient $f(\cdot)$. The diffusion coefficient $g(\cdot)$ satisfies a linear growth condition. Specific time-stepping rules are proposed for SDEs 
	\begin{equation}  \label{re-ab2}
		\delta (x) = \frac{{1 \vee {{\left| x \right|}^2}}}{{1 \vee {{\left| {f(x)} \right|}^2}}} .
	\end{equation}
	One computes that for any $x \in \mathbb{R}^d$,
	\begin{equation*}
		\begin{split}
			{\left| {f(x)} \right|^2} \le& 1 \vee {\left| {f(x)} \right|^2} = \frac{{1 \vee {{\left| x \right|}^2}}}{{\delta (x)}} \le {\left( {\delta (x)} \right)^{ - 1}}{\left( {1 + \left| x \right|} \right)^2}.
		\end{split}
	\end{equation*}
	Hence, by Remark \ref{rem-a1}, when the diffusion coefficient satisfies a linear growth condition, Assumption \ref{A-stepsize} is satisfied with $\gamma = 1/2$.
	
	In addition, an admissible time-stepping strategy for an adaptive EM scheme with a tamed backstop is proposed in \cite[Definition 2.2]{kelly2018adaptive}, that is, there exist real non-negative constants $R_1, R_2 < \infty$ such that whenever $\delta_{\min} < \delta_n < \delta_{\max}$,
	\begin{equation}  \label{re-ab3}
		{\left| {f({Y_{{t_n}}})} \right|^2} \le {R_1} + {R_2}{\left| {{Y_{{t_n}}}} \right|^2}, \quad n = 0,1,...
	\end{equation}	
	Setting ${\delta _{\max }} \le {({R_1} \vee {R_2})^{ - 1}}$, then
	\begin{equation*}
		|f(Y_{t_n})|^2 \le (R_1 \vee R_2)(1+|Y_{t_n}|^2)
		\le \delta_{\max}^{-1}(1+|x|^2)
		\le \delta_n^{-1}(1+|Y_{t_n}|)^2.
	\end{equation*}
Thus, for the adaptive timesteps satisfying $\delta_{\min} < \delta_n < \delta_{\max}$, Assumption \ref{A-stepsize} holds with $\gamma = 1/2$.

\end{remark}


\section{Strong convergence and convergence rate in finite time intervals}  \label{sec4}

For fixed-step numerical methods, strong convergence is typically analyzed by letting uniform stepsize $\delta$ tend to zero (see e.g., \cite[Theorem 2.2]{higham2002strong}, \cite[Theorem 3.5]{mao2015the}, \cite[Theorem 1]{sabanis2016euler}, \cite[Theorem 3.3]{li2019explicit}). However, the primary issue is that adaptive methods do not necessarily drive the time steps to zero \cite{lamba2007an}. To tackle this problem, we introduce a refined timestep function $\delta^\Delta(\cdot)$, indexed by a refinement parameter $\Delta \in (0, 1]$. The strong convergence of the adaptive time-stepping EM scheme is then established as $\Delta \to 0$.

\subsection{Strong convergence}
\begin{assumption}  \label{A-Delta}
	For each $\Delta \in (0,1]$, the refined timestep function ${\delta ^\Delta }:{\mathbb{R}^d} \to {\mathbb{R}_ + }$ is continuous and satisfies
	\begin{equation*}
		0 < {\delta ^\Delta }(x) \le \min \{ \Delta T,\delta (x)\} , \quad  x \in \mathbb{R}^d.
	\end{equation*}
\end{assumption}

With this assumption in hand, we present strong convergence of scheme \eqref{scheme} with timestep function $\delta^\Delta (\cdot)$.

\begin{theorem}  \label{th-convergence}
Let Assumptions \ref{A2.1}, \ref{Ag}, \ref{A-stepsize} and \ref{A-Delta} hold and assume that $p > l$. Then for any $q \in [2, p+2-l)$,
	\begin{equation*}
		\mathop {\lim }\limits_{\Delta  \to 0} \mathbb{E}\left( {\mathop {\sup }\limits_{0 \le t \le T} {{\left| {X(t) - \hat Y(t)} \right|}^q}} \right) = 0, \quad  \forall T \ge 0. 
	\end{equation*}
\end{theorem}
The proof is essentially identical to the uniform timestep EM analysis in \cite[Theorem 2.2]{higham2002strong}. The only change required by the use of an adaptive timestep is to note that 
$t- \underline{t} \le \Delta T$, and that $\mathbb{E}\left( {{{\left| {W(t) - W(\underline{t})} \right|}^p}}| \mathcal{F}_{\underline{t}} \right) \le {C_p}{\left| {t - \underline{t}} \right|^{p/2}}$  a.s.

\subsection{Convergence rate}
\begin{assumption} \label{A3.1}
There exist positive constants ${p^ * } >2$, $K_4, K_5$ and $l > 2 $ such that 
\begin{equation}  \label{A3.1-1}
	\left\langle {x - y,f(x) - f(y)} \right\rangle  + \frac{{{p^*} - 1}}{2}{\left| {g(x) - g(y)} \right|^2} \le {K_4}{\left| {x - y} \right|^2},
\end{equation}
\begin{equation}   \label{f}
	\left| {f(x) - f(y)} \right| \le {K_5}\left( {1 + |x{|^{l-2}} + |y{|^{l-2}}} \right)\left| {x - y} \right|,
\end{equation}
for any $x, y \in \mathbb{R}^d$.
\end{assumption}

\begin{remark}  \label{remark-2}
Notice that under Assumption \ref{A3.1}, we have
\begin{equation}  \label{g}
	{\left| {g(x) - g(y)} \right|^2} \le C\left( {1 + |x{|^{l-2}} + |y{|^{l-2}}} \right){\left| {x - y} \right|^2}
\end{equation}
for any $ x, y \in \mathbb{R}^d$.
Moreover, using triangle inequality and Young's inequality we obtain
\begin{equation}   \label{fg-super}
	\left| {f(x)} \right| \le  C\left( {1 + |x{|^{l - 1}}} \right), \quad \left| {g(x)} \right| \le C\left( {1 + |x{|^{\frac{l}{2}}}} \right), 
\end{equation}
for any $ x \in \mathbb{R}^d$.
This implies that Assumption \ref{Ag} holds.
\end{remark}

\begin{lemma}   \label{YYbar}
	Assume that \eqref{A2.1-1} holds for some $ p \ge 2(l-1)$, and that Assumptions \ref{A-stepsize}, \ref{A-Delta} and \ref{A3.1} are satisfied. For any ${q_1} \in [2, (2(p+2-l))/l]$, 
	\begin{equation*}
		\mathop {\sup }\limits_{0 \le t \le T} \mathbb{E}{\left| {\hat Y(t) - Y(t)} \right|^{{q_1}}} \le {C_T}{\Delta ^{\frac{{{q_1}}}{2}}}, \quad \forall T > 0.
	\end{equation*}
\end{lemma}
\begin{proof}
Fix $T > 0$. For any $t \in [0, T]$, using \eqref{continuous1}, Remark \ref{remark-steps} and Assumptions \ref{A-stepsize}, condition \eqref{fg-super} we get
\begin{equation*}
	\begin{split}
		&\mathbb{E}\left( {{{\left| {\hat Y(t) - Y(t)} \right|}^{{q_1}}}|{{\cal F}_{{\underline{t}}}}} \right) \\
		&\le {2^{{q_1} - 1}}{\left| {f({Y_{{\underline{t}}}})} \right|^{{q_1}}}\mathbb{E}\left( {{{\left| {t - {\underline{t}}} \right|}^{{q_1}}}|{{\cal F}_{{\underline{t}}}}} \right) + {2^{{q_1} - 1}}{\left| {g({Y_{{\underline{t}}}})} \right|^{{q_1}}}\mathbb{E}\left( {{{\left| {W(t) - W({\underline{t}})} \right|}^{{q_1}}}|{{\cal F}_{{\underline{t}}}}} \right) \\
		&\le C{\left( {\delta _{{n_t}}^\Delta } \right)^{{q_1}(1 - \gamma )}}{\left( {1 + \left| {{Y_{{\underline{t}}}}} \right|} \right)^{{q_1}}} + C{\left( {1 + {{\left| {{Y_{{\underline{t}}}}} \right|}^{\frac{l}{2}}}} \right)^{{q_1}}}{\left( {\delta _{{n_t}}^\Delta } \right)^{\frac{{{q_1}}}{2}}} \\
		&\le C_T{\Delta ^{\frac{{{q_1}}}{2}}}\left( {1 + {{\left| {{Y_{{\underline{t}}}}} \right|}^{\frac{{{q_1}l}}{2}}}} \right).
	\end{split}
\end{equation*}
Taking expectations on both sides, using Theorem \ref{th2.1} yields
\begin{equation*}
	\begin{split}
		\mathbb{E}{\left| {\hat Y(t) - Y(t)} \right|^{{q_1}}} 
		\le& {C_T}{\Delta ^{\frac{{{q_1}}}{2}}}\left( {1 + \mathbb{E}{{\left| {{Y_{{\underline{t}}}}} \right|}^{\frac{{{q_1}l}}{2}}}} \right)\\
		\le& {C_T}{\Delta ^{\frac{{{q_1}}}{2}}}\left[ {1 + {{\left( {\mathbb{E}\left( {\mathop {\sup }\limits_{0 \le t \le T} {{\left| {\hat Y(t)} \right|}^{\bar p}}} \right)} \right)}^{\frac{{{q_1}l}}{{2\bar p}}}}} \right] \\
		\le& {C_T}{\Delta ^{{{{q_1}}}/{2}}} .
	\end{split}
\end{equation*}
The proof is therefore complete.
\end{proof}

Before showing the convergence rate of scheme \eqref{scheme}, we introduce the stochastic Gronwall inequality.
\begin{lemma} (\cite[Theorem 4]{scheutzow2013a})  \label{lem-gronwall}
Let $Z$ and $H$ be nonnegative, adapted processes with continuous paths. Let $M$ be a continuous local martingale starting at $0$. If 
\begin{equation*}
	Z(t) \le \int_0^t {Z(s)ds}  + H(t) + M(t)
\end{equation*}	
holds for all $t \ge 0$, then for every $\alpha \in (0, 1)$, we have
\begin{equation*}
	\mathbb{E}\left( {\mathop {\sup }\limits_{0 \le s \le t} {{\left| {Z(s)} \right|}^\alpha }} \right) \le {C_\alpha }{e^{\alpha t}}\mathbb{E}\left( {\mathop {\sup }\limits_{0 \le s \le t} {{\left| {H(s)} \right|}^\alpha }} \right)
\end{equation*}

\end{lemma}

\begin{theorem}   \label{rate1}
Let \eqref{A2.1-1}, Assumptions \ref{A-stepsize}, \ref{A-Delta} and \ref{A3.1} hold with $p \ge 4l-6 $. Then, for every $\alpha \in (0,1)$ and any 
$q \in [2\alpha ,{p^*}\alpha ) \cap \left[ {2\alpha ,{\frac{{2\alpha (p + 2 - l)}}{{3l - 4}}}} \right]$,
\begin{equation*}
	\mathbb{E}\left( {\mathop {\sup }\limits_{0 \le t \le T} {{\left| {X(t) - \hat Y(t)} \right|}^q}} \right) \le {C_T}{\Delta ^{{q}/{2}}}, \quad \forall T > 0.
\end{equation*}
\end{theorem}

\begin{proof}
Fix $\alpha \in (0,1)$, set $ q_2 := q/\alpha \ge 2$. By It\^o's formula, for any $t \in [0,T ]$, 
\begin{equation}   \label{lem-con-3}
	\begin{split}
		{\left| {X(t) - \hat Y(t)} \right|^{q_2}}=& \frac{{q_2}}{2}\int_0^t {{{\left| {X(s) - \hat Y(s)} \right|}^{{q_2} - 2}}\Big[ {2\left\langle {X(s) - \hat Y(s),f(X(s)) - f( Y(s))} \right\rangle } }  \\
		&{ + ({q_2} - 1){{\left| {g(X(s)) - g( Y(s))} \right|}^2}} \Big]ds 
		+ M(t) ,
	\end{split}
\end{equation}
where $M(t) = {q_2}\int_0^t {{{| {X(s) - \hat Y(s)} |}^{q_2 - 2}}{{( {X(s) - \hat Y(s)} )}^{\rm T}}( {g(X(s)) - g(Y(s))} )dW(s)} $ is a local martingale with initial value 0.
Note that for any constants $u > v \ge 1$, $x,y \in \mathbb{R}^d$, we have 
\begin{equation*}
	(v - 1){\left| {x + y} \right|^2} \le (u - 1){\left| x \right|^2} + \frac{{(u - 1)(v - 1)}}{{u - v}}{\left| y \right|^2}.
\end{equation*} 
Taking $u = p^*$, $v = q_2$ and using $q_2 < p^*$, we obtain from Assumption \ref{A3.1} that
\begin{equation}    \label{ac5}
	\begin{split}
		&2\left\langle {X(s) - \hat Y(s),f(X(s)) - f( Y(s))} \right\rangle  + (q_2 - 1){\left| {g(X(s)) - g( Y(s))} \right|^2} \\
		&\le 2\left\langle {X(s) - \hat Y(s),f(X(s)) - f(\hat Y(s))} \right\rangle  + 2\left\langle {X(s) - \hat Y(s),f(\hat Y(s)) - f( Y(s))} \right\rangle  \\
		&\quad+ ({p^*} - 1){\left| {g(X(s)) - g(\hat Y(s))} \right|^2} + \frac{{({p^*} - 1)(q_2 - 1)}}{{{p^*} - q_2}}{\left| {g(\hat Y(s)) - g( Y(s))} \right|^2} \\
		&\le 2{K_4}{\left| {X(s) - \hat Y(s)} \right|^2} + 2\left| {X(s) - \hat Y(s)} \right|\left| {f(\hat Y(s)) - f( Y(s))} \right| \\
		&\quad+ \frac{{({p^*} - 1)(q_2 - 1)}}{{{p^*} - q_2}}{\left| {g(\hat Y(s)) - g( Y(s))} \right|^2}.
	\end{split}
\end{equation}
Inserting \eqref{ac5} into \eqref{lem-con-3} and using Young's inequality yields
\begin{equation*} 
	\begin{split}
		&{\left| {X(t) - \hat Y(t)} \right|^{q_2}} \\
		&\le {q_2}{K_4}\int_0^t {{{\left| {X(s) - \hat Y(s)} \right|}^{q_2}}ds}  + q_2\int_0^t {{{\left| {X(s) - \hat Y(s)} \right|}^{{q_2} - 1}}\left| {f(\hat Y(s)) - f( Y(s))} \right|ds} \\
		&\quad+ \frac{{q_2({p^*} - 1)({q_2} - 1)}}{{2({p^*} - {q_2})}}\int_0^t {{{\left| {X(s) - \hat Y(s)} \right|}^{{q_2} - 2}}{{\left| {g(\hat Y(s)) - g( Y(s))} \right|}^2}ds}+ M(t) \\
		&\le C\int_0^t {{{\left| {X(s) - \hat Y(s)} \right|}^{q_2}}ds}+ C\int_0^t {{{\left| {f(\hat Y(s)) - f( Y(s))} \right|}^{q_2}}ds}  + C\int_0^t {{{\left| {g(\hat Y(s)) - g( Y(s))} \right|}^{q_2}}ds} \\
		&\quad + M(t) .
	\end{split}
\end{equation*}
Define 
\begin{equation}  \label{ac7}
	H(t) = C\int_0^t {{{\left| {f(\hat Y(s)) - f( Y(s))} \right|}^{q_2}}ds}  + C\int_0^t {{{\left| {g(\hat Y(s)) - g( Y(s))} \right|}^{q_2}}ds}.
\end{equation}
Since $H(t)$ is nonnegative, Lemma \ref{lem-gronwall} gives
\begin{equation*} 
	\mathbb{E}\left( {\mathop {\sup }\limits_{0 \le t \le T} {{\left| {X(t) - \hat Y(t)} \right|}^q}} \right) \le {C_{\alpha ,T}}{e^{\alpha ,T}}\mathbb{E}\left( {\mathop {\sup }\limits_{0 \le t \le T} {{\left( {H(t)} \right)}^\alpha }} \right) \le {C_{\alpha ,T}}\left( {\mathbb{E}{{\left( {H(T)} \right)}^\alpha }} \right).
\end{equation*}
Due to the fact $\alpha \in (0, 1)$, applying Jensen's inequality yields $\mathbb{E}{\left( {H(T)} \right)^\alpha } \le {\left( {\mathbb{E}H(T)} \right)^\alpha }$. Therefore,
\begin{equation}   \label{ac6}
	\mathbb{E}\left( {\mathop {\sup }\limits_{0 \le t \le T} {{\left| {X(t) - \hat Y(t)} \right|}^q}} \right) \le C_{\alpha,T}{\left( {\mathbb{E}H(T)} \right)^\alpha } .
\end{equation}
It remains to estimate $\mathbb{E}H(T)$.
Using \eqref{f}, \eqref{g} we have
\begin{equation}  \label{ac4}
\begin{split}
	\mathbb{E}H(T) 
	\le& C\mathbb{E}\left( {\int_0^T {{{\left( {1 + {{\left| {\hat Y(t)} \right|}^{l - 2}} + {{\left| {Y(t)} \right|}^{l - 2}}} \right)}^{{q_2}}}{{\left| {\hat Y(t) - Y(t)} \right|}^{{q_2}}}dt} } \right) \\
	&+ C\mathbb{E}\left( {\int_0^T {{{\left( {1 + {{\left| {\hat Y(t)} \right|}^{l - 2}} + {{\left| {Y(t)} \right|}^{l - 2}}} \right)}^{{q_2}/{2}}}{{\left| {\hat Y(t) - Y(t)} \right|}^{{q_2}}}dt} } \right) \\
	\le& C\int_0^T {\mathbb{E}\left[ {{{\left( {1 + {{\left| {\hat Y(t)} \right|}^{l - 2}} + {{\left| {Y(t)} \right|}^{l - 2}}} \right)}^{{q_2}}}{{\left| {\hat Y(t) - Y(t)} \right|}^{{q_2}}}} \right]dt} .
\end{split}
\end{equation}
Since $q_2 = q/\alpha$ and $q \in \left[ {2\alpha ,2\alpha (p + 2 - l)/(3l - 4)} \right]$,
using H\"older's inequality, Jensen's inequality, Theorem \ref{th2.1} and Lemma \ref{YYbar} yields
\begin{equation}  \label{lem-con-2}
	\begin{split}
		&\mathbb{E}\left[ {{{\left( {1 + {{\left| {\hat Y(t)} \right|}^{l - 2}} + {{\left| {Y(t)} \right|}^{l - 2}}} \right)}^{q_2}}{{\left| {\hat Y(t) - Y(t)} \right|}^{q_2}}} \right] \\
		&\le C\mathbb{E}\left[ {\left( {1 + {{\left| {\hat Y(t)} \right|}^{{q_2}(l - 2)}} + {{\left| {Y(t)} \right|}^{{q_2}(l - 2)}}} \right){{\left| {\hat Y(t) - Y(t)} \right|}^{q_2}}} \right] \\
		&\le C{\left[ {\mathbb{E}{{\left( {1 + {{\left| {\hat Y(t)} \right|}^{{q_2}(l - 2)}} + {{\left| {Y(t)} \right|}^{{q_2}(l - 2)}}} \right)}^2}} \right]^{\frac{1}{2}}}{\left[ {\mathbb{E}{{\left| {\hat Y(t) - Y(t)} \right|}^{2{q_2}}}} \right]^{\frac{1}{2}}} \\
		&\le C{\left[ {\mathbb{E}{{\left( {1 + {{\left| {\hat Y(t)} \right|}^{{q_2}(l - 2)}} + {{\left| {Y(t)} \right|}^{{q_2}(l - 2)}}} \right)}^{\frac{{\bar p}}{{{q_2}(l - 2)}}}}} \right]^{\frac{{{q_2}(l - 2)}}{{\bar p}}}}{\left[ {\mathbb{E}{{\left| {\hat Y(t) - Y(t)} \right|}^{\frac{{\bar p{q_2}}}{{\bar p - {q_2}(l - 2)}}}}} \right]^{\frac{{\bar p - {q_2}(l - 2)}}{{\bar p}}}} \\
		&\le C{\left[ {1 + \mathbb{E}{{\left| {\hat Y(t)} \right|}^{\bar p}} + \mathbb{E}{{\left| {Y(t)} \right|}^{\bar p}}} \right]^{\frac{{{q_2}(l - 2)}}{{\bar p}}}}{\left[ {\mathbb{E}{{\left| {\hat Y(t) - Y(t)} \right|}^{\frac{{2\bar p}}{l}}}} \right]^{\frac{{{q_2}l}}{{2\bar p}}}} \\
		&\le {C_T}{\Delta ^{{{q_2}}/{2}}} .
	\end{split}
\end{equation}
Substituting \eqref{ac4} and \eqref{lem-con-2} into \eqref{ac6} yields
\begin{equation*}
	\mathbb{E}\left( {\mathop {\sup }\limits_{0 \le t \le T} {{\left| {X(t) - \hat Y(t)} \right|}^q}} \right) \le {C_{\alpha ,T}}{\Delta ^{q_2\alpha/2}} = {C_{\alpha ,T}}{\Delta ^{{q}/{2}}}.
\end{equation*}
This completes the proof.
\end{proof}

A first-order strong convergence rate of the adaptive time-stepping EM scheme \eqref{scheme} is attainable for the SDE \eqref{equation} with additive noise, namely, when $m=d$ and $g(\cdot)\equiv I_d$. This improvement, however, requires stronger assumptions on the drift coefficient $f(\cdot)$. 

\begin{remark}  \label{remark-3}
		If we take $m = d$ and $g  \equiv  I_d$ in SDE \eqref{equation},  condition \eqref{A2.1-1}, Assumption \ref{A-stepsize} and \eqref{A3.1-1} reduce to
		\begin{equation}  \label{pq11}
			{\left| {f(x)} \right|^2}\delta (x) \le C(1 + {\left| x \right|^2}), \quad  \left\langle {x,f(x)} \right\rangle  \le C(1 + {\left| x \right|^2}), \quad \left\langle {x - y,f(x) - f(y)} \right\rangle  \le {K_4}{\left| {x - y} \right|^2}.
		\end{equation}
		Assume that $f(\cdot)$ satisfies \eqref{pq11} and is differentiable, and that $f(\cdot)$ and $\nabla f (\cdot)$ satisfy the polynomial growth Lipschitz condition,
		\begin{equation}  \label{pq12}
			\left| {f(x) - f(y)} \right| + \left| {\nabla f(x) - \nabla f(y)} \right| \le C\left( {1 + {{\left| x \right|}^{l - 2}} + {{\left| y \right|}^{l - 2}}} \right)\left| {x - y} \right|
		\end{equation}
		for some $l \ge 2$ and $C > 0$. Hence, for any $q, T \in (0, \infty)$, there exists a constant $C_{ T}$ such that
		\begin{equation*}
			\mathbb{E}\left( {\mathop {\sup }\limits_{0 \le t \le T} {{\left| {X(t) - \hat Y(t)} \right|}^q}} \right) \le {C_{T}}{\Delta ^q}.
		\end{equation*}
		The proof is analogous to that of Fang and Giles \cite[Theorem 3]{fang2016adaptive} and is therefore omitted.
	\end{remark}

\section{Uniform-in-time $p$th moment boundedness}  \label{sec5}
In this section, we establish the uniform-in-time $p$th moment boundedness of the numerical solution \eqref{continuous2}. This property guarantees that the numerical scheme \eqref{scheme} remains stable over long-time intervals. Furthermore, such a bound provides a fundamental basis for subsequent theoretical analysis and plays a key role in the study of the uniform-in-time strong convergence rate and the ergodicity of the numerical scheme. First, we impose a coercive Khasminskii-type condition.

\begin{assumption}  \label{A4.1}
There exist constants $p \ge 2$ and ${K_6} >0,{K_7} \ge 0$ such that for any $x \in \mathbb{R}^d$,
\begin{equation*}
	\left\langle {x,f(x)} \right\rangle  + \frac{{p - 1}}{2}{\left| {g(x)} \right|^2} \le  - {K_6}{\left| x \right|^2} + {K_7}.
\end{equation*}
\end{assumption}

\begin{lemma} (\cite[Theorem 5.2]{li2019explicit})  \label{lemma-EXt0}
Under \eqref{llc} and Assumption \ref{A4.1}, the solution $X(t)$ of SDE \eqref{equation} satisfies
\begin{equation*}
	\mathop {\sup }\limits_{t \ge 0} \mathbb{E}{\left| {X(t)} \right|^p} \le C .
\end{equation*}
\end{lemma}

We now impose a critical assumption about the adaptive timestep for the uniform-in-time estimate.

\begin{assumption}  \label{A-stepbound}
	Assume that the adaptive timestep function $\delta: \mathbb{R}^d \to (0, \delta_{\max}]$ is continuous and satisfies
	\begin{equation*}
		\mathop {\lim }\limits_{|x| \to \infty } \delta (x) = 0, 
	\end{equation*}
 where $\delta_{\max} \in (0, \infty)$ is a positive constant.
\end{assumption}

\begin{theorem}  \label{th-bounded}
Suppose that \eqref{llc}, Assumptions \ref{A-stepsize}, \ref{A4.1} and \ref{A-stepbound} hold. Then, the adaptive time-stepping EM scheme \eqref{scheme} has the property
\begin{equation*} 
	 \mathop {\sup }\limits_{t \ge 0} \mathbb{E}{\left| {Y(t)} \right|^p} \le C,  \quad \mathop {\sup }\limits_{t \ge 0} \mathbb{E}{\left| {\hat Y(t)} \right|^p} \le C.
\end{equation*}
\end{theorem}

\begin{proof}
Since Assumption \ref{A4.1} implies the Khasminskii-type condition \eqref{A2.1-1}, Theorem  \ref{th2.1} ensures that $T$ is almost surely attainable. It is therefore enough to prove the uniform-in-time pth-moment bounds for the adaptive time-stepping EM scheme \eqref{scheme}. For any integer $n \ge 0$, from \eqref{scheme} we have
\begin{equation*}  
	\begin{split}
		{\left| {{Y_{{t_{n + 1}}}}} \right|^2} =& {\left| {{Y_{{t_n}}}} \right|^2} + 2\left\langle {{Y_{{t_n}}},f({Y_{{t_n}}})} \right\rangle \delta _n  + {\left| {f({Y_{{t_n}}})} \right|^2}{\left| {\delta _n } \right|^2} + {\left| {g({Y_{{t_n}}})\Delta {W_n}} \right|^2} \\
		&+ 2\left\langle {{Y_{{t_n}}} + f({Y_{{t_n}}})\delta _n ,g({Y_{{t_n}}})\Delta {W_n}} \right\rangle .
	\end{split}
\end{equation*}
Let $V_p(x) := {(1 + |x|^2)^{p/2}}$ 
for any $x \in \mathbb{R}^d$, where $p$ is given in Assumption \ref{A4.1}, and let $V_n := V_p(Y_{t_n})$.
Using a similar argument as in \textbf{Step 2} of Theorem \ref{th2.1}, there exists a constant $C_0> 0$ independent of $n, Y_{t_n}, \delta_n$ such that
\begin{equation*}
\begin{split}
	&\mathbb{E}\left[ {V_{n+1}|{{\cal F}_{{t_n}}}} \right] \\
	&\le \left[ {1 + C_0{{\left( {{\delta _n}} \right)}^{2(1 - \gamma )}}} \right]{V_n} + \frac{{p{\delta _n}}}{2}{\left( {1 + {{\left| {{Y_{{t_n}}}} \right|}^2}} \right)^{p/2 - 1}}\left( {2\left\langle {{Y_{{t_n}}},f({Y_{{t_n}}})} \right\rangle  + (p - 1){{\left| {g({Y_{{t_n}}})} \right|}^2}} \right) \quad \text{a.s.}
\end{split}
\end{equation*} 
By Assumption \ref{A4.1} one sees
\begin{equation}  \label{bd1}
\begin{split}
	&\mathbb{E}\left[ {V_{n+1}|{{\cal F}_{{t_n}}}} \right] \\
	&\le \left[ {1 + C_0{{\left( {{\delta _n}} \right)}^{2(1 - \gamma )}}} \right]{V_n} + \frac{{p{\delta _n}}}{2}{\left( {1 + {{\left| {{Y_{{t_n}}}} \right|}^2}} \right)^{p/2 - 1}}\left( { - 2{K_6}\left( {1 + {{\left| {{Y_{{t_n}}}} \right|}^2}} \right) + 2({K_6} + {K_7})} \right) \\
	&= \left[ {1 + C_0{{\left( {{\delta _n}} \right)}^{2(1 - \gamma )}}} \right]{V_n} - p{K_6}{\delta _n}{V_n} + p({K_6} + {K_7}){\delta _n}{\left( {1 + {{\left| {{Y_{{t_n}}}} \right|}^2}} \right)^{p/2 - 1}} \quad \text{a.s.}
\end{split}
\end{equation}
On the one hand, since
\begin{equation*}
	\frac{{{{(1 + |x|^2)}^{p/2 - 1}}}}{{{{(1 +  |x|^2)}^{p/2}}}} = \frac{1}{{1 +  |x|^2}} \to 0 \quad \text{as} \; |x|\to \infty,
\end{equation*}
there exists a positive constant $R_1 > 0$ such that 
\begin{equation*}
	p({K_6} + {K_7}){\left( {1 + {{\left| x \right|}^2}} \right)^{p/2 - 1}} \le {pK_6}V_p(x)/4, \quad |x| > R_1.
\end{equation*}
Besides, using Assumption \ref{A-stepbound} and the fact $\gamma \in (0, 1/3]$, there exists a positive constant $R_2 > 0$ such that
\begin{equation*}
	C_0{(\delta (x))^{1 - 2\gamma }} \le {p{K_6}}/{4}, \quad |x| > R_2.
\end{equation*}
Choose $R = \max\{R_1, R_2\}$, we obtain
\begin{equation} \label{bd8}
	\mathbb{E}\left[ {{V_{n + 1}}|{\mathcal{F}_{{t_n}}}} \right] \le \left( {1 - \frac{{p{K_6}}}{2}{\delta _n}} \right){V_n}, \quad |Y_{t_n}| > R.
\end{equation}
On the other hand, if $|Y_{t_n}| \le R$, then $V_n
\le V_R := (1+R^2)^{p/2}$. Applying Assumption \ref{A-stepbound} we have
\begin{equation*}
	p({K_6} + {K_7}){\delta _n}{\left( {1 + {{\left| {{Y_{{t_n}}}} \right|}^2}} \right)^{p/2 - 1}} \le {C_R}{\delta _n}, \quad C_0{\left( {{\delta _n}} \right)^{2(1 - \gamma )}}V_n \le C_0{\left( {{\delta _{\max }}} \right)^{1 - 2\gamma }}{(1 + {R^2})^{p/2}}{\delta _n}. 
\end{equation*}
Hence, we have
\begin{equation}  \label{bd6}
	\mathbb{E}\left[ {{V_{n + 1}}|{{\cal F}_{{t_n}}}} \right] \le {V_n} + {C_R}{\delta _n}, \quad |Y_{t_n}| \le R
\end{equation}
Consequently, together with \eqref{bd8} and \eqref{bd6} we obtain
\begin{equation}  \label{bd9}
\begin{split}
	\mathbb{E}\left[ {V_{n+1}|{{\cal F}_{{t_n}}}} \right] \le& V_n - \frac{{p{K_6}}}{2}{\delta _n}V_n{\mathcal{I}_{\{ |{Y_{{t_n}}}| > R\} }} + {C_R}{\delta _n}{\mathcal{I}_{\{ |{Y_{{t_n}}}| \le R\} }} \\
	=& V_n - \frac{{p{K_6}}}{2}{\delta _n}V_n + \frac{{p{K_6}}}{2}{\delta _n}V_n{\mathcal{I}_{\{ |{Y_{{t_n}}}| \le R\} }} + {C_R}{\delta _n}{\mathcal{I}_{\{ |{Y_{{t_n}}}| \le R\} }} \\
	\le& \left( {1 - \frac{{p{K_6}}}{2}{\delta _n}} \right)V_n + C{\delta _n}  \quad \text{a.s.}
\end{split}
\end{equation} 
where $C = pK_6 V_R/2 + C_R$ is a positive constant depending on $R, p, K_6, K_7, \gamma, \delta_{\max}$, but is independent of $n$ and $t$.
Set $\lambda = pK_6 /2$. Multiplying both sides of the above inequality by ${e^{\lambda {t_{n + 1}}}}$, using the fact $1 - x \le e^{-x}$ for any $x \ge 0 $, one has
\begin{equation}  \label{bd3}
	\begin{split}
		{e^{\lambda {t_{n + 1}}}}\mathbb{E}\left[ {{V_{n + 1}}|{\mathcal{F}_{{t_n}}}} \right] 
		\le& {e^{\lambda {t_{n + 1}}}}(1 - \lambda {\delta _n})V_n + C{e^{\lambda {t_{n + 1}}}}{\delta _n}  \\
		\le& {e^{\lambda {t_n}}}V_n + C{e^{\lambda {\delta _{\max }}}}{e^{\lambda {t_n}}}{\delta _n} \quad \text{a.s.}
	\end{split}
\end{equation}
Since ${e^{\lambda t}}$ is strictly increasing with respect to $t$ for any $t \ge 0$, we have
\begin{equation*}
	{e^{\lambda {t_n}}}{\delta _n} \le \int_{{t_n}}^{{t_{n + 1}}} {{e^{\lambda s}}ds}   \quad \text{a.s.}
\end{equation*}
This together with \eqref{bd3} yields
\begin{equation}  \label{bd4}
	\begin{split}
		{e^{\lambda {t_{n + 1}}}}\mathbb{E}\left[ {{V_{n + 1}}|{\mathcal{F}_{{t_n}}}} \right]
		\le {e^{\lambda {t_n}}}V_n + C{e^{\lambda {\delta _{\max }}}}\int_{{t_n}}^{{t_{n + 1}}} {{e^{\lambda s}}ds}   \quad  \text{a.s.}
	\end{split}
\end{equation}
Since $t_{n+1}$ is $\mathcal{F}_{t_n}$-measurable by Remark \ref{remark-filtration},
\begin{equation*}
\begin{split}
	\mathbb{E}\left[ {{e^{\lambda {t_{n + 1}}}}V_{n+1} - C{e^{\lambda {\delta _{\max }}}}\int_0^{{t_{n + 1}}} {{e^{\lambda s}}ds} |{\mathcal{F}_{{t_n}}}} \right] 
	\le {e^{\lambda {t_n}}}V_n - C{e^{\lambda {\delta _{\max }}}}\int_0^{{t_n}} {{e^{\lambda s}}ds}  \quad \text{a.s.}
\end{split}
\end{equation*}
Hence, 
\begin{equation*}
	{M_n} := {e^{\lambda {t_n}}}V_n - C{e^{\lambda {\delta _{\max }}}}\int_0^{{t_n}} {{e^{\lambda s}}ds} 
\end{equation*}
is a supermartingale with respect to ${\{ {\mathcal{F}_{{t_n}}}\} _{n \ge 0}}$. For any $t > 0$ and $m \in \mathbb{N}$, define the bounded stopping time ${\varsigma _{t,m}} = {n_t} \wedge m$. By the optional stopping theorem \cite[Theorem 7.14, p.~189]{klebaner2012introduction}, one has $\mathbb{E}{M_{{\varsigma _{t,m}} }} \le \mathbb{E}{M_0} = V_0$, that is,
\begin{equation*}
\begin{split}
	\mathbb{E}\left[ {{e^{\lambda {t_{{\varsigma _{t,m}}}}}}{V_{{\varsigma _{t,m}}}}} \right] 
	\le& V_0 + C\mathbb{E}\left( {{e^{\lambda {\delta _{\max }}}}\int_0^{{t_{{\varsigma _{t,m}}}}} {{e^{\lambda s}}ds} } \right) \\
	\le& V_0 + C{e^{\lambda {\delta _{\max }}}}\int_0^t {{e^{\lambda s}}ds}  \\
	=& V_0 + \frac{{C{e^{\lambda {\delta _{\max }}}}}}{\lambda }({e^{\lambda t}} - 1)
\end{split}
\end{equation*}
Letting $m \to \infty$ and using Fatou's lemma \cite[Theorem 2.6.2, p.~187]{shiryaev2019probability} yields
\begin{equation*}
	\mathbb{E}\left[ {{e^{\lambda {\underline{t}}}}V_p(Y_{\underline{t}})} \right] \le V_0 + \frac{{C{e^{\lambda {\delta _{\max }}}}}}{\lambda }({e^{\lambda t}} - 1).
\end{equation*}
Since $\underline{t} \ge t- \delta_{\max}$, one sees
\begin{equation*}
\begin{split}
	\mathbb{E}\left[ V_p(Y_{\underline{t}}) \right] \le& \mathbb{E}\left[ {{e^{ - \lambda (t - {\delta _{\max }})}}{e^{\lambda \underline{t}}}V_p({Y_{\underline{t}}})} \right] \\
	\le& {e^{\lambda {\delta _{\max }}}}{e^{ - \lambda t}}\mathbb{E}\left[ {{e^{\lambda \underline{t}}}V_p(Y_{\underline{t}})} \right] \\
	\le& {e^{\lambda {\delta _{\max }}}}V_0 + \frac{{C{e^{2\lambda {\delta _{\max }}}}}}{\lambda } \\
	\le& C .
\end{split}
\end{equation*}
This together with \eqref{continuous} yields
\begin{equation*}
	\mathop {\sup }\limits_{t \ge 0} \mathbb{E}{\left| {Y(t)} \right|^p} = \mathop {\sup }\limits_{t \ge 0} \mathbb{E}{\left| {{Y_{\underline{t}}}} \right|^p} \le \mathop {\sup }\limits_{t \ge 0} \mathbb{E}\left[ V_p(Y_{\underline{t}}) \right] \le C.
\end{equation*}
This proves the first assertion.

On the other hand, for the partial timestep from $\underline{t}$ to $t$, from \eqref{continuous1} we have
\begin{equation*}
	\begin{split}
		{| {{\hat Y}(t)} |^2} =& {\left| {{Y_{\underline{t}}}} \right|^2} + 2\left\langle {{Y_{\underline{t}}},f({Y_{\underline{t}}})} \right\rangle (t - \underline{t}) + {\left| {f({Y_{\underline{t}}})} \right|^2}{(t - \underline{t})^2} + {\left| {g({Y_{\underline{t}}})\left( {W(t) - W(\underline{t})} \right)} \right|^2} \\
		&+ 2\left\langle {{Y_{\underline{t}}} + f({Y_{\underline{t}}})(t - \underline{t}),g({Y_{\underline{t}}})(W(t) - W(\underline{t}))} \right\rangle  .
	\end{split}
\end{equation*}
Similar to \eqref{bd1}-\eqref{bd4} one also has
\begin{equation}  \label{bd5}
\begin{split}
	{e^{\lambda t}}\mathbb{E}\left[ {{{V_p(\hat Y(t))}}|{{\cal F}_{\underline{t}}}} \right] - {e^{\lambda \underline{t}}}V_p(Y_{\underline{t}}) 
	\le C{e^{\lambda {\delta _{\max }}}}\int_{\underline{t}}^t {{e^{\lambda s}}ds}  \quad \text{a.s.}
\end{split}
\end{equation}
Summing \eqref{bd4} over multiple timesteps and then adding \eqref{bd5}, taking expectations on both sides yields
\begin{equation*}
\begin{split}
	{e^{\lambda t}}\mathbb{E}\left[ V_p(\hat{Y}(t)) \right] \le V_0 + C{e^{\lambda {\delta _{\max }}}}\int_0^t {{e^{\lambda s}}ds} . 
\end{split}
\end{equation*}
This implies
\begin{equation}  \label{bd7}
\begin{split}
	\mathbb{E}\left[ V_p(\hat{Y}(t)) \right] \le {e^{ - \lambda t}}V_0 + \frac{{C{e^{\lambda {\delta _{\max }}}}}}{\lambda } \le C.
\end{split}
\end{equation}
Hence, 
\begin{equation*}
\begin{split}
	\mathop {\sup }\limits_{t \ge 0} \mathbb{E}{\left| {\hat Y(t)} \right|^p} \le \mathop {\sup }\limits_{t \ge 0} \mathbb{E}\left[ V_p(\hat{Y}(t)) \right] 
	\le C .
\end{split}
\end{equation*}
The desired result follows.

\end{proof}

\section{Uniform-in-time strong convergence rate}  \label{sec6}
In this section, we establish the uniform-in-time estimate for the adaptive time-stepping EM scheme \eqref{scheme}. Unlike a finite-time estimate, the resulting error bound is independent of the terminal time. This uniform-in-time convergence theory plays an essential role in the subsequent convergence analysis of the numerical invariant measure.

\begin{assumption}  \label{A5.1}
There exists a pair of constants $p^* > 2$ and $\eta>0$ such that 
\begin{equation*}
	\left\langle {x - y,f(x) - f(y)} \right\rangle  + \frac{{{p^*} - 1}}{2}{\left| {g(x) - g(y)} \right|^2} \le  - \eta {\left| {x - y} \right|^2}.
\end{equation*}
for any $x, y \in \mathbb{R}^d$.
\end{assumption}

As Assumption \ref{A-Delta} was designed for finite-time analysis, based on Assumption \ref{A-stepbound}, we reformulate it below to ensure its compatibility with the uniform-in-time setting.
\begin{assumption}  \label{A6.2}
For each $\Delta \in (0, 1]$, the refined timestep function $\delta^\Delta: \mathbb{R}^d \to (0, \delta_{\max}]$ is continuous and satisfies
\begin{equation*}
	0 < {\delta ^\Delta }(x) < \min \left\{ {\Delta {\delta _{\max }},\delta (x)} \right\}, \quad  x \in \mathbb{R}^d.
\end{equation*}
\end{assumption}

Before establishing the uniform-in-time convergence rate, we present a preliminary result analogous to Lemma \ref{YYbar}. Since the proof follows the same line of argument, it is omitted here.

\begin{lemma}  \label{lem-6.3}
Assume that Assumption \ref{A4.1} holds with $p \ge l$, and Assumptions \ref{A-stepsize}, \ref{A-stepbound}, \ref{A5.1}, \ref{A6.2} and \eqref{f} are satisfied. Then, for any $q_2 \in [2, 2p/l]$, 
\begin{equation*}
	\mathop {\sup }\limits_{t \ge 0} \mathbb{E}{\left| {\hat Y(t) - Y(t)} \right|^{{q_2}}} \le C{\Delta ^{\frac{{{q_2}}}{2}}}.
\end{equation*}

\end{lemma}

\begin{theorem}  \label{th-a3}
Assume that Assumption \ref{A4.1} holds with $p \ge 3l-4$, and Assumptions \ref{A-stepsize}, \ref{A-stepbound}, \ref{A5.1}, \ref{A6.2} and \eqref{f} are satisfied.
Then, for any $q \in [2,p^ * ) \cap  \left[ {2,{{2p}}/{(3l - 4)}} \right]$,
\begin{equation*}
	\mathop {\sup }\limits_{t \ge 0} \mathbb{E}{\left| {X(t) - \hat Y(t)} \right|^q} \le {C}\Delta^{{q}/{2}}, \quad  \mathop {\sup }\limits_{t \ge 0} \mathbb{E}{\left| {X(t) - Y(t)} \right|^q} \le {C}\Delta^{{q}/{2}}.
\end{equation*}
\end{theorem}

\begin{proof}
For any $t \ge 0$, it follows from \eqref{equation} and \eqref{continuous2} that
\begin{equation*}
	X(t) -  \hat Y(t) = \int_0^t {\left[ {f(X(s)) - f(Y(s))} \right]ds}  + \int_0^t {\left[ {g(X(s)) - g(Y(s))} \right]dW(s)} .
\end{equation*}
For the given $ q\ge 2$, choose a constant $\kappa \in (0, q \eta)$ and use It\^o's formula to obtain
\begin{equation}  \label{con-rate2-1}
\begin{split}  
	&{e^{\left( {q\eta  - \kappa } \right)t}}\mathbb{E}{\left| {X(t) - \hat Y(t)} \right|^q} \\
	&\le \left( {q\eta  - \kappa } \right)\mathbb{E}\left[ {\int_0^t {{e^{\left( {q\eta  - \kappa } \right)s}}{{\left| {X(s) - \hat Y(s)} \right|}^q}ds} } \right] \\
	&\quad+ \frac{q}{2}\mathbb{E}\Big[ {\int_0^t {{e^{\left( {q\eta  - \kappa } \right)s}}{{\left| {X(s) - \hat Y(s)} \right|}^{q - 2}}\Big( {2\left\langle {X(s) - \hat Y(s),f(X(s)) - f(Y(s))} \right\rangle } } }   \\
	& \quad  + (q - 1){{\left| {g(X(s)) - g(Y(s))} \right|}^2} \Big) \Big].\end{split}
\end{equation}
Similar to \eqref{ac5}, applying Assumption \ref{A5.1} and the fact $q \in [2, p^*)$ one has
\begin{equation*}
\begin{split}
	&2\left\langle {X(s) -  \hat Y(s),f(X(s)) - f(Y(s))} \right\rangle  + (q - 1){\left| {g(X(s)) - g(Y(s))} \right|^2} \\
	&\le  2\left\langle {X(s) - \hat Y(s),f(X(s)) - f(\hat Y(s))} \right\rangle  + 2\left\langle {X(s) - \hat Y(s),f(\hat Y(s)) - f(Y(s))} \right\rangle  \\
	&\quad + (p^* - 1){\left| {g(X(s)) - g(\hat Y(s))} \right|^2} + \frac{{(p^* - 1)(q - 1)}}{{p^* - q}}{\left| {g(\hat Y(s)) - g(Y(s))} \right|^2} \\
	&\le -2\eta {\left| {X(s) - \hat Y(s)} \right|^2} + 2\left| {X(s) - \hat Y(s)} \right|\left| {f(\hat Y(s)) - f(Y(s))} \right| \\
	&\quad+ \frac{{(p^* - 1)(q - 1)}}{{p^* - q}}{\left| {g(\hat Y(s)) - g(Y(s))} \right|^2}.
\end{split}
\end{equation*}
Substituting the above into \eqref{con-rate2-1} yields
\begin{equation}  \label{ce4}
\begin{split}
	&{e^{\left( {q\eta  - \kappa } \right)t}}\mathbb{E}{\left| {X(t) - \hat Y(t)} \right|^q} \\
	&\le  - \kappa \mathbb{E}\left[ {\int_0^t {{e^{\left( {q\eta  - \kappa } \right)s}}{{\left| {X(s) - \hat Y(s)} \right|}^q}ds} } \right] \\
	&\quad+ q\mathbb{E}\left[ {\int_0^t {{e^{\left( {q\eta  - \kappa } \right)s}}{{\left| {X(s) - \hat Y(s)} \right|}^{q - 1}}\left| {f(\hat Y(s)) - f(Y(s))} \right|ds} } \right] \\
	&\quad+ \frac{{q({p^*} - 1)(q - 1)}}{{2({p^*} - q)}}\mathbb{E}\left[ {\int_0^t {{e^{\left( {q\eta  - \kappa } \right)s}}{{\left| {X(s) -\hat  Y(s)} \right|}^{q - 2}}{{\left| {g(\hat Y(s)) - g(Y(s))} \right|}^2}ds} } \right] .
\end{split}
\end{equation}
For any $\varepsilon > 0$, for any $s \in [0, t]$, using Young's inequality one has
\begin{equation}   \label{ce2}
\begin{split}
	{\left| {X(s) - \hat Y(s)} \right|^{q - 1}}\left| {f(\hat Y(s)) - f(Y(s))} \right| \le \frac{{q - 1}}{q}\varepsilon {\left| {X(s) - \hat Y(s)} \right|^q} + \frac{1}{{q{\varepsilon ^{q - 1}}}}{\left| {f(\hat Y(s)) - f(Y(s))} \right|^q}
\end{split}
\end{equation}
and 
\begin{equation} \label{ce3}
	{\left| {X(s) - \hat Y(s)} \right|^{q - 2}}{\left| {g(\hat Y(s)) - g(Y(s))} \right|^2} \le \frac{{q - 2}}{q}\varepsilon {\left| {X(s) - \hat Y(s)} \right|^q} + \frac{2}{{q{\varepsilon ^{\frac{{q - 2}}{2}}}}}{\left| {g(\hat Y(s)) - g(Y(s))} \right|^q} .
\end{equation}
Choose $\varepsilon > 0$ sufficiently small such that
\begin{equation}  \label{ce7}
	\left( {(q - 1) + \frac{{({p^*} - 1)(q - 1)(q - 2)}}{{2({p^*} - q)}}} \right)\varepsilon  \le \kappa  < q\eta .
\end{equation}
Substituting \eqref{ce2}, \eqref{ce3} and \eqref{ce7} into \eqref{ce4}, one has
\begin{equation}  \label{ce1}
\begin{split}
	&{e^{\left( {q\eta  - \kappa } \right)t}}\mathbb{E}{\left| {X(t) -  \hat Y(t)} \right|^q} \\
	&\le  - \left[ {\kappa  - (q - 1)\varepsilon  - \frac{{({p^*} - 1)(q - 1)(q - 2)}}{{2({p^*} - q)}}\varepsilon } \right]\mathbb{E}\left[ {\int_0^t {{e^{\left( {q\eta  - \kappa } \right)s}}{{\left| {X(s) - \hat Y(s)} \right|}^q}ds} } \right]  \\
	&\quad+ \frac{1}{{{\varepsilon ^{q - 1}}}}\mathbb{E}\left[ {\int_0^t {{e^{\left( {q\eta  - \kappa } \right)s}}{{\left| {f(\hat Y(s)) - f(Y(s))} \right|}^q}ds} } \right] \\
	&\quad+ \frac{{({p^*} - 1)(q - 1)}}{{({p^*} - q){\varepsilon ^{\frac{{q - 2}}{2}}}}}\mathbb{E}\left[ {\int_0^t {{e^{\left( {q\eta  - \kappa } \right)s}}{{\left| {g(\hat Y(s)) - g(Y(s))} \right|}^q}ds} } \right] \\
	&\le \frac{K_5^q}{{{\varepsilon ^{q - 1}}}}\mathbb{E}\left[ {\int_0^t {{e^{\left( {q\eta  - \kappa } \right)s}}{{\left( {1 + {{\left| {\hat Y(s)} \right|}^{l-2}} + {{\left| {Y(s)} \right|}^{l-2}}} \right)}^q}{{\left| {\hat Y(s) - Y(s)} \right|}^q}ds} } \right]  \\
	&\quad+ \frac{{({p^*} - 1)(q - 1)}}{{({p^*} - q){\varepsilon ^{\frac{{q - 2}}{2}}}}}\mathbb{E}\left[ {\int_0^t {{e^{\left( {q\eta  - \kappa } \right)s}}{{\left( {1 + {{\left| {\hat Y(s)} \right|}^{l-2}} + {{\left| {Y(s)} \right|}^{l-2}}} \right)}^{{q}/{2}}}{{\left| {\hat Y(s) - Y(s)} \right|}^q}ds} } \right] \\
	&\le C\int_0^t {{e^{\left( {q\eta  - \kappa } \right)s}}\mathbb{E}\left[ {{{\left( {1 + {{\left| {\hat Y(s)} \right|}^{l-2}} + {{\left| {Y(s)} \right|}^{l-2}}} \right)}^q}{{\left| { \hat Y(s) - Y(s)} \right|}^q}} \right]ds} .
\end{split}
\end{equation}
Applying the same technique of \eqref{lem-con-2}, using Lemma \ref{lem-6.3}, Theorem \ref{th-bounded} and the fact $q \in \left[ {2,2p/(3l - 4)} \right]$, one deduces
\begin{equation}  \label{ce6}
	\mathbb{E}\left[ {{{\left( {1 + {{\left| {\hat Y(s)} \right|}^{l-2}} + {{\left| {Y(s)} \right|}^{l-2}}} \right)}^q}{{\left| { \hat Y(s) - Y(s)} \right|}^q}} \right]  \le {C}\Delta^{{q}/{2}} . 
\end{equation} 
Inserting \eqref{ce6} into \eqref{ce1} yields
\begin{equation*}
	{e^{\left( {q\eta  - \kappa } \right)t}}\mathbb{E}{\left| {X(t) -  \hat Y(t)} \right|^q} \le C{\Delta ^{{q}/{2}}}\int_0^t {{e^{\left( {q\eta  - \kappa } \right)s}}ds} \le \frac{{C{e^{\left( {q\eta  - \kappa } \right)t}}}}{{q\eta  - \kappa }}{\Delta ^{{q}/{2}}},
\end{equation*} 
which implies
\begin{equation*}
	\mathbb{E}{\left| {X(t) -  \hat Y(t)} \right|^q} \le C\Delta^{{q}/{2}}, \quad \forall t \ge 0.
\end{equation*}
The first assertion is obtained as follows. Using the triangle inequality and Lemma \ref{lem-6.3} we obtain 
\begin{equation*}
	\mathbb{E}{\left| {X(t) -  Y(t)} \right|^q} \le {2^{q - 1}}\mathbb{E}{\left| {X(t) -  \hat Y(t)} \right|^q} + {2^{q - 1}}\mathbb{E}{\left| {\hat Y(t) - Y(t)} \right|^q} \le C\Delta^{{q}/{2}}, \quad \forall t\ge 0.
\end{equation*}
The proof is therefore complete.
\end{proof}

\section{Exponential stability in $p$th moment}  \label{sec7}
Since stability describes the dynamical behavior more precisely than boundedness, our goal in this section is to show the $p$th moment exponential stability of the adaptive time-stepping EM scheme \eqref{scheme}. We first impose the following assumption.

\begin{assumption}  \label{A6.1}
	There exists a pair of positive constants $p \ge 2$ and $\nu $ such that
	\begin{equation*}  
		\left\langle {x,f(x)} \right\rangle  + \frac{{p - 1}}{2}{\left| {g(x)} \right|^2} \le  - \nu {\left| x \right|^2},  \quad \forall x \in \mathbb{R}^d.
	\end{equation*}
\end{assumption}

\begin{lemma} (\cite[Theorem 4.4.4]{mao2008stochastic}) \label{lem-expon}
Under Assumption \ref{A6.1}, the solution $X(t)$ of SDE \eqref{equation} satisfies
\begin{equation*}
	\mathbb{E}{\left| {X(t)} \right|^p} \le {\left| {{x_0}} \right|^2}{e^{ - p\nu t}}, \quad \forall t\ge 0.
\end{equation*}
\end{lemma}

It was pointed out in \cite[p. 299]{higham2003exponential} that the result of Lemma \ref{lem-expon} forces $f(\mathbf{0}) = \mathbf{0}$ and $g(\mathbf{0}) = \mathbf{0}$. Based on this, we give the following assumption on the adaptive timesteps, which is slightly stronger than Assumption \ref{A-stepsize}.

\begin{assumption}  \label{A-stability}
Assume that the adaptive timestep function $\delta: \mathbb{R}^d \to (0, \delta_{\max}]$ is continuous, where $\delta_{\max} \in (0, \infty)$ is a positive constant. Moreover, there exists a constant $\gamma \in (0, 1/2)$ such that
\begin{equation*}
	\left| {f(x)} \right|{\left( {\delta (x)} \right)^\gamma } \le \left| x \right|,\quad \left| {g(x)} \right|^2{\left( {\delta (x)} \right)^\gamma } \le {\left| x \right|^2}, \quad \forall x \in \mathbb{R}^d.
\end{equation*}
\end{assumption}

\begin{remark}
The requirements $f(\mathbf{0}) = \mathbf{0}$ and $g(\mathbf{0}) = \mathbf{0}$ guarantee that the adaptive timesteps are well defined at $x = \mathbf{0}$.
Here we present an adaptive timestep function satisfying Assumption \ref{A-stability},
\begin{equation*}
	\delta(x)
	=
	\begin{cases}
		\displaystyle
		\delta_{\max}
		\wedge
		\left(\frac{\lvert x\rvert}{\lvert f(x)\rvert}\right)^{\frac{1}{\gamma}}
		\wedge
		\left(\frac{\lvert x\rvert^{2}}{\lvert g(x)\rvert^{2}}\right)^{\frac{1}{\gamma}},
		& x\neq \mathbf{0}, \\[1.2ex]
		\delta_{\max},
		& x=\mathbf{0}.
	\end{cases}
\end{equation*}

\end{remark}

\begin{theorem}  \label{Th-numer-stability}
Let \eqref{llc}, Assumptions \ref{A6.1} and \ref{A-stability} hold. Then, 
for any $\varepsilon \in (0,p\nu)$, there is a constant $\delta_*$ such that for any $0 <\delta_{\max} \le \delta_*$,
the continuous-time adaptive time-stepping EM solution \eqref{continuous2} has the property
\begin{equation*}
	\mathbb{E}{\left| {\hat Y(t)} \right|^p} \le {\left| {{x_0}} \right|^p}{e^{ - \left( {p\nu  - \varepsilon } \right)t}}, \quad \forall t\ge 0.
\end{equation*}
\end{theorem}
\begin{proof}
For any interger $n \ge 0$ and any $\eta \in (0, 1]$, from \eqref{scheme} one sees
\begin{equation} \label{gh11}
	{\left( {\eta  + {{\left| {{Y_{{t_{n + 1}}}}} \right|}^2}} \right)^{p/2}} = {\left( {\eta  + {{\left| {{Y_{{t_n}}}} \right|}^2}} \right)^{p/2}}{\left( {1 + \zeta _n } \right)^{p/2}}  ,
\end{equation}
where
\begin{equation*}
	\zeta _n  = \frac{{2\left\langle {{Y_{{t_n}}},f({Y_{{t_n}}})} \right\rangle \delta _n  + {{\left| {f({Y_{{t_n}}})} \right|}^2}{{\left( {\delta _n } \right)}^2} + {{\left| {g({Y_{{t_n}}})\Delta {W_n}} \right|}^2} + 2\left\langle {{Y_{{t_n}}} + f({Y_{{t_n}}})\delta _n ,g({Y_{{t_n}}})\Delta {W_n}} \right\rangle }}{{\eta  + {{\left| {{Y_{{t_n}}}} \right|}^2}}} .
\end{equation*}
For a given constant $p \ge 2$, choosing a non-negative integer $k$ such that $2k < p \le 2(k+1)$. By \cite[Lemma 3.3]{yang2018explicit}, we get
\begin{equation}  \label{gh7}
	\begin{split}
		&\mathbb{E}\left[ {{{\left( {\eta + {{\left| {{Y_{{t_{n + 1}}}}} \right|}^2}} \right)}^{p/2}}|{\mathcal{F}_{{t_n}}}} \right] \\
		\le& {\left( {\eta + {{\left| {{Y_{{t_n}}}} \right|}^2}} \right)^{p/2}}\left[ {1 + p/2\mathbb{E}(\zeta _n |{\mathcal{F}_{{t_n}}}) + \frac{{p(p - 2)}}{8}\mathbb{E}({{(\zeta _n )}^2}|{\mathcal{F}_{{t_n}}}) + \mathbb{E}({{(\zeta _n )}^3}{P_k}(\zeta _n )|{\mathcal{F}_{{t_n}}})} \right],
	\end{split}
\end{equation}
where $P_k(\cdot)$ represents a $k$th-order polynomial whose coefficients
depend only on $p$. By an argument analogous to $\mathbf{Step\;2}$ in Theorem \ref{th2.1}, applying Assumption \ref{A-stability}, one has
\begin{equation}  \label{gh8}
\begin{split}
	&\mathbb{E}\left( {\zeta _n |{\mathcal{F}_{{t_n}}}} \right) \le {\left( {\eta + {{\left| {{Y_{{t_n}}}} \right|}^2}} \right)^{ - 1}}\left[ {2\left\langle {{Y_{{t_n}}},f({Y_{{t_n}}})} \right\rangle  + {{\left| {g({Y_{{t_n}}})} \right|}^2}} \right]\delta _n  + C{\left( {\delta _n } \right)^{2(1 - \gamma )}} \quad \text{a.s.} \\
	&\mathbb{E}\left( {{{\left( {\zeta _n } \right)}^2}|{\mathcal{F}_{{t_n}}}} \right) \le 4{\left( {\eta  + {{\left| {{Y_{{t_n}}}} \right|}^2}} \right)^{ - 2}}{\left| {{{({Y_{{t_n}}})}^{\rm T}}g({Y_{{t_n}}})} \right|^2}{\delta _n}  + C{\left( {\delta _n } \right)^{2(1 - \gamma )}} \quad \text{a.s.}  \\
	&\mathbb{E}\left( {{{\left( {\zeta _n } \right)}^3}{P_k}(\zeta _n )|{\mathcal{F}_{{t_n}}}} \right) \le C{\left( {\delta _n } \right)^{2(1 - \gamma )}} \quad \text{a.s.}
\end{split}
\end{equation}
Inserting \eqref{gh8} into \eqref{gh7}, using Assumption \ref{A6.1} yields
\begin{equation}  \label{gh9}
\begin{split}
	&\mathbb{E}\left[ {{{\left( {\eta + {{\left| {{Y_{{t_{n + 1}}}}} \right|}^2}} \right)}^{p/2}}|{\mathcal{F}_{{t_n}}}} \right] \\
	&\le {\left( {\eta + {{\left| {{Y_{{t_n}}}} \right|}^2}} \right)^{p/2}}\Bigg[ {1 + C{{\left( {\delta _n } \right)}^{2(1 - \gamma )}}}   \\
	& \quad+ p{\delta _n}\frac{{\left( {\eta  + {{\left| {{Y_{{t_n}}}} \right|}^2}} \right)\left( {2\left\langle {{Y_{{t_n}}},f({Y_{{t_n}}})} \right\rangle  + {{\left| {g({Y_{{t_n}}})} \right|}^2}} \right) + (p - 2){{\left| {{{({Y_{{t_n}}})}^{\rm T}}g({Y_{{t_n}}})} \right|}^2}}}{{2{{\left( {\eta  + {{\left| {{Y_{{t_n}}}} \right|}^2}} \right)}^2}}}\Bigg]  \\
	&\le \left( {1 + C{{\left( {\delta _n } \right)}^{2(1 - \gamma )}}} \right){\left( {\eta  + {{\left| {{Y_{{t_n}}}} \right|}^2}} \right)^{p/2}} + \frac{p\delta _n}{2} {\left( {\eta  + {{\left| {{Y_{{t_n}}}} \right|}^2}} \right)^{p/2 - 1}}\left( {2\left\langle {{Y_{{t_n}}},f({Y_{{t_n}}})} \right\rangle  + (p - 1){{\left| {g({Y_{{t_n}}})} \right|}^2}} \right) \\
	&\le \left( {1 + C{{\left( {\delta _n } \right)}^{2(1 - \gamma )}}} \right){\left( {\eta  + {{\left| {{Y_{{t_n}}}} \right|}^2}} \right)^{p/2}} - p\nu \delta _n {\left| {{Y_{{t_n}}}} \right|^2}{\left( {\eta  + {{\left| {{Y_{{t_n}}}} \right|}^2}} \right)^{p/2 - 1}} \quad \text{a.s.}
\end{split}
\end{equation}
Due to the fact $\eta \in (0, 1]$, applying Theorem \ref{th-bounded} one has 
\begin{equation*}
	\begin{split}
		\mathbb{E}{\left( {\eta  + {{\left| {{Y_{{t_{n + 1}}}}} \right|}^2}} \right)^{p/2}} \le \mathbb{E}{\left( {1 + {{\left| {{Y_{{t_{n + 1}}}}} \right|}^2}} \right)^{p/2}} \le {2^{p/2 - 1}}\left( {1 + \mathop {\sup }\limits_{t \ge 0} \mathbb{E}{{\left| {Y(t)} \right|}^p}} \right) \le C .
	\end{split}
\end{equation*}
Hence, applying the dominated convergence theorem for conditional expectation (see e.g., \cite[p. 218]{shiryaev2019probability}) we get
\begin{equation*}
\begin{split}
	\mathbb{E}\left( {{{\left| {{Y_{{t_{n + 1}}}}} \right|}^p}|{\mathcal{F}_{{t_n}}}} \right) =& \mathop {\lim }\limits_{\eta  \to {0^ + }} \mathbb{E}\left[ {{{\left( {\eta  + {{\left| {{Y_{{t_{n + 1}}}}} \right|}^2}} \right)}^{p/2}}|{\mathcal{F}_{{t_n}}}} \right] \\
	\le& \left( {1 + C{{\left( {\delta _n } \right)}^{2(1 - \gamma )}}} \right){\left| {{Y_{{t_n}}}} \right|^p} - {p\nu }\delta _n {\left| {{Y_{{t_n}}}} \right|^p} \quad \text{a.s.}
\end{split}
\end{equation*}
For any $\varepsilon  \in (0,p\nu )$, choose $\delta_* > 0$ sufficiently small such that $C\delta _*^{1 - 2\gamma } < \varepsilon $ and $(p\nu - \varepsilon)\delta_* < 1$. For $0 < \delta_{\max} \le \delta_*$, one has
\begin{equation*}
\begin{split}
	\mathbb{E}\left( {{{\left| {{Y_{{t_{n + 1}}}}} \right|}^p}|{{\cal F}_{{t_n}}}} \right) \le \left[ {1 - \left( {p\nu  - \varepsilon } \right)\delta _n } \right]{\left| {{Y_{{t_n}}}} \right|^p} \quad \text{a.s.}
\end{split}
\end{equation*}
Multiplying both sides by ${e^{\left( {{p\nu } - \varepsilon } \right){t_{n + 1}}}}$, and using the fact $ 1 - x \le e^{-x}$ for any $x \ge 0$, we have
\begin{equation*} 
	{e^{\left( {{p\nu } - \varepsilon } \right){t_{n + 1}}}}\mathbb{E}\left( {{{\left| {{Y_{{t_{n + 1}}}}} \right|}^p}|{\mathcal{F}_{{t_n}}}} \right) \le {e^{\left( {{p\nu } - \varepsilon } \right){t_{n + 1}}}}{e^{ - \left( {{p\nu } - \varepsilon } \right)\delta _n }}{\left| {{Y_{{t_n}}}} \right|^p} = {e^{\left( {{p\nu } - \varepsilon } \right){t_n}}}{\left| {{Y_{{t_n}}}} \right|^p} \quad \text{a.s.}
\end{equation*}
Since $t_{n+1}$ is $\mathcal{F}_{t_n}$-measurable, the above inequality then gives
\begin{equation}  \label{gh10}
	\mathbb{E}\left( {{e^{\left( {{p\nu } - \varepsilon } \right){t_{n + 1}}}}{{\left| {{Y_{{t_{n + 1}}}}} \right|}^p}|{\mathcal{F}_{{t_n}}}} \right) \le {e^{\left( {{p\nu } - \varepsilon } \right){t_n}}}{\left| {{Y_{{t_n}}}} \right|^p} \quad \text{a.s.}
\end{equation}
Hence, 
\begin{equation*}
	{M_n} := {e^{\left( {p\nu  - \varepsilon } \right){t_n}}}{\left| {{Y_{{t_n}}}} \right|^p}
\end{equation*}
is a supermartingale with respect to ${\{ {\mathcal{F}_{{t_n}}}\} _{n \ge 0}}$. For any $t > 0$ and $m \in \mathbb{N}$, define the bounded stopping time ${\vartheta  _{t,m}} = {n_t} \wedge m$. By the optional stopping theorem \cite[Theorem 7.14, p.~189]{klebaner2012introduction}, one has $\mathbb{E}{M_{{\vartheta _{t,m}} }} \le \mathbb{E}{M_0} = |x_0|^p$, that is,
\begin{equation*}
	\mathbb{E}\left( {{e^{\left( {p\nu  - \varepsilon } \right){t_{{\varsigma _{t,m}}}}}}{{\left| {{Y_{{t_{{\varsigma _{t,m}}}}}}} \right|}^p}} \right) \le {\left| {{x_0}} \right|^p}.
\end{equation*}
Letting $m \to \infty$ and using Fatou's lemma yields
\begin{equation}  \label{gh14}
	\mathbb{E}\left( {{e^{\left( {p\nu  - \varepsilon } \right)\underline{t}}}{{\left| {{Y_{\underline{t}}}} \right|}^p}} \right) \le {\left| {{x_0}} \right|^p} .
\end{equation}

On the other hand, for the partial timestep from $\underline{t}$ to $t$, similar to \eqref{gh11}-\eqref{gh10} we get
\begin{equation}  \label{gh13}
	\mathbb{E}\left( {{e^{\left( {{p\nu } - \varepsilon } \right)t}}{{\left| {\hat Y(t)} \right|}^p}|{\mathcal{F}_{{\underline{t}}}}} \right) \le {e^{\left( {{p\nu } - \varepsilon } \right){\underline{t}}}}{\left| {{Y_{{\underline{t}}}}} \right|^p} \quad \text{a.s.}
\end{equation}
Taking expectations on both sides of \eqref{gh13}, and  using \eqref{gh14} yields
\begin{equation*}
	{e^{\left( {p\nu  - \varepsilon } \right)t}}\mathbb{E}{\left| {\hat Y(t)} \right|^p} \le {\left| {{x_0}} \right|^p}, \quad  t \ge 0.
\end{equation*}
The desired assertion follows.
\end{proof}

\section{Ergodicity and approximation of numerical invariant measure}  \label{sec8}
In this section, we first show that the adaptive time-stepping EM scheme  admits a unique numerical invariant probability measure. Under a slightly stronger assumption, we further establish its polynomial ergodicity in the $L^q$-Wasserstein distance. Finally, by combining the ergodicity of the underlying SDE with the uniform-in-time strong convergence rate of the numerical scheme, we derive a convergence rate for the numerical invariant measure toward the invariant measure of the SDE \eqref{equation}.

We first give the definition of an $n$-step transition kernel.

\begin{definition}
Let $(E, \mathcal{B}(E))$ be a measurable space and let $\Psi=\{\Psi_n\}_{n\ge0}$ be a time-homogeneous Markov chain on the state space $E$. The one-step transition kernel of $\Psi$ is the mapping $P_\Psi : E \times \mathcal{B}(E) \to [0,1]$ defined by
\begin{equation*}
	P_\Psi(x,A):= \mathbb P(\Psi_1 \in A \mid \Psi_0 = x),
	\quad x \in E, \quad A\in \mathcal{B}(E).
\end{equation*}
For each fixed $x\in E$, $P_\Psi(x,\cdot)$ is a probability measure on $(E,\mathcal{B}(E))$, and for each fixed $A \in \mathcal{B}(E)$, $P_\Psi(\cdot,A)$ is $\mathcal{B}(E)$-measurable.
For $n \ge 1$, the $n$-step transition kernel is defined by
	\begin{equation*}
		{P}_\Psi ^n(x,A): = \mathbb{P}({\Psi _{n}} \in A|{\Psi _0} = x) \quad x \in E, \quad A \in \mathcal{B}(E).
	\end{equation*}
	with $P_\Psi^1=P_\Psi$. 
Moreover, the family $\{P_\Psi^n\}_{n\ge1}$ satisfies the Chapman--Kolmogorov equations
\begin{equation*}
	P_\Psi^{m+n}(x,A) = \int_E P_\Psi^m(x,dy) \, P_\Psi^n(y,A), 
	\quad x\in E, \; A\in \mathcal{B}(E), \; m,n\ge 1.
\end{equation*}
\end{definition}

Before showing that numerical solution process
${\{ {Y_{{t_n}}}\} _{n \ge 0}}$	forms a time-homogeneous Markov chain, we first give an essential lemma.

\begin{lemma} (\cite[Lemma 1.9.2, p. 87]{mao2008stochastic})
Let $h(x, \omega)$ be a scalar bounded measurable random function of $x$, independent of $\mathcal{F}_s$. Let $\zeta $ be an $\mathcal{F}_s$-measurable random variable. Then $\mathbb{E}(h(\zeta ,\omega )|{{\cal F}_s}) = H(\zeta)$, where $H(x) = \mathbb{E}h(x, \omega)$.
\end{lemma}

\begin{lemma}  \label{markov}  (Time-homogeneous Markov chain) Let Assumptions \ref{A2.1}, \ref{A-stepsize} and \ref{A4.1} hold. The numerical solution process
${\{ {Y_{{t_n}}}\} _{n \ge 0}}$ generated by the adaptive time-stepping EM scheme \eqref{scheme} is a time-homogeneous Markov chain on $(\mathbb{R}^d, \mathcal{B}(\mathbb{R}^d))$. More precisely, for every $n \ge 0$, $Y_0 \in \mathbb{R}^d$ and $A \in \mathcal{B}(\mathbb{R}^d)$,
\begin{equation*}
\begin{split}
	\mathbb{P}\left( {{Y_{{t_{n + 1}}}} \in A|{\mathcal{F}_{{t_n}}}} \right) = \mathbb{P}\left( {{Y_{{t_{n + 1}}}} \in A|{Y_{{t_n}}}} \right)
\end{split}
\end{equation*}
and
\begin{equation*}
\begin{split}
	\mathbb{P}\left( {{Y_{{t_{n + 1}}}} \in A|{Y_{{t_n}}} = x } \right) = \mathbb{P}\left( {{Y_{{t_1}}} \in A|{Y_{{0}}} = x } \right), \quad x \in \mathbb{R}^d.
\end{split}
\end{equation*}
\end{lemma}

\begin{proof}	
By Remark \ref{remark-steps}, the Brownian increment $\Delta W_n$ is $\mathcal{F}_{t_n}$-conditionally normally distributed. Define $\xi_{n+1} = \Delta W_n /\sqrt{\delta_n}$. Then, for every bounded Borel measurable function $\psi : \mathbb{R}^m \to \mathbb{R}$,
\begin{equation*}
	\mathbb{E}\left( {\psi ({\xi _{n + 1}})|{\mathcal{F}_{{t_n}}}} \right) = \int_{{\mathbb{R}^m}} {\psi (z){\gamma _m}(dz)} ,\quad \text{a.s.}
\end{equation*}
where $\gamma_m$ denotes the standard Gaussian probability measure $N(0, I_m)$ on $\mathbb{R}^m$. Hence, ${\xi _{n + 1}} \sim N(0,{I_m})$ and $\xi_{n+1}$ is independent of $\mathcal{F}_{t_n}$. 
Define 
\begin{equation*}
	\varphi (x,z) = x + f(x)\delta (x) + g(x)\sqrt {\delta (x)} z, \quad x \in \mathbb{R}^d, \quad z \in \mathbb{R}^m.
\end{equation*}
By the local Lipschitz condition \eqref{llc} and Assumption \ref{A-stepsize}, the mapping $\varphi:\mathbb{R}^d \times \mathbb{R}^m \to \mathbb{R}^d$ is jointly continuous. Then,
for every $n \ge 0$ and $A \in \mathcal{B}(\mathbb{R}^d)$,
	\begin{equation}  \label{op1}
		\begin{split}
			&\mathbb{P}\left( {{Y_{{t_{n + 1}}}} \in A|{\mathcal{F}_{{t_n}}}} \right) \\
			=& \mathbb{P}\left( {\varphi ({Y_{{t_n}}},{\xi _{n + 1}}) \in A|{\mathcal{F}_{{t_n}}}} \right) = \mathbb{E}\left[ {{\mathcal{I}_A}(\varphi ({Y_{{t_n}}},{\xi _{n + 1}}))|{\mathcal{F}_{{t_n}}}} \right] \\
			=& \mathbb{E}\left[ {{\mathcal{I}_A}(\varphi (x,{\xi _{n + 1}}))} \right]{|_{x = {Y_{{t_n}}}}} = \mathbb{P}\left( {\varphi (x,{\xi _{n + 1}}) \in A} \right){|_{x = {Y_{{t_n}}}}} = \mathbb{P}\left( {\varphi ({Y_{{t_n}}},{\xi _{n + 1}}) \in A|{Y_{{t_n}}}} \right) \\
			=& \mathbb{P}\left( {{Y_{{t_{n + 1}}}} \in A|{Y_{{t_n}}}} \right) ,
		\end{split}
	\end{equation}
which proves the Markov property. For any $x \in \mathbb{R}^d$ and $A \in \mathcal{B}(\mathbb{R}^d)$, since $\xi_{n+1}$ is independent of $\mathcal{F}_{t_n}$, we have
\begin{equation}  \label{pq9}
	\mathbb{P}\left( {{Y_{{t_{n + 1}}}} \in A|{Y_{{t_n}}} = x } \right) = \mathbb{P}\left( {\varphi ({Y_{{t_n}}},{\xi _{n + 1}}) \in A|{Y_{{t_n}}} = x } \right) = \mathbb{P}\left( {\varphi (x ,{\xi _{n + 1}}) \in A} \right).
\end{equation}
Similarly,
\begin{equation}  \label{pq10}
	\mathbb{P}\left( {{Y_{{t_1}}} \in A|{Y_0} = x } \right) = \mathbb{P}\left( {\varphi ({Y_{{0}}},{\xi _1}) \in A|{Y_0} = x } \right) = \mathbb{P}\left( {\varphi (x ,{\xi _1}) \in A} \right) .
\end{equation}
Since $\xi_{n + 1}$ and $\xi_1$ have the same distribution, from \eqref{pq9} and \eqref{pq10} one has
\begin{equation*}
	\mathbb{P}\left( {{Y_{{t_{n + 1}}}} \in A|{Y_{{t_n}}} = x } \right) = \mathbb{P}\left( {{Y_{{t_1}}} \in A|{Y_0} = x } \right).
\end{equation*}
This proves that $\{Y_{t_n}\}_{n\ge 0}$ is a time-homogeneous Markov chain.
\end{proof}

\begin{lemma} \label{feller} 
	Let \eqref{llc}, Assumptions \ref{A-stepsize} and \ref{A4.1} hold. The numerical solution process ${\{ {Y_{{t_n}}}\} _{n \ge 0}}$ generated by adaptive time-stepping EM scheme \eqref{scheme} is Feller.
\end{lemma}

\begin{proof}
Define $\xi_{n+1} = \Delta W_n /\sqrt{\delta_n}$. By Remark \ref{remark-steps}, ${\xi _{n + 1}} \sim N(0,{I_m})$ and $\xi_{n+1}$ is independent of $\mathcal{F}_{t_n}$. Define 
\begin{equation*}
	\varphi (x,z) = x + f(x)\delta (x) + g(x)\sqrt {\delta (x)} z, \quad x \in \mathbb{R}^d, \quad z \in \mathbb{R}^m.
\end{equation*}
By the local Lipschitz condition \eqref{llc} and Assumption \ref{A-stepsize}, the mapping $\varphi:\mathbb{R}^d \times \mathbb{R}^m \to \mathbb{R}^d$ is jointly continuous. Let
$\phi \in \mathcal{C}_b(\mathbb{R}^d)$, by Lemma \ref{markov}, for every $x \in \mathbb{R}^d$,
\begin{equation*}
	\begin{split}
		\mathbb{E}\left[ {\phi  ({Y_{{t_{n + 1}}}})|{Y_{{t_n}}} = x} \right] = \mathbb{E}\left[ {\phi  (\varphi ({Y_{{t_n}}},{\xi _{n + 1}}))|{Y_{{t_n}}} = x} \right] = \mathbb{E}\left[ {\phi  (\varphi (x,{\xi _{n + 1}}))} \right] .
	\end{split}
\end{equation*} 
Let ${\{ {x_k}\} _{k \ge 1}} \subset {\mathbb{R}^d} $ satisfy $x_k \to x$ as $k \to \infty$.
By the joint continuity of $\varphi$ and the continuity of $\phi$,
	\begin{equation*}
		\mathop {\lim }\limits_{x_k \to x} \phi  (\varphi (x_k,{\xi _{n + 1}})) = \phi  (\varphi (x,{\xi _{n + 1}})).
	\end{equation*}
Furthermore, since $\phi  \in {\mathcal{C}_b}({\mathbb{R}^d})$, there exists a constant $M > 0$ such that
\begin{equation*}
	\left| {\phi  (\varphi (x_k,{\xi _{n + 1}}))} \right| \le M, \quad k \ge 1.
\end{equation*}
Hence, by Lebesgue's dominated convergence theorem (see e.g., \cite[Theorem 3, p. 187]{shiryaev2019probability}),
\begin{equation*}
	\mathop {\lim }\limits_{k \to \infty} \mathbb{E}\left[ {\phi  (\varphi (x_k,{\xi _{n + 1}}))} \right] = \mathbb{E}\left[ {\mathop {\lim }\limits_{k \to \infty} \phi  (\varphi (x_k,{\xi _{n + 1}}))} \right] =  \mathbb{E}\left[ {\phi  (\varphi (x,{\xi _{n + 1}}))} \right],
\end{equation*}
Therefore, the mapping 
	\begin{equation*}
		x \mapsto \mathbb{E}\left[ {\phi ({Y_{{t_{n + 1}}}})|{Y_{{t_n}}} = x} \right]
	\end{equation*}
	is continuous on $\mathbb{R}^d$.
	Since $\phi\in\mathcal C_b(\mathbb R^d)$ is arbitrary, the numerical
	solution process $\{Y_{t_n}\}_{n\ge0}$ is Feller.
	The proof is therefore complete.
\end{proof}

\subsection{Existence and uniqueness of the numerical invariant measure}

We first recall an ergodicity theorem for a time-homogeneous Markov chain $\Psi  = {\{ {\Psi _n}\} _{n \ge 0}}$.

\begin{lemma} \label{meyn14} (\cite[Theorem 14.0.1]{meyn2009markov})
Suppose that $\Psi  = {\{ {\Psi _n}\} _{n \ge 0}}$ is a $\psi$-irreducible and aperiodic Markov chain on $(E, \mathcal{B}(E))$. Assume that there exist a petite set $K \in \mathcal{B}(E)$, a constant $b < \infty$, and a non-negative function $V:E \to [0, \infty)$ such that
\begin{equation*}  
	\int_E {V(y){{P}_\Psi }(x,dy)} -V(x) \le -1 + b{\mathcal{I}_K}(x), \quad x \in E.
\end{equation*}
Then the chain is positive recurrent and admits an invariant measure $\pi$. Moreover,
\begin{equation*}
	\mathop {\lim }\limits_{n \to \infty }  {\left\| {{P_\Psi^n}(x, \cdot ) - \pi} \right\|_{{\rm TV}}} = 0, \quad  x\in E.
\end{equation*}

\end{lemma}

Before establishing the existence and uniqueness of the numerical invariant measure of the adaptive time-stepping EM scheme \eqref{scheme}, we require an assumption on how quickly the adaptive timestep function $\delta(x)$ can approach zero as $|x| \to \infty$.

\begin{assumption}  \label{A-step}
There exists a constant $\ell > 0$ such that the adaptive timestep function satisfies the inequality
\begin{equation*}
	\delta (x) \ge \ell {\left( {1 + {{\left| x \right|}^2}} \right)^{ - p/2}}, \quad x \in \mathbb{R}^d,
\end{equation*}
where $p$ is given in Assumption \ref{A4.1}.
\end{assumption}

Note that Assumption \ref{A-step} is a weak condition that provides a polynomial lower bound for the adaptive timestep. Similar conditions can also be seen in \cite[Theorem 11]{lemaire2007an} and \cite[Assumption 5]{fang2020adaptive}.

\begin{proposition} \label{th-a2} 
Let \eqref{llc}, Assumptions \ref{A-stepsize}, \ref{A4.1}, \ref{A-stepbound} and \ref{A-step} hold. Assume further that the diffusion coefficient $g(\cdot)$ is uniformly nondegenerate, i.e., there exists a constant $\lambda > 0$ such that
\begin{equation*}
	g(x)g^{\rm T}(x) \ge \lambda I_{d}, \quad x \in \mathbb{R}^{d}.
\end{equation*}
Then the adaptive time-stepping EM numerical solution ${\{ {Y_{{t_n}}}\}  _{n \ge 0}}$ defines a $\mu^{Leb}$-irreducible, aperiodic Markov chain. Moreover, the chain is positive recurrent and admits a unique invariant probability measure $\mu^\Delta$. Furthermore,
\begin{equation*}
	\mathop {\lim }\limits_{n \to \infty }{\left\| {{P_Y^n}(x_0, \cdot ) - \mu^\Delta} \right\|_{{\rm TV}}} = 0.
\end{equation*}
for every $x_0 \in \mathbb{R}^d$.
\end{proposition}

\begin{proof}
To establish the ergodicity of the adaptive time-stepping EM scheme \eqref{scheme}, we first establish irreducibility and aperiodicity, and then verify the Foster-Lyapunov drift condition required by Lemma \ref{meyn14}. Since the proof is rather technical, we divide it into two steps.

\textbf{Step 1}: 
We first show that Markov chain ${\{ {Y_{{t_n}}}\} _{n \ge 0}}$ is ${\mu ^{Leb}}$-irreducible and aperiodic, where ${\mu ^{Leb}}$ denotes the Lebesgue measure on $\mathbb{R}^d$. Fix $x \in \mathbb{R}^d$ and suppose that $Y_0 = x$. It follows from scheme \eqref{scheme} that
\begin{equation*}
	{Y_{{t_{ 1}}}} = x + f(x)\delta (x) + g(x)\left( {W({t_{ 1}}) - W({t_0})} \right).
\end{equation*}
By Remark \ref{remark-steps}, the Brownian increment $W(t_{1}) - W(t_0)$ has the same distribution as $\sqrt {\delta (x)} {\xi _{ 1}}$, where ${\xi _1} \sim N(0,{I_{m}})$ is a standard Gaussian random vector independent of $\mathcal{F}_{t_0}$. 
Hence, we have ${Y_{{t_1}}} \sim N\left( {m(x),\Sigma (x)} \right)$,
where $ m(x) = x+f(x)\delta(x)$ and $\Sigma(x) = \delta(x)g(x)g^{\rm T}(x)$. 
Since $g(\cdot)$ is uniformly nondegenerate and $\delta(\cdot) > 0$, $\Sigma (x)  \ge \lambda \delta (x){I_m} > 0$. Thus, $\Sigma(x)$ is positive definite for every $x \in \mathbb{R}^d$. Hence,
the one-step transition kernel $P_Y(x, \cdot)$ is a nondegenerate Gaussian measure on $\mathbb{R}^d$, and admits the strictly positive density 
\begin{equation*}
	{p_{m(x),\Sigma (x)}}(y) = \frac{{\exp \left( { - \frac{1}{2}{{(y - m(x))}^{\rm{T}}}\Sigma {{(x)}^{ - 1}}(y - m(x))} \right)}}{{\sqrt {{{(2\pi )}^d}\det (\Sigma (x))} }}, \quad y \in \mathbb{R}^d,
\end{equation*}
with respect to the Lebesgue measure $\mu^{Leb}$. Therefore, $P_Y(x,\cdot)$ is equivalent to the Lebesgue measure $\mu^{Leb}$ on $\mathbb{R}^d$ \cite[p. 87]{liu2025numerical}. In particular, for every $A \in \mathcal{B}(\mathbb{R}^d)$ satisfying $\mu^{Leb}(A) > 0$,
\begin{equation*}
	{P_Y}(x,A) = \int_A {{p_{m(x),\Sigma (x)}}(y){\mu ^{Leb}}(dy) > 0} , \quad x\in \mathbb{R}^d.
\end{equation*}
Consequently, using Lemma \ref{markov}, the Markov chain ${\{ {Y_{{t_n}}}\} _{n \ge 0}}$ is ${\mu ^{Leb}}$-irreducible (see e.g., \cite[Section 4.2.1, p. 82]{meyn2009markov}).

Next, we prove that Markov chain ${\{ {Y_{{t_n}}}\} _{n \ge 0}}$ is aperiodic. Suppose, to the contrary, that it has period $j \ge 2$. By \cite[Theorem 5.4.4, p. 113]{meyn2009markov}, there exist pairwise disjoint sets $D_1, ..., D_j \in \mathcal{B}(\mathbb{R}^d)$ such that
\begin{equation*}
	\begin{split}
		{{P}_Y}(x,{D_{i + 1}}) = 1, \quad x \in D_i, \quad i=1,2,...,j-1, \quad \text{and} \quad
		{{P}_Y}(x,{D_1}) = 1,\quad  x \in D_j,
	\end{split} 
\end{equation*}
and ${\mu ^{Leb}}\left( {{\mathbb{R}^d}{\rm{\backslash }}\bigcup\nolimits_{i = 1}^j {{D_i}} } \right) = 0$.
On the one hand, there exists ${i_0} \in \{ 1,2,...,j\} $ such that $\mu^{Leb}(D_{i_0}) > 0$. 
By the strict positivity of the one-step transition density, we get
\begin{equation}	\label{pq1}
	{P}_Y(x,{D_{i_0}}) > 0 , \quad x \in D_{i_0}.
\end{equation}
On the other hand, if ${i_0} \in \{ 1,2,...,j-1\} $, then 
\begin{equation*}
	0 \le {{P}_Y}(x,{D_{{i_0}}}) \le {{P}_Y}(x,D_{{i_0} + 1}^c) = 1 - {{P}_Y}(x,{D_{{i_0} + 1}}) = 0, \quad x \in D_{i_0}.
\end{equation*}
If $i_0 = j$, then
\begin{equation*}
	0 \le {{P}_Y}(x,{D_{{i_0}}}) \le {{P}_Y}(x,D_1^c) = 1 - {{P}_Y}(x,{D_1}) = 0, \quad x \in D_{i_0}.
\end{equation*}
Thus, ${{P}_Y}(x,{D_{{i_0}}}) = 0$ holds for all $x \in D_{i_0}$, which contradicts \eqref{pq1}. Therefore, ${\{ {Y_{{t_n}}}\} _{n \ge 0}}$ is aperiodic.

\textbf{Step 2}:
It follows from \eqref{bd9} with $n = 1$ that
\begin{equation*}  
	\begin{split}
		\mathbb{E}\left[ {{\left( {1 + {{\left| {{Y_{{t_{ 1}}}}} \right|}^2}} \right)^{p/2}}|{{\cal F}_{{t_0}}}} \right] 
		\le \left( {1 - \frac{{p{K_6}}}{2}{\delta _0}} \right){\left( {1 + {{\left| {{Y_{{t_0}}}} \right|}^2}} \right)^{p/2}} + C{\delta _0}  \quad \text{a.s.}
	\end{split}
\end{equation*} 
Let $V_p(x) = {\left( {1 + {{\left| x \right|}^2}} \right)^{p/2}}$ for any $x \in \mathbb{R}^d$. By the Markov property of ${\{ {Y_{{t_n}}}\} _{n \ge 0}}$, 
\begin{equation*}
	\begin{split}
		\mathbb{E}\left[ {V_p({Y_{{t_{1}}}})|{Y_{{t_0}}} = x} \right] \le& \left( {1 - \frac{pK_6}{2} \delta (x)} \right)V_p(x) + C\delta (x) \\
		=& \left( {1 - \frac{{pK_6 }}{4}}\delta (x) \right)V_p(x) + \delta (x)\left( {C - \frac{pK_6 }{4}V_p(x)} \right), \quad  x \in \mathbb{R}^d.
	\end{split}
\end{equation*}
If $V_p(x) > 4C/(pK_6) $, using Assumption \ref{A-step} we have
\begin{equation} 	\label{pq2}
	\mathbb{E}\left[ {V_p({Y_{{t_{ 1}}}})|{Y_{{t_0}}} = x} \right] \le \left( {1 - \frac{{pK_6 }}{4}}\delta (x) \right)V_p(x) \le V_p (x) - \frac{pK_6\ell}{4}.
\end{equation}
On the other hand, if $V_p(x) \le 4C/(pK_6) $, then Assumptions \ref{A-stepbound} and \ref{A-step} imply
\begin{equation}	\label{pq3}
	\mathbb{E}\left[ {V_p({Y_{{t_{ 1}}}})|{Y_{{t_0}}} = x} \right] \le \left( {1 - \frac{{pK_6 }}{4}} \delta (x) \right)V_p(x) + C{\delta _{\max }} \le V_p (x) - \frac{pK_6\ell}{4} + C\delta_{\max}.
\end{equation}
Combining \eqref{pq2} and \eqref{pq3} we obtain
\begin{equation}  \label{pq5}
	\mathbb{E}\left[ {V_p({Y_{{t_{ 1}}}})|{Y_{{t_0}}} = x} \right] \le V_p(x) - \frac{pK_6\ell}{4} + C{\delta _{\max }}{\mathcal{I}_K}(x), \quad x\in \mathbb{R}^d.
\end{equation}
where $K = \left\{ {x \in {\mathbb{R}^d}:V_p(x) \le \max \{ 4C/(pK_6), 1\} } \right\}$. Since $V_p(\cdot)$ is continuous and radially unbounded, the set $K$ is compact in $\mathbb{R}^d$.
Moreover, the chain ${\{ {Y_{{t_n}}}\} _{n \ge 0}}$ is Feller  and $\mu^{Leb}$-irreducible, and $supp \mu^{Leb} = \mathbb{R}^d$ has non-empty interior; hence, every compact set is petite \cite[Proposition 6.2.8, p. 133]{meyn2009markov}. 
In particular, $K$ is a petite set.
Let 
\begin{equation*}
	c_1 := {pK_6\ell}/{4}, \quad {{\tilde V}_p}(x) := \frac{{{V_p}(x)}}{c_1 }, \quad \tilde b := \frac{{C{\delta _{\max }}}}{c_1 }.
\end{equation*}
Then, \eqref{pq5} becomes
\begin{equation*}
	\mathbb{E}\left[ {{{\tilde V}_p}({Y_{{t_{ 1}}}})|{Y_{{t_0}}} = x} \right] \le {{\tilde V}_p}(x) - 1 + \tilde b{{\cal I}_K}(x) .
\end{equation*}
By Lemma \ref{meyn14}, the Markov chain ${\{ {Y_{{t_n}}}\} _{n \ge 0}}$ is positive recurrent and admits an invariant probability measure $\mu^\Delta$. Moreover, for every initial value $x \in \mathbb{R}^d$,
\begin{equation}  \label{pq13}
	{\left\| {{P_Y^n}(x, \cdot ) - {\mu ^\Delta }} \right\|_{{\rm TV}}} \to 0, \quad n \to \infty.
\end{equation}
Finally, $\mu^\Delta$ is unique. Indeed, if $\nu^{\Delta}$ is another invariant probability measure of $P_Y$, then for any $n \ge 1$ and $A \in \mathcal{B}(\mathbb{R}^d)$,
\begin{equation*}
\begin{split}
	\left| {\nu^{\Delta} P_Y^n(A) - {\mu ^\Delta }(A)} \right| =& \left| {\int_{{\mathbb{R}^d}} {P_Y^n(x,A)\nu^{\Delta} (dx)}  - \int_{{\mathbb{R}^d}} {{\mu ^\Delta }(A)\nu^{\Delta} (dx)} } \right| \\
	\le& \int_{{\mathbb{R}^d}} {\left| {P_Y^n(x,A) - {\mu ^\Delta }(A)} \right|\nu^{\Delta} (dx)} \\
	\le& \int_{{\mathbb{R}^d}} {\mathop {\sup }\limits_{B \in \mathcal{B}({\mathbb{R}^d})} \left| {P_Y^n(x,B) - {\mu ^\Delta }(B)} \right|\nu^{\Delta} (dx)} \\
	\le& \int_{{\mathbb{R}^d}} {{{\left\| {P_Y^n(x, \cdot ) - {\mu ^\Delta }} \right\|}_{{\rm TV}}}\nu^{\Delta} (dx)} .
\end{split}
\end{equation*}
Taking the supresum on both sides, by \eqref{pq13} and ${{{\left\| {P_Y^n(x, \cdot ) - {\mu ^\Delta }} \right\|}_{{\rm TV}}}} \le 2$, using the dominated convergence theorem we have
\begin{equation*}
\begin{split}
	{\left\| {\nu^{\Delta}  - {\mu ^\Delta }} \right\|_{{\rm TV}}} = {\left\| {\nu^{\Delta} P_Y^n - {\mu ^\Delta }} \right\|_{{\rm TV}}} \le \int_{{\mathbb{R}^d}} {{{\left\| {P_Y^n(x, \cdot ) - {\mu ^\Delta }} \right\|}_{{\rm TV}}}\nu^{\Delta} (dx)}  \to 0.
\end{split}
\end{equation*}
Hence $\nu^{\Delta}=\mu^\Delta$, proving the uniqueness of the invariant probability measure.
The proof is therefore complete.
\end{proof}

The following theorem shows the ergodicity of adaptive time-stepping scheme \eqref{scheme} in the $L^q$-Wasserstein distance with $q \in [1,p]$.

\begin{theorem}   \label{th-a4}
Assume that all the conditions of Proposition \ref{th-a2} hold. Let $\mu^\Delta$ be the unique invariant probability measure of the adaptive time-stepping scheme \eqref{scheme}. Then, for any $q \in [1,p)$, where $p$ is given in Assumption \ref{A4.1}, 
\begin{equation*}
	\mathop {\lim }\limits_{n \to \infty } {{\cal W}_q}({\cal L}({Y_{{t_n}}}),{\mu ^\Delta }) = 0,
\end{equation*}
where $\mathcal{L}({Y_{{t_n}}})$ denotes the law of $Y_{t_n}$.
\end{theorem}

\begin{proof}
Set $\mu_n = \mathcal{L}(Y_{t_n})$. By Theorem \ref{th-bounded},
\begin{equation}  \label{ym1}
	\mathop {\sup }\limits_{n \ge 0} \int_{{\mathbb{R}^d}} {{{\left| x \right|}^p}{\mu _n}(dx)}  = \mathop {\sup }\limits_{n \ge 0} \mathbb{E}{\left| {{Y_{{t_n}}}} \right|^p} \le C.
\end{equation}
For $R > 0$, define 
\begin{equation*}
	{\varphi _R}(x) = {\left| x \right|^p} \wedge R, \quad x \in \mathbb{R}^d.
\end{equation*}
Then, for each fixed $x \in \mathbb{R}^d$, $\varphi_R(x)$ is nondecreasing with respect to $R$. It follows from \eqref{aa2} and Theorem \ref{th-a2} that
\begin{equation*}
\begin{split}
	&\left| {\int_{{\mathbb{R}^d}} {{\varphi _R}(x){\mu _n}(dx)}  - \int_{{\mathbb{R}^d}} {{\varphi _R}(x){\mu ^\Delta }(dx)} } \right|\\
	&\le R\mathop {\sup }\limits_{\left| f \right| \le 1} \left| {\int_{{\mathbb{R}^d}} {f(x){\mu _n}(dx)}  - \int_{{\mathbb{R}^d}} {f(x){\mu ^\Delta }(dx)} } \right|  \\
	&= 2R{\left\| {{\mu _n} - {\mu ^\Delta }} \right\|_{{\rm TV}}} \to 0  \quad n \to \infty.
\end{split}
\end{equation*}
This together with \eqref{ym1} yields
\begin{equation*}
	\int_{{\mathbb{R}^d}} {{\varphi _R}(x){\mu ^\Delta }(dx)}  = \mathop {\lim }\limits_{n \to \infty } \int_{{\mathbb{R}^d}} {{\varphi _R}(x){\mu _n}(dx)} \le \mathop {\sup }\limits_{n \ge 0} \int_{{\mathbb{R}^d}} {{{\left| x \right|}^p}{\mu _n}(dx)} \le C .
\end{equation*}
Taking the limit $R \to \infty$ on both sides and using the monotone convergence theorem yields
\begin{equation}  \label{ym5}
	\int_{{\mathbb{R}^d}} {{{\left| x \right|}^p}{\mu ^\Delta }(dx)}  = \mathop {\lim }\limits_{R \to \infty } \int_{{\mathbb{R}^d}} {{\varphi _R}(x){\mu ^\Delta }(dx)}  \le C.
\end{equation}
Then, by \eqref{ym1} and \eqref{ym5},
\begin{equation}  \label{ym4}
	\int_{{\mathbb{R}^d}} {{{\left| x \right|}^p}\left| {{\mu _n} - {\mu ^\Delta }} \right|(dx)}  \le \int_{{\mathbb{R}^d}} {{{\left| x \right|}^p}{\mu _n}(dx)}  + \int_{{\mathbb{R}^d}} {{{\left| x \right|}^p}{\mu ^\Delta }(dx)}  \le 2C.
\end{equation}
Fix $q \in [1, p)$, it follows from \cite[Theorem 6.15]{villani2009optimal} that 
\begin{equation}  \label{ym2}
	{\mathcal{W}_q}({\mu _n},{\mu ^\Delta }) \le {2^{{1}/{q'}}}{\left( {\int_{{\mathbb{R}^d}} {{{\left| x \right|}^q}\left| {{\mu _n} - {\mu ^\Delta }} \right|(dx)} } \right)^{{1}/{q}}}, 
\end{equation}
where $1/q + 1/q' = 1$.
Using H\"older's inequality, \eqref{aa1} and \eqref{ym4}, we have
\begin{equation*}  
\begin{split}
	\int_{{\mathbb{R}^d}} {{{\left| x \right|}^q}\left| {{\mu _n} - {\mu ^\Delta }} \right|(dx)}  \le& {\left( {\int_{{\mathbb{R}^d}} {{{\left| x \right|}^p}\left| {{\mu _n} - {\mu ^\Delta }} \right|(dx)} } \right)^{q/p}}{\left( {\left| {{\mu _n} - {\mu ^\Delta }} \right|({\mathbb{R}^d})} \right)^{1 - q/p}} \\
	=& {2^{1 - q/p}}{\left( {\int_{{\mathbb{R}^d}} {{{\left| x \right|}^p}\left| {{\mu _n} - {\mu ^\Delta }} \right|(dx)} } \right)^{q/p}}\left\| {\mu_n  - \mu^\Delta } \right\|_{\rm TV}^{1 - q/p} .
\end{split}
\end{equation*}
Therefore, substituting this into \eqref{ym2}, using Proposition \ref{th-a2}, we conclude that
\begin{equation}  \label{ym6}
	{\mathcal{W}_q}({\mu _n},{\mu ^\Delta }) \le C\left\| {{\mu _n} - {\mu ^\Delta }} \right\|_{{\rm TV}}^{\frac{1}{q} - \frac{1}{p}} \to 0 \quad n \to \infty.
\end{equation}
This completes the proof.
\end{proof}

\subsection{Polynomial rate of convergence to the numerical invariant measure}

We first introduce a lemma for polynomial convergence of a Markov chain to its invariant probability measure in the total variation distance.

\begin{lemma} \label{aap2004} (\cite[Proposition 2.5]{douc2004practical})
	Suppose that $\Psi  = {\{ {\Psi _n}\} _{n \ge 0}}$ is a $\psi$-irreducible and aperiodic Markov chain on $(E, \mathcal{B}(E))$.
	Assume that there exist a petite set $K \in \mathcal{B}(E)$, a constant $b \in [0, \infty)$, and a measurable function $V:E \to [1, \infty)$, and a concave, nondecreasing, differentiable function $\phi: [1,\infty) \to \mathbb{R}_+ $ satisfying ${\lim _{t \to \infty }}\phi '(t) = 0$ such that
	\begin{equation*}  
		{P_\Psi }V(x)  +\phi(V(x)) \le V(x) + b{\mathcal{I}_K}(x), \quad x \in E.
	\end{equation*}
	where 
	\begin{equation*}
		{P_\Psi }V(x) = \int_E {V(y){P_\Psi }(x,dy)}. 
	\end{equation*}
	Define
	\begin{equation*}
		{H_\phi }(v) = \int_1^v {\frac{1}{{\phi (x)}}dx}, \quad {r_\phi }(n) = (H_\phi ^{ - 1})'(n) = \phi (H_\phi ^{ - 1}(n)).
	\end{equation*}
	Then, the chain $\Psi$ admits an invariant probability measure $\pi$, and
	\begin{equation*}
		\mathop {\lim }\limits_{n \to \infty } {r_\phi }(n){\left\| {P_\Psi ^n({x}, \cdot ) - \pi } \right\|_{{\rm{TV}}}} = 0, \quad x \in E.
	\end{equation*}
\end{lemma}

\begin{proposition}  \label{th-a6}
Assume that all the conditions of Proposition \ref{th-a2} hold, except that Assumption \ref{A-step} is replaced by the stronger condition
\begin{equation}  \label{delta1}
	\delta (x) \ge \ell {\left( {1 + {{\left| x \right|}^2}} \right)^{ - p\theta /2}}, \quad x\in \mathbb{R}^d,
\end{equation}
for some constants $\ell > 0$ and $\theta \in (0,1)$. 
Then the adaptive time-stepping EM numerical solution ${\{ {Y_{{t_n}}}\}  _{n \ge 0}}$ defines a $\mu^{Leb}$-irreducible, aperiodic Markov chain. Moreover, the chain ${\{ {Y_{{t_n}}}\} _{n \ge 0}}$ admits a unique invariant probability measure $\mu^\Delta$, and satisfies
\begin{equation*}
	\mathop {\lim }\limits_{n \to \infty } {n^{(1 - \theta )/\theta}}{\left\| {P_Y^n({x_0}, \cdot ) - {\mu ^\Delta }} \right\|_{{\rm{TV}}}} = 0, \quad x_0 \in \mathbb{R}^d.
\end{equation*}
\end{proposition}

\begin{proof}
Let $V_p(x) = {\left( {1 + {{\left| x \right|}^2}} \right)^{p/2}}$ for any $x \in \mathbb{R}^d$. Then, 
\begin{equation*}
	\delta (x) \ge \ell {V_p}{(x)^{ - \theta }} \ge \ell {V_p}{(x)^{ - 1}} = \ell {\left( {1 + {{\left| x \right|}^2}} \right)^{ - p/2}}, \quad x \in \mathbb{R}^d.
\end{equation*}
Thus, Assumption \ref{A-step} is satisfied. The $\mu^{Leb}$-irreducibility and aperiodicity of the chain ${\{ {Y_{{t_n}}}\} _{n \ge 0}}$, as well as the existence and uniqueness of its invariant probability measure $\mu^\Delta$, therefore follow from Theorem \ref{th-a2}. It remains to establish the polynomial convergence rate.

By the Lyapunov estimate established in the proof of Theorem \ref{th-a2}, 
\begin{equation*}
	\begin{split}
		\mathbb{E}\left[ {V_p({Y_{{t_{1}}}})|{Y_{{t_0}}} = x} \right] \le \left( {1 - \frac{{pK_6 }}{4}}\delta (x) \right)V_p(x) + \delta (x)\left( {C - \frac{pK_6 }{4}V_p(x)} \right), \quad  x \in \mathbb{R}^d.
	\end{split}
\end{equation*}
On the one hand, if $V_p(x) > 4C/(pK_6) $, using \eqref{delta1} we have
\begin{equation}  \label{jm1}	
	\mathbb{E}\left[ {{V_p}({Y_{{t_{1}}}})|{Y_{{t_0}}} = x} \right] \le {V_p}(x) - \frac{{p{K_6}\ell }}{4}{\left( {{V_p}(x)} \right)^{1 - \theta }}.
\end{equation}
On the other hand, if $V_p(x) \le 4C/(pK_6) $, then Assumption \ref{A-stepbound} and \eqref{delta1} imply
\begin{equation}	\label{jm3}
	\mathbb{E}\left[ {V_p({Y_{{t_{1 }}}})|{Y_{{t_0}}} = x} \right] \le V_p (x) - \frac{{p{K_6}\ell }}{4}{\left( {{V_p}(x)} \right)^{1 - \theta }} + C\delta_{\max}.
\end{equation}
Combining \eqref{jm1} and \eqref{jm3} we obtain
\begin{equation}  \label{jm5}
	\mathbb{E}\left[ {V_p({Y_{{t_{ 1}}}})|{Y_{{t_0}}} = x} \right] \le V_p(x) - \frac{{p{K_6}\ell }}{4}{\left( {{V_p}(x)} \right)^{1 - \theta }} + C{\delta _{\max }}{\mathcal{I}_K}(x), \quad x\in \mathbb{R}^d.
\end{equation}
where $K = \left\{ {x \in {\mathbb{R}^d}:V_p(x) \le \max \{4C/(pK_6), 1\} } \right\}$. As shown in the proof of Theorem \ref{th-a2}, $K$ is a petite set.
Define 
\begin{equation*}
	c_0 =\frac{{p{K_6}\ell }}{4}, \quad b = C\delta_{\max}, \quad \phi (x) = c_0{x^{1 - \theta }}.
\end{equation*}
Then \eqref{jm5} becomes
\begin{equation*}
	\mathbb{E}\left[ {{V_p}({Y_{{t_{n + 1}}}})|{Y_{{t_n}}} = x} \right] + \phi ({V_p}(x)) \le {V_p}(x) + b{{\cal I}_K}(x), \quad x \in \mathbb{R}^d.
\end{equation*}
Moreover, one computes that $\phi(\cdot)$ is positive, nondecreasing, differentiable, and concave. In addition, ${\lim _{x \to \infty }}\phi '(x) = 0$. Moreover, we compute
\begin{equation*}
\begin{split}
	\quad {H_\phi }(x) = \frac{{{x^\theta } - 1}}{c_0 \theta }, \quad H_\phi ^{ - 1}(x) = {\left( {1+c_0 \theta x} \right)^{\frac{1}{\theta }}}, \quad 
	 {r_\phi }(n) = c_0{\left( {1+c_0 \theta n} \right)^{\frac{{1 - \theta }}{\theta }}}.
\end{split}
\end{equation*}
Lemma \ref{aap2004} then implies that the chain ${\{ {Y_{{t_n}}}\} _{n \ge 0}}$ admits an invariant probability measure $\pi$ such that 
\begin{equation*}
	\mathop {\lim }\limits_{n \to \infty } {r_\phi }(n){\left\| {P_Y^n({x_0}, \cdot ) - \pi} \right\|_{{\rm{TV}}}} = 0 .
\end{equation*}
Since the chain ${\{ {Y_{{t_n}}}\} _{n \ge 0}}$ admits a unique invariant probability measure $\mu^\Delta$, it follows that $ \pi = \mu^\Delta$.
One sees
\begin{equation*}
	\mathop {\lim }\limits_{n \to \infty } \frac{{{r_\phi }(n)}}{{{n^{(1 - \theta )/\theta }}}} = {c_0}{({c_0}\theta )^{(1 - \theta )/\theta }} \le C,
\end{equation*}
this implies
\begin{equation*}
	\mathop {\lim }\limits_{n \to \infty } {n^{ (1 - \theta )/\theta}}{\left\| {P_Y^n({x_0}, \cdot ) - {\mu ^\Delta }} \right\|_{{\rm{TV}}}} = 0, \quad x \in \mathbb{R}^d.
\end{equation*}
The proof is therefore complete.
\end{proof}

The following theorem establishes the polynomial ergodicity of the adaptive time-stepping scheme \eqref{scheme} in $L^q$-Wasserstein distance with $q \in [1, p]$.

\begin{theorem}  \label{th-a7}
Under the assumptions of Proposition \ref{th-a6}, let $\mu^\Delta$ be the unique invariant probability measure of the adaptive time-stepping scheme
$\{Y_{t_n}\}_{n\ge0}$. Then, for any $q \in [1,p)$, where $p$ is given in Assumption \ref{A4.1}, 
\begin{equation*}
	\mathop {\lim }\limits_{n \to \infty } {n^{\beta '}}{\mathcal{W}_q}(\mathcal{L}({Y_{{t_n}}}),{\mu ^\Delta }) = 0 .
\end{equation*}
where $\beta' = (1-\theta)(p-q)/(\theta pq)$ and $\mathcal{L}({Y_{{t_n}}})$ denotes the law of $Y_{t_n}$.	
\end{theorem}

\begin{proof}
Set $\mu_n  = \mathcal{L}(Y_{t_n})$.
It follows from \eqref{ym6} that
\begin{equation*}  
	{\mathcal{W}_q}({\mu _n},{\mu ^\Delta }) \le C\left\| {{\mu _n} - {\mu ^\Delta }} \right\|_{{\rm TV}}^{\frac{1}{q} - \frac{1}{p}}.
\end{equation*}
Let
\begin{equation*}
	\beta = \frac{1-\theta}{\theta}, \quad \gamma = \frac{1}{q} - \frac{1}{p}>0, \quad \beta' = \beta \gamma.
\end{equation*}
By Theorem \ref{th-a6},
\begin{equation*}
	{n^\beta }{\left\| {{\mu _n} - {\mu ^\Delta }} \right\|_{\rm TV}} \to 0 \quad n \to \infty .
\end{equation*}
Then, we have
\begin{equation*}
	{n^{\beta' }}{\mathcal{W}_q}(\mathcal{L}({Y_{{t_n}}}),{\mu ^\Delta }) \le C{\left( {{n^\beta }{{\left\| {{\mu _n} - {\mu ^\Delta }} \right\|}_{\rm TV}}} \right)^\gamma } \to 0 \quad n \to \infty
\end{equation*}
This completes the proof.
\end{proof}

\subsection{Convergence rate of the numerical invariant measure}
In this subsection, we first consider the existence and uniqueness of the invariant probability measure for the SDE \eqref{equation}. We then prove that the numerical invariant measure converges to the exact invariant measure in $L_q$-Wasserstein distance.

\begin{lemma} (\cite[Theorem 3.1]{wang2018distribution}) \label{lem-a1}
	Let \eqref{llc}, Assumptions \ref{A4.1} and \ref{A5.1} hold. Then the SDE \eqref{equation} admits a unique invariant probability measure $\mu \in \mathcal{P}_p (\mathbb{R}^d)$, where $p$ is given in Assumption \ref{A4.1}. Furthermore, for any $\nu  \in {{\cal P}_p}({\mathbb{R}^d}) \cap {{\cal P}_{{p^*}}}(\mathbb{R}^d)$, there exists a constant $C > 0$ such that
	\begin{equation*}
		{\mathcal{W}_{p^*}}(\nu{P_t} ,\mu ) \le C{e^{ - \eta t}}{\mathcal{W}_{p^*}}(\nu ,\mu ), \quad t \ge 0.
	\end{equation*}
\end{lemma}

\begin{theorem}  \label{th-a5}
Let \eqref{llc}, \eqref{delta1}, Assumptions \ref{A-stepsize}, \ref{A4.1}, \ref{A-stepbound}, \ref{A5.1} and \ref{A6.2} hold with $p \ge 3l-4$. For any $q \in [2,p^ * ) \cap \left[ {2,{{2p}}/(3l - 4)} \right]$, the exact and numerical invariant measures $\mu$ and $\mu^{\Delta}$ have the property
\begin{equation*}
	{\mathcal{W}_q}(\mu ,{\mu^\Delta }) \le C{\Delta ^{\frac{1}{2}}} .
\end{equation*}
\end{theorem}

\begin{proof}
By the triangle inequality, for any $n \ge 0$ we arrive at
\begin{equation}  \label{pq8}
	{\mathcal{W}_q}(\mu ,{\mu^\Delta }) \le {\mathcal{W}_q}(\mathcal{L}({X_{{t_n}}}), \mu) + {\mathcal{W}_q}(\mathcal{L}({X_{{t_n}}}),\mathcal{L}({Y_{{t_n}}})) + {\mathcal{W}_q}(\mathcal{L}({Y_{{t_n}}}),{\mu^\Delta }).
\end{equation}
From Lemma \ref{lem-a1} we see
\begin{equation} \label{pq7}
	 \mathop {\lim }\limits_{n \to \infty } {{\cal W}_q}({\cal L}(X({t_n})),\mu ) \le \mathop {\lim }\limits_{n \to \infty } {{\cal W}_{{p^*}}}({\cal L}(X({t_n})),\mu ) = 0 .
\end{equation}
Using Theorem \ref{th-a7} one has
\begin{equation}  \label{pq0}
	\mathop {\lim }\limits_{n \to \infty } {\mathcal{W}_q}(\mathcal{L}({Y_{{t_n}}}),{\mu^\Delta }) = 0.
\end{equation}
Hence, taking $n \to \infty$ on both sides of \eqref{pq8}, and applying \eqref{pq7}, \eqref{pq0} and Theorem \ref{th-a3} yields
\begin{equation*}
	{\mathcal{W}_q}(\mu ,{\mu^\Delta }) \le {\left( {\mathbb{E}{{\left| {X({t_n}) - {Y_{{t_n}}}} \right|}^q}} \right)^{\frac{1}{q}}} \le C{\Delta ^{\frac{1}{2}}} .
\end{equation*}
The proof is therefore complete.

\end{proof}

\section{Numerical experiments}  \label{sec9}
We compare the performance of the adaptive time-stepping EM scheme (denoted by AEM) with several existing methods, including the adaptive LMS algorithm \cite{lamba2007an}, Lemaire's method \cite{lemaire2007an}, and three widely used fixed-step schemes: the backward EM \cite{kloeden1992numerical, hu1996semi}, tamed EM \cite{hutzenthaler2012strong, sabanis2016euler}, truncated EM \cite{li2019explicit} schemes. For the latter, we take as the uniform stepsize $\delta_{\mathrm{mean}}$ the average of all time steps $\delta_n^{(m)}$ across all $M$ sample paths, defined by
\begin{equation}
	{\delta_{\mathrm{mean}}} = \frac{1}{M}\sum\limits_{m = 1}^M {\frac{1}{{{N_T^{(m)}}}}\sum\limits_{n = 1}^{{N_T^{(m)}}} {\delta _n^{(m)}} }. 
\end{equation}
This implies that AEM and the fixed-step numerical schemes require the same number of timesteps to reach the terminal time $T$. Consequently, their computational costs are comparable, which allows for a fair comparison of their numerical performance \cite{kelly2018adaptive}. In the following numerical experiments, error bounds are measured using the root mean square error (RMSE). Moreover, the expectations are approximated by computing averages over $M$ sample paths. Numerical experiments are implemented in  MATLAB R2023a on DESKTOP-QRSSSIC (11VECTO1WW).

\subsection{Stiff SDEs}
Stiff SDEs, whose solutions involve components with widely differing time scales, have been extensively studied in various applications, including stochastic chemical kinetics \cite{rathinam2003stiffness,ilie2012adaptive} and stochastic reaction-diffusion systems \cite{ta2015an}.

\begin{example}  \label{ex-c} {\rm 
Consider the scalar stiff SDE
\begin{equation}	\label{ex-stiff}
	dX(t) = [(X(t) - 1)(5 - X(t))(X(t) -20)]dt + 10X(t)dW(t), \quad X(0) = x_0.
\end{equation}
A direct computation implies that Assumptions \ref{A2.1} and \ref{A3.1} hold with $p, p^* >2$ and $l = 4$. 
The adaptive timestep function can be chosen as
\begin{equation}  \label{exam-2}
	{\delta ^\Delta }(x) = \left[ {\left( {\frac{{1 \vee {{\left| x \right|}^2}}}{{1 \vee {{\left| {f(x)} \right|}^2}}}} \right) \wedge {{\left( {\frac{{1 \vee {{\left| x \right|}^2}}}{{1 \vee {{\left| {g(x)} \right|}^2}}}} \right)}^2}} \right]\Delta , \quad x \in \mathbb{R}^d. 
\end{equation}
One computes that the adaptive timestep function \eqref{exam-2} satisfies Assumptions \ref{A-stepsize} and \ref{A-Delta}.
By Theorem \ref{rate1}, AEM \eqref{scheme} admits a $1/2$-order of convergence rate.

Figure \ref{z1} displays a sample path generated by the AEM scheme for SDE \eqref{ex-stiff}, together with the corresponding adaptive timestep at each step. One observes that the timestep reduces in regions where the path exhibits rapid oscillations, whereas larger timesteps are used in relatively smooth regions.
Figure \ref{z2} plots the convergence rate of AEM scheme with different timesteps. It is evident from Figure \ref{z2} that the convergence rate of AEM scheme is close to $1/2$, which is consistent with our theoretical result.

Next, we are going to compare the performance of AEM with fixed-step schemes with uniform stepsize $\delta_{\mathrm{mean}}$.
Since the exact solution of SDE \eqref{ex-stiff} cannot be given explicitly, we take the solution generated by the same method with a timestep that is two times smaller as the reference solution. 
Figure \ref{z3} shows the comparisons of AEM and fixed-step backward EM, tamed EM, truncated EM schemes for RMSE against the average number of timesteps and CPU time. For a given RMSE = $0.1$, it can be observed in Figure \ref{z3} (Left) that the AEM costs slightly fewer timesteps than backward EM, tamed EM and truncated EM schemes, and that in Figure \ref{z3} (right), AEM costs less CPU time than backward EM, tamed EM and truncated EM schemes. These show the better performance of AEM for the stiff SDE \eqref{ex-stiff}.
	
\begin{figure}[H]
	\centering
	\includegraphics[width=0.7\textwidth]{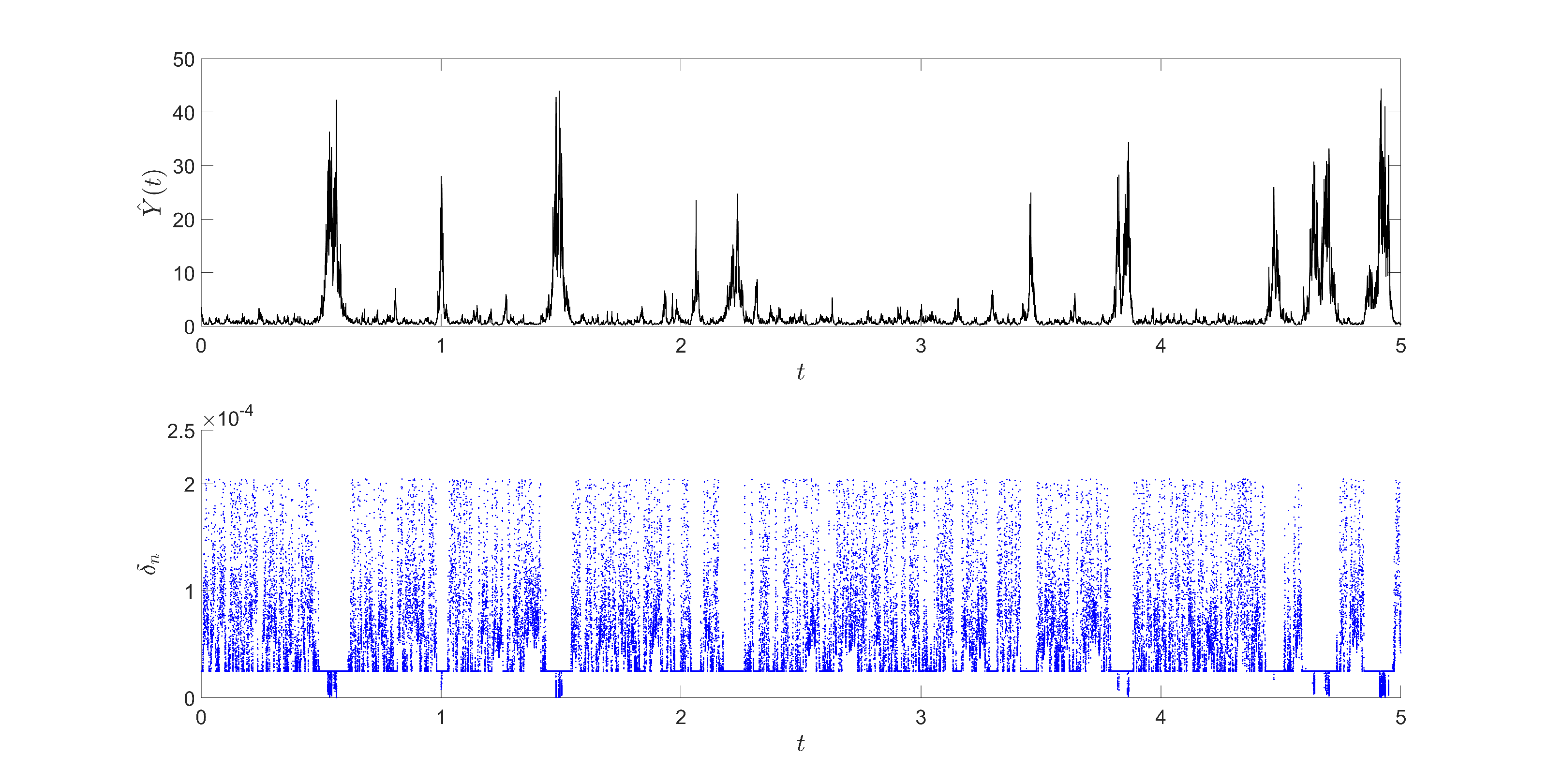}
	\caption{One sample path generated by AEM scheme and the corresponding adaptive timesteps for SDE \eqref{ex-stiff}, based on initial value $x_0 = 3$ and $\Delta = 2^{-2}$.}
	\label{z1}
\end{figure}
	
\begin{figure}[H]
	\centering
	\includegraphics[width=0.7\textwidth]{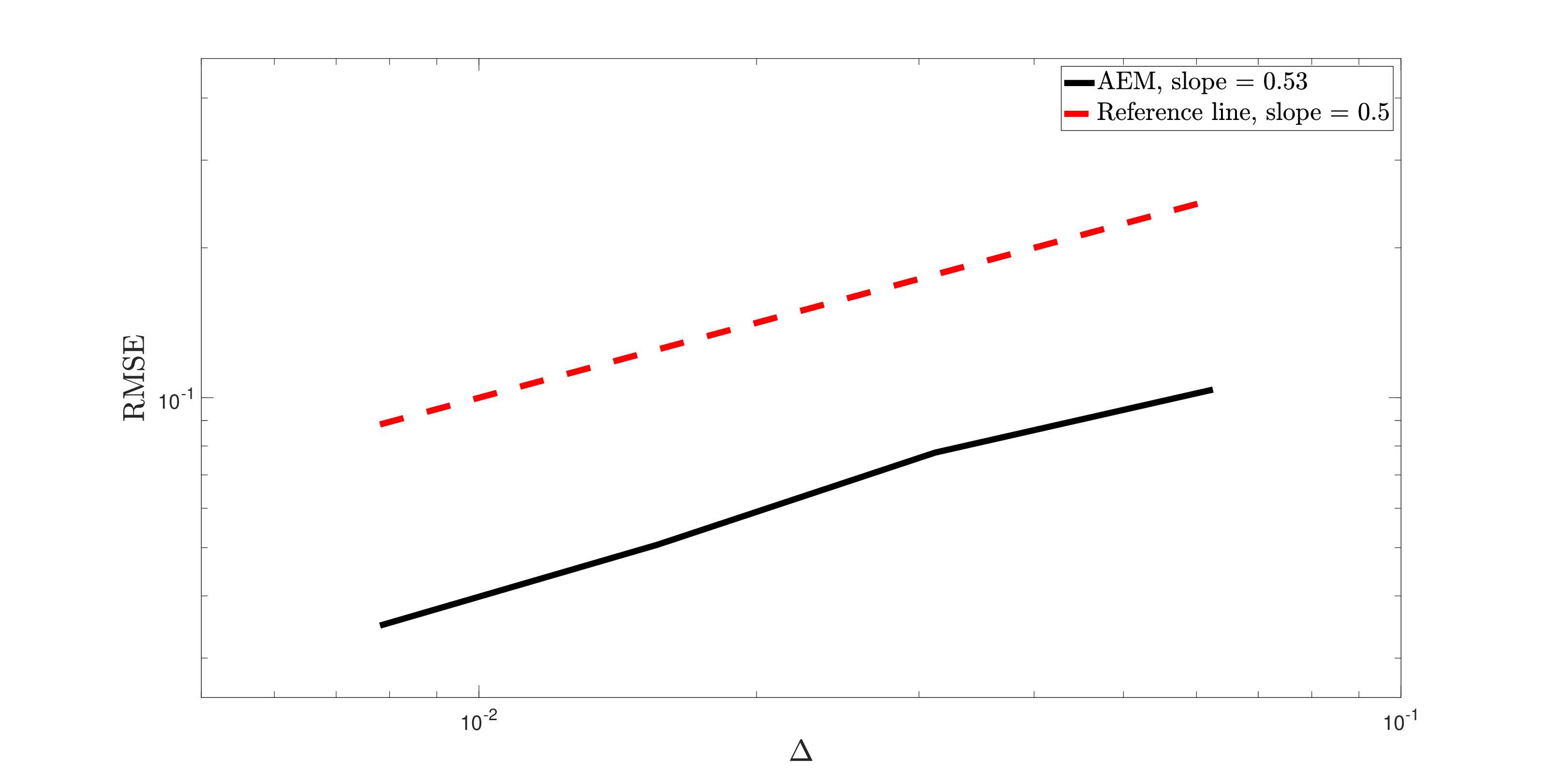}
	\caption{The strong convergence rate of AEM for SDE \eqref{ex-stiff}, based on initial value $x_0 = 3$, $M = 500$ sample paths and $\Delta  \in \{ {2^{ - 7}},{2^{ - 6}},{2^{ - 5}},{2^{ - 4}}\} $.}
	\label{z2}
\end{figure}

\begin{figure}[H]
	\centering
	\begin{subfigure}{0.48\textwidth}
		\centering
		\includegraphics[width=\textwidth]{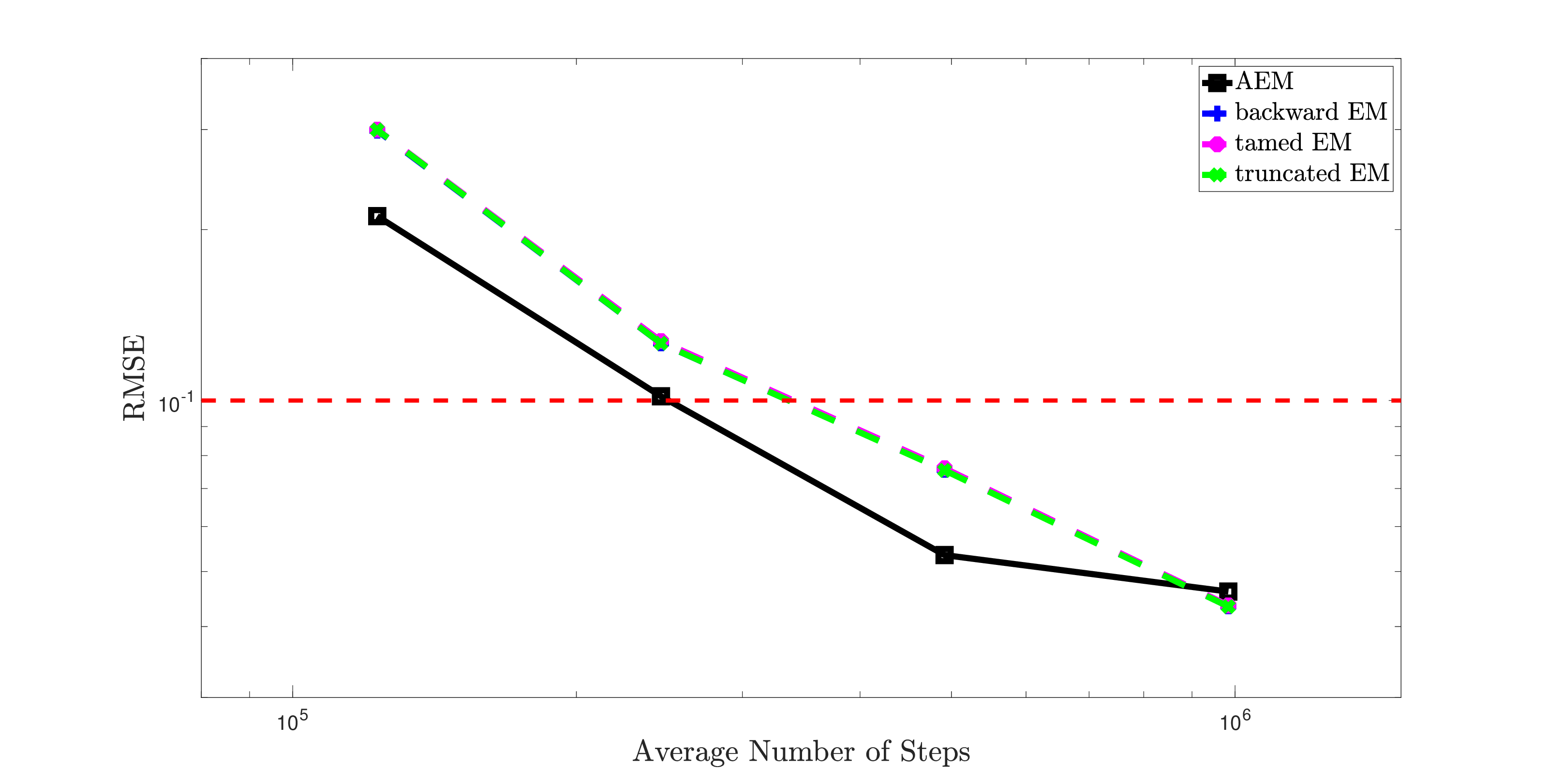}
	\end{subfigure}
	\hspace{0.01\textwidth}
	\begin{subfigure}{0.48\textwidth}
		\centering
		\includegraphics[width=\textwidth]{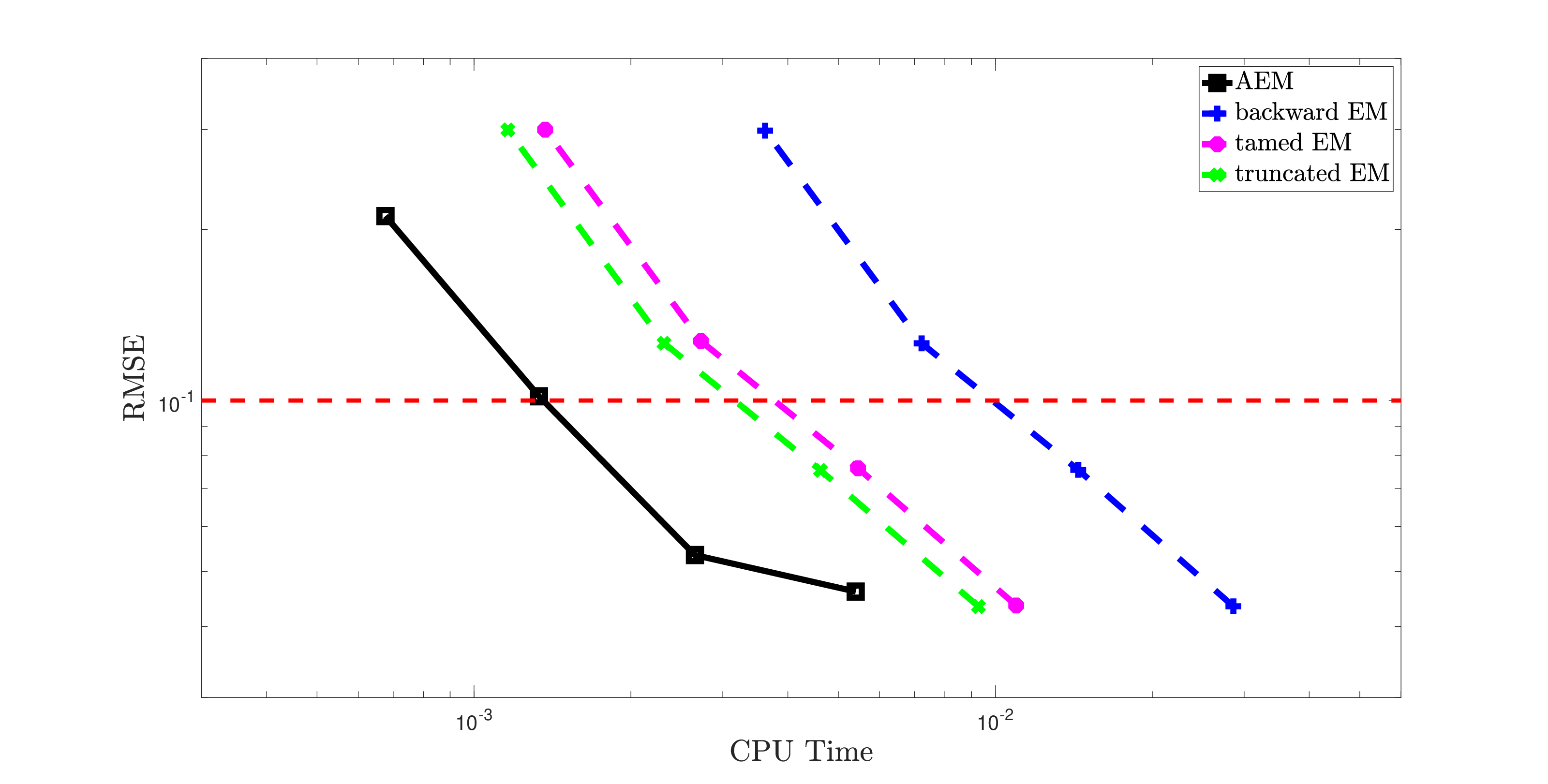}
	\end{subfigure}
	\caption{RMSE of numerical schemes for SDE \eqref{ex-stiff} with initial value $x_0 = 3$ and $M = 500$ sample paths.}
	\label{z3} 
\end{figure}
}	
\end{example}

\subsection{Nonstiff SDEs}
\begin{example}  \label{ex_a} {\rm 
Consider the two-dimensional Ginzburg-Landau equation, which describes a phase transition in superconductivity theory (see e.g., \cite{ginzburg2009on, kloeden1992numerical, hutzenthaler2015numerical})
\begin{equation}  \label{exam-1}
	dX(t) = \left( {\left( {\eta  + \frac{1}{2}{\sigma ^2}} \right)X(t) - {X^3}(t)} \right)dt + \sigma X(t)dW(t), \quad X(0) = x_0.
\end{equation}
It has an analytical solution
\begin{equation} \label{ex11}
	\begin{split}
		X(t) = \frac{{{x_0}\exp ({\eta}t + {\sigma}W(t))}}{{\sqrt {1 + 2x_0^2\int_0^t {\exp (2{\eta}s + 2{\sigma}W(s))ds} } }}, \quad t \ge 0.
	\end{split}
\end{equation}

Set $\eta = -3/2$, $\sigma = 1$, and $x_0 = [3,5]^{\rm T}$. One computes that SDE \eqref{exam-1} satisfies Assumptions \ref{A5.1} and \ref{A6.1}. It follows from Theorem \ref{th-a3} that the AEM scheme has a strong convergence rate of order $1/2$ uniformly in time. Moreover, by choosing an adaptive timestep function
\begin{equation}  \label{ex12}
	\delta(x)
	=
	\begin{cases}
		\displaystyle
		\delta_{\max}
		\wedge
		\left(\frac{\lvert x\rvert}{\lvert f(x)\rvert}\right)^{2}
		\wedge
		\left(\frac{\lvert x\rvert^{2}}{\lvert g(x)\rvert^{2}}\right)^{2},
		& x\neq \mathbf{0}, \\[1.2ex]
		\delta_{\max},
		& x=\mathbf{0}.
	\end{cases}
\end{equation}
One can verify that  \eqref{ex12} satisfies Assumption \ref{A-stability}. Hence, AEM can reproduce the moment exponential stability of the SDE \eqref{exam-1}.

Next, we carry out the numerical experiments using the AEM scheme.
Figure \ref{u4} plots the strong error of AEM with different timesteps. One observes that the convergence rate is close to 1/2, which is consistent with the theoretical results.
Figure \ref{5} presents the mean-square exponential stability of the AEM scheme for SDE \eqref{exam-1}. As observed from the figure, the logarithm of the mean square error 
$\log (E{| {\hat Y(t)} |^2})$ for both components decays linearly with time, which matches the fitted least squares line closely and confirms the theoretical exponential stability.
 
\begin{figure}[H]
	\centering
	\includegraphics[width=0.7\textwidth]{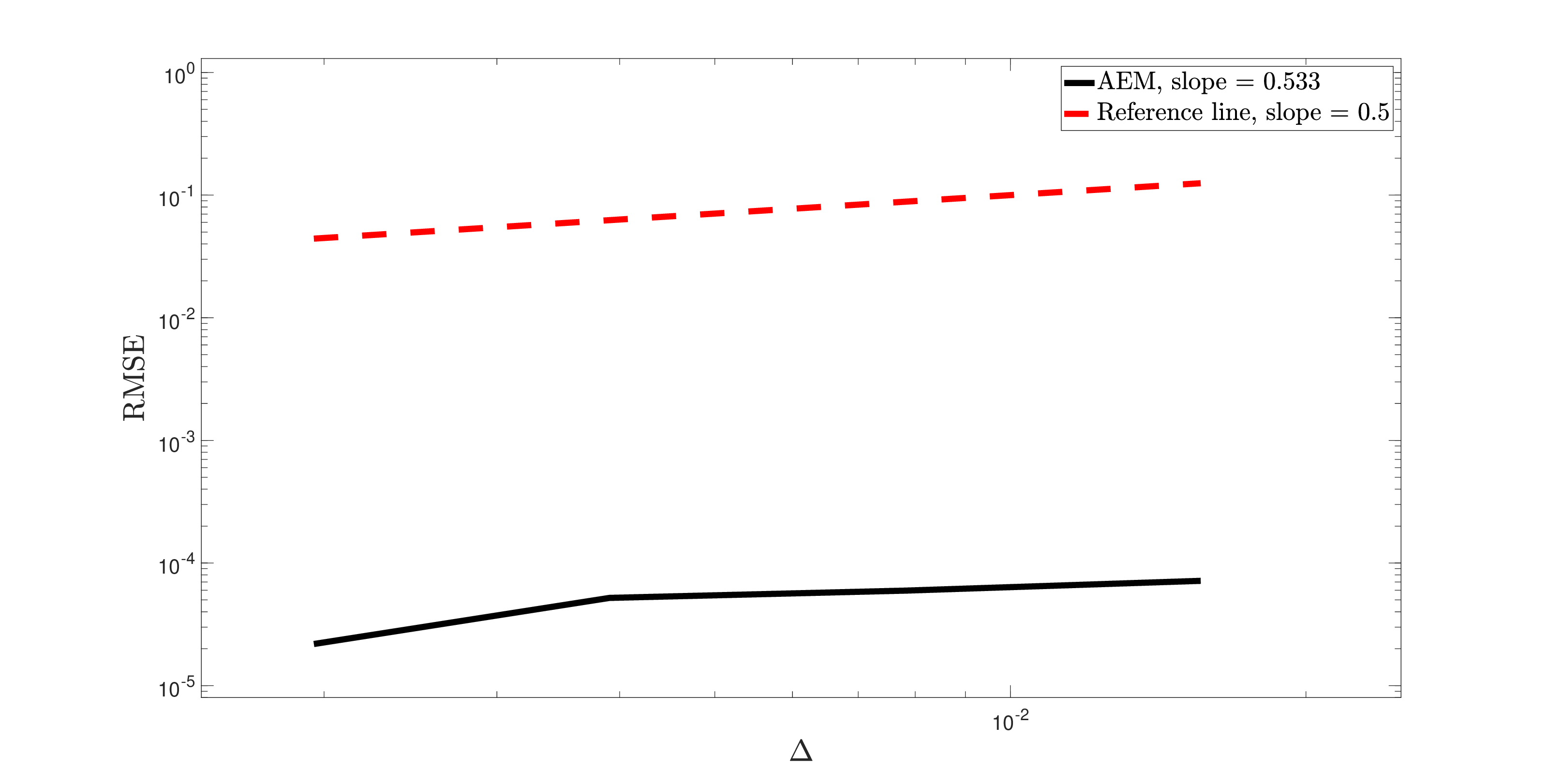}
	\caption{The strong convergence rate of AEM for SDE \eqref{exam-1}, based on initial value $x_0 = [1.5,1]^{\rm T}$, $T = 10$, $M = 3000$ sample paths and $\Delta  \in \{ {2^{ - 9}},{2^{ - 8}},{2^{ - 7}},{2^{ - 6}}\} $.}
	\label{u4}
\end{figure}

\begin{figure}[H]
	\centering
	\begin{subfigure}{0.48\textwidth}
		\centering
		\includegraphics[width=\textwidth]{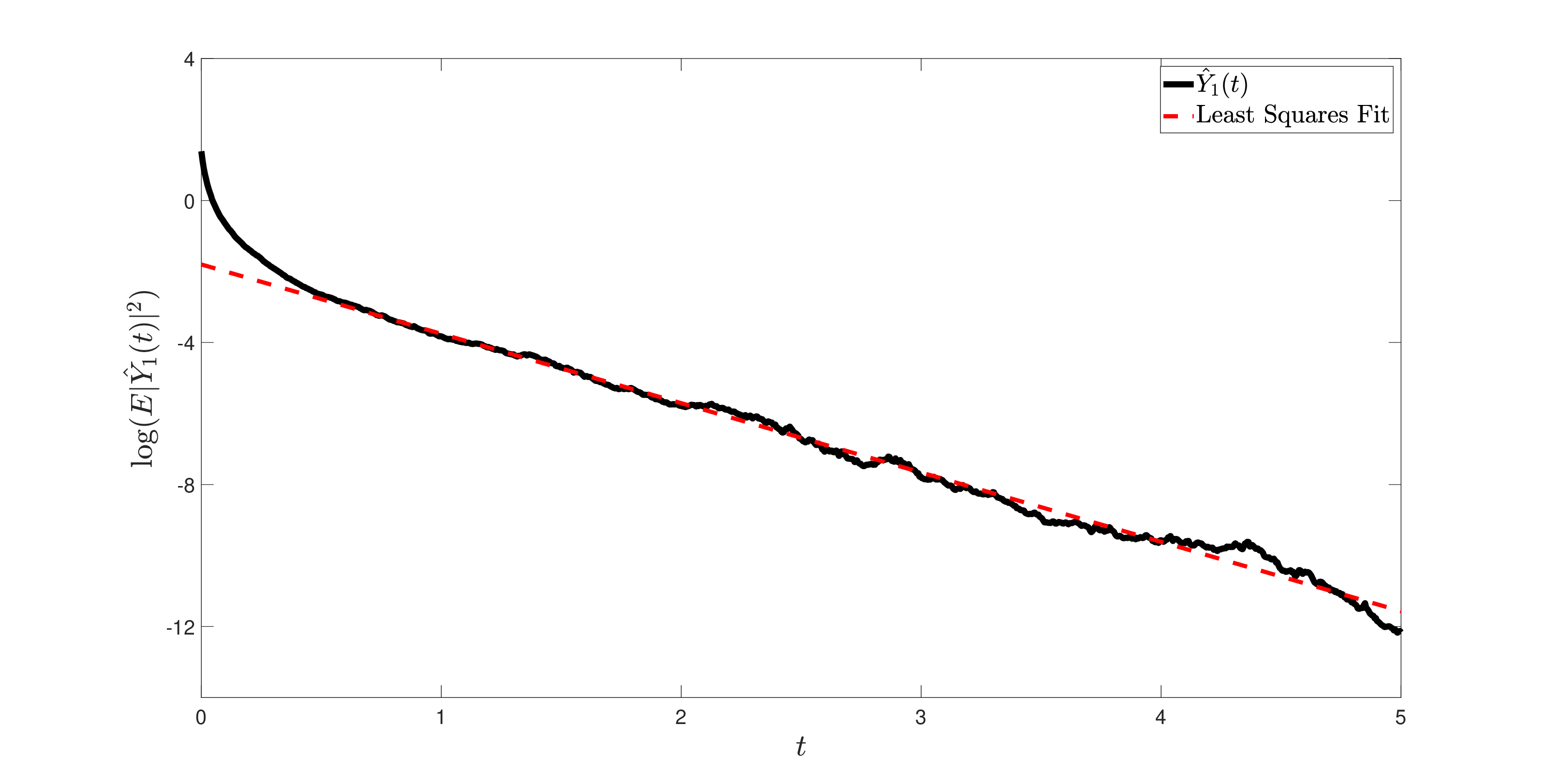}
	\end{subfigure}
	\hspace{0.01\textwidth}
	\begin{subfigure}{0.48\textwidth}
		\centering
		\includegraphics[width=\textwidth]{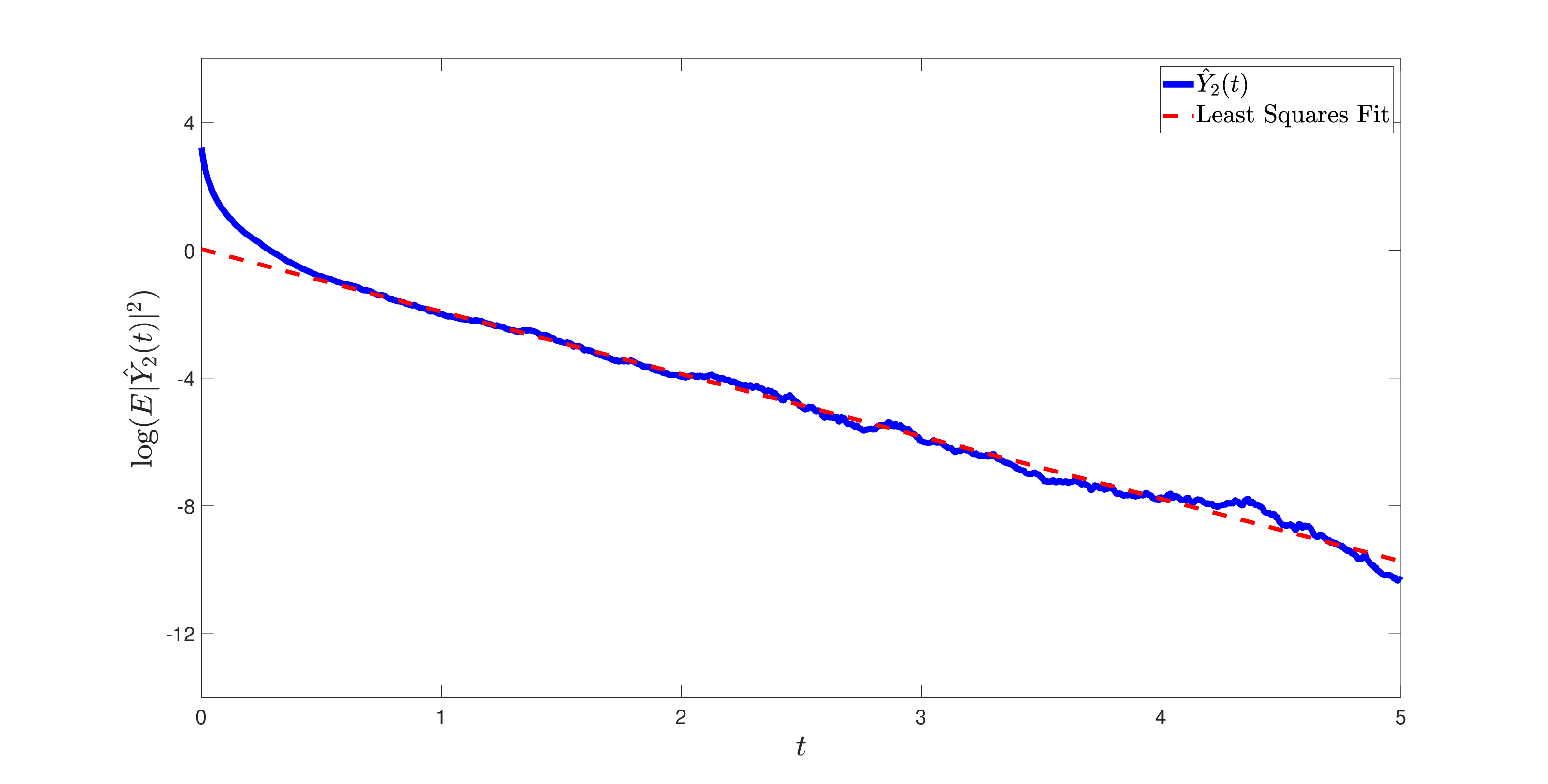}
	\end{subfigure}
	\caption{The mean-square exponential stability of AEM scheme for SDE \eqref{exam-1} with initial value $x_0 = [2,5]^{\rm T}$.}
	\label{5} 
\end{figure}

Then, we compare the performance of AEM with backward EM, tamed EM and truncated EM schemes using the uniform stepsize $\delta_{\mathrm{mean}}$.
Figure \ref{u3} shows the comparisons of RMSE against the number of timesteps and CPU time. For a given RMSE = 0.01, Figure \ref{u3} (Left) presents that AEM requires slightly fewer timesteps than the fixed-step schemes. We also observe from Figure \ref{u3} (Right) that AEM costs less CPU time than tamed EM and truncated EM schemes, and substantially less CPU time than the backward EM scheme for a given RMSE = 0.01. These results indicate that for SDE \eqref{exam-1}, AEM performs slightly better than fixed-step schemes that use the uniform stepsize.

\begin{figure}[H]
	\centering
	\begin{subfigure}{0.48\textwidth}
		\centering
		\includegraphics[width=\textwidth]{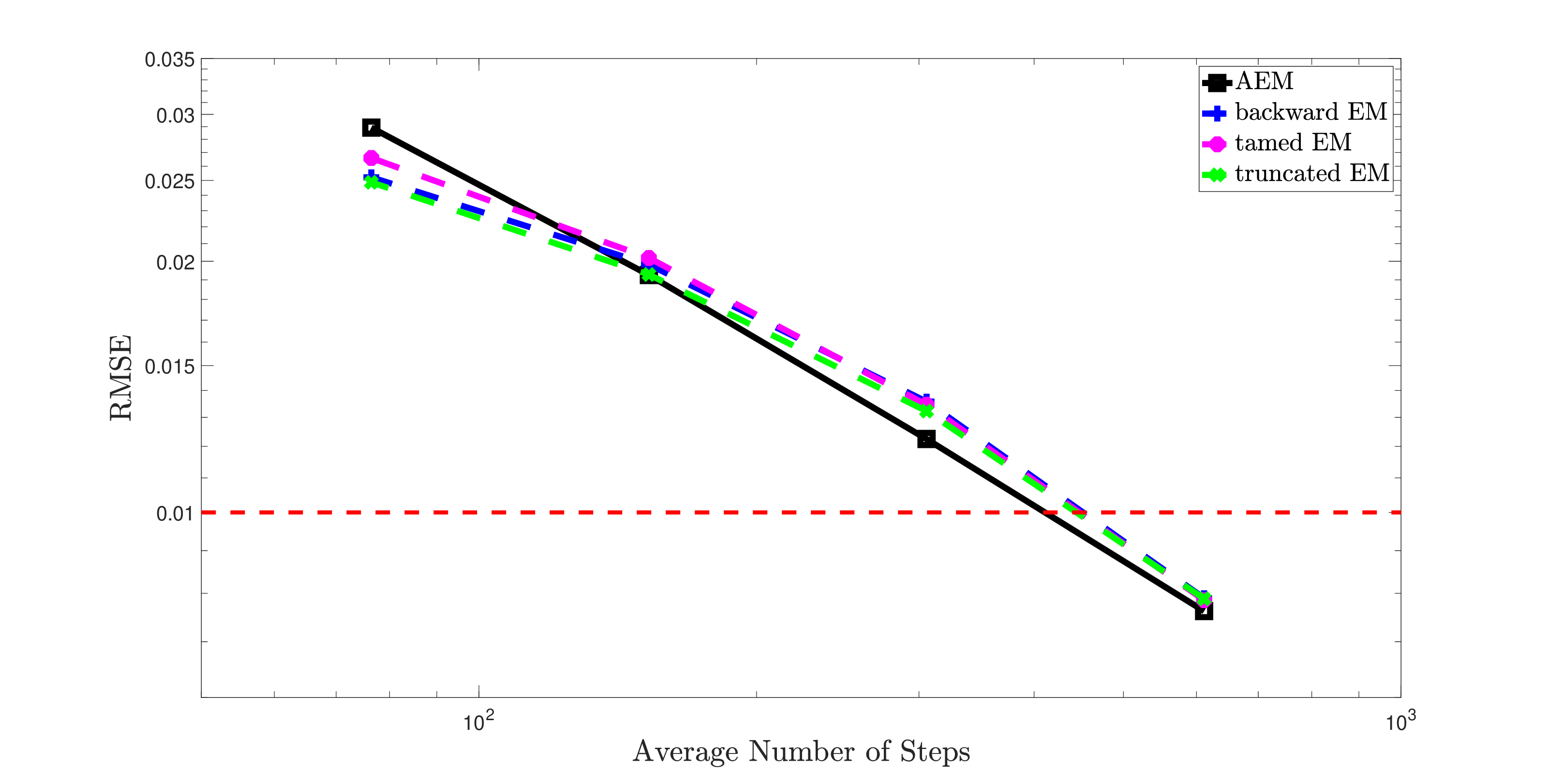}
	\end{subfigure}
	\hspace{0.01\textwidth}
	\begin{subfigure}{0.48\textwidth}
		\centering
		\includegraphics[width=\textwidth]{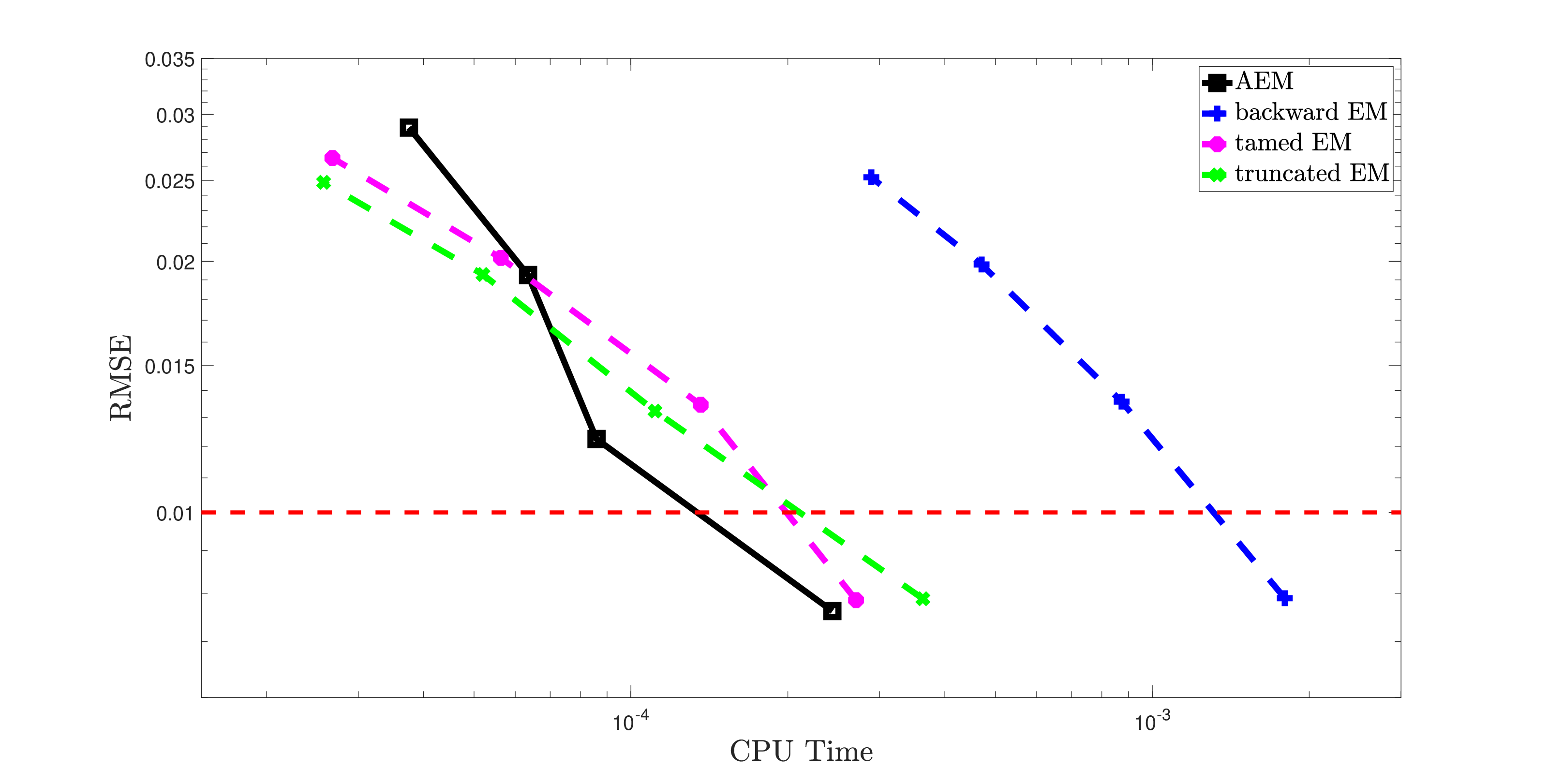}
	\end{subfigure}
	\caption{RMSE of numerical schemes for SDE \eqref{exam-1} with initial value $x_0 =[1.5,1]^{\rm T}$ and $M = 1000$ sample paths.}
	\label{u3} 
\end{figure}
}
\end{example}

\subsection{Langevin-type SDEs}
In the following, we use a high-dimensional overdamped Langevin equation and a Bayesian example with a steep prior to verify the numerical ergodicity of the AEM scheme, and the convergence of its numerical invariant distribution in $L_q$-Wasserstein distance.
Assume that the AEM numerical solution $\{Y_{t_n}\}_{n\ge0}$ is ergodic with invariant measure $\mu^\Delta$. Then, for any test function $\varphi \in C_b(\mathbb{R}^d)$, taking the first moments with respect to $\mu^\Delta$ yields
\begin{equation*}
	\int_{{\mathbb{R}^d}} {\varphi (x){\mu^\Delta }(x)dx}  = \mathop {\lim }\limits_{t \to \infty } \frac{1}{t}\int_0^t {\varphi(x(s))ds}  \approx \frac{1}{T}\int_0^T {\varphi(x(s))ds}  \approx \frac{{\sum\nolimits_{n = 0}^{{N_T}} {\varphi({Y_{{t_n}}})\delta _n^\Delta } }}{{\sum\nolimits_{n = 0}^{{N_T}} {\delta _n^\Delta } }},
\end{equation*}
where the first equality follows from the ergodicity of $\{Y_{t_n}\}_{n\ge0}$, the first approximation relation stems from choosing a sufficiently large terminal time $T > 0$, and the second approximation accounts for the sampling error.
This motivates the introduction of the weighted empirical measure \cite{lemaire2007an, leroy2024adaptive, lamberton2002recursive}
\begin{equation*}
	\mu _T^\Delta (x) = \frac{1}{{\sum\nolimits_{n = 1}^{{N_T}} {\delta _n^\Delta } }}\sum\limits_{n = 1}^{{N_T}} {\delta _n^\Delta {\mathbf{1}_{\{ {Y_{{t_n}}}\} }}(x)} ,
\end{equation*}
where $\mathbf{1}_{\{x\}}$ denotes the Dirac measure centered at $x$. 

To evaluate the performance in approximating the invariant measure, we compare our AEM scheme with both the LMS algorithm \cite{lamba2007an} and Lemaire's mthod \cite{lemaire2007an}.

\begin{example}  {\rm 
Consider the 100-dimensional overdamped Langevin SDE
\begin{equation}  \label{gradient}
	dX(t) =  - \nabla V(X(t))dt + \sqrt {2{\beta ^{ - 1}}} dW(t), \quad t \ge 0,
\end{equation}
where $V:\mathbb{R}^{100} \to \mathbb{R}$ is a potential function and $W(t)$ is a 100-dimensional Brownian motion. 

Set $V(x) = |x|^4/4 + |x|^2/2$ for $x \in \mathbb{R}^{100}$, $\beta = 10$ and initial value $X(0) = x_0 = (1,1,...,1)^{\rm T} \in \mathbb{R}^{100}$. 
One can verify that Assumptions \ref{A2.1}, \ref{A4.1} and \ref{A5.1} hold. 
The adaptive timestep function can be chosen as
\begin{equation}  \label{stepfunction}
	{\delta ^\Delta }(x) = {\left( {\frac{{1 + {{\left| x \right|}^2}}}{{1 + {{\left| {\nabla V(x)} \right|}^2}}} \wedge 1} \right)^2}\Delta , \quad  x \in \mathbb{R}^{100}.
\end{equation}
It can be verifed that \eqref{stepfunction} satisfies Assumptions \ref{A-stepbound} and \ref{A6.2}.
Hence, the Boltzmann-Gibbs distribution with density \eqref{Gibbs} is the invariant measure of SDE \eqref{gradient}. Moreover, it follows from Theorem \ref{th-a5} that the numerical invariant measure of AEM converges to the exact Boltzmann-Gibbs density \eqref{Gibbs} with $1/2$-order convergence rate.

Samples from the first $50\%$ of the time interval are discarded as burn-in to eliminate the effect of the initial condition and ensure that the retained samples represent the target distribution. 
Figure \ref{x1} presents the weighted kernel density estimates (KDEs) and cumulative distribution functions (CDFs) of the first component of AEM solutions for various terminal time $T$. 
It is evident that the numerical invariant measure asymptotically converges to the exact Gibbs invariant measure as $T$ increases.
Figure \ref{inv_convergence1} depicts the convergence rate of the numerical invariant measure for the AEM scheme. In the case of additive noise, the numerical invariant measure converges to the exact invariant measure of SDE \eqref{gradient} at a rate close to $1$, which validates our theoretical results.

\begin{figure}[H]
	\centering
	\begin{subfigure}{0.48\textwidth}
		\centering
		\includegraphics[width=\textwidth]{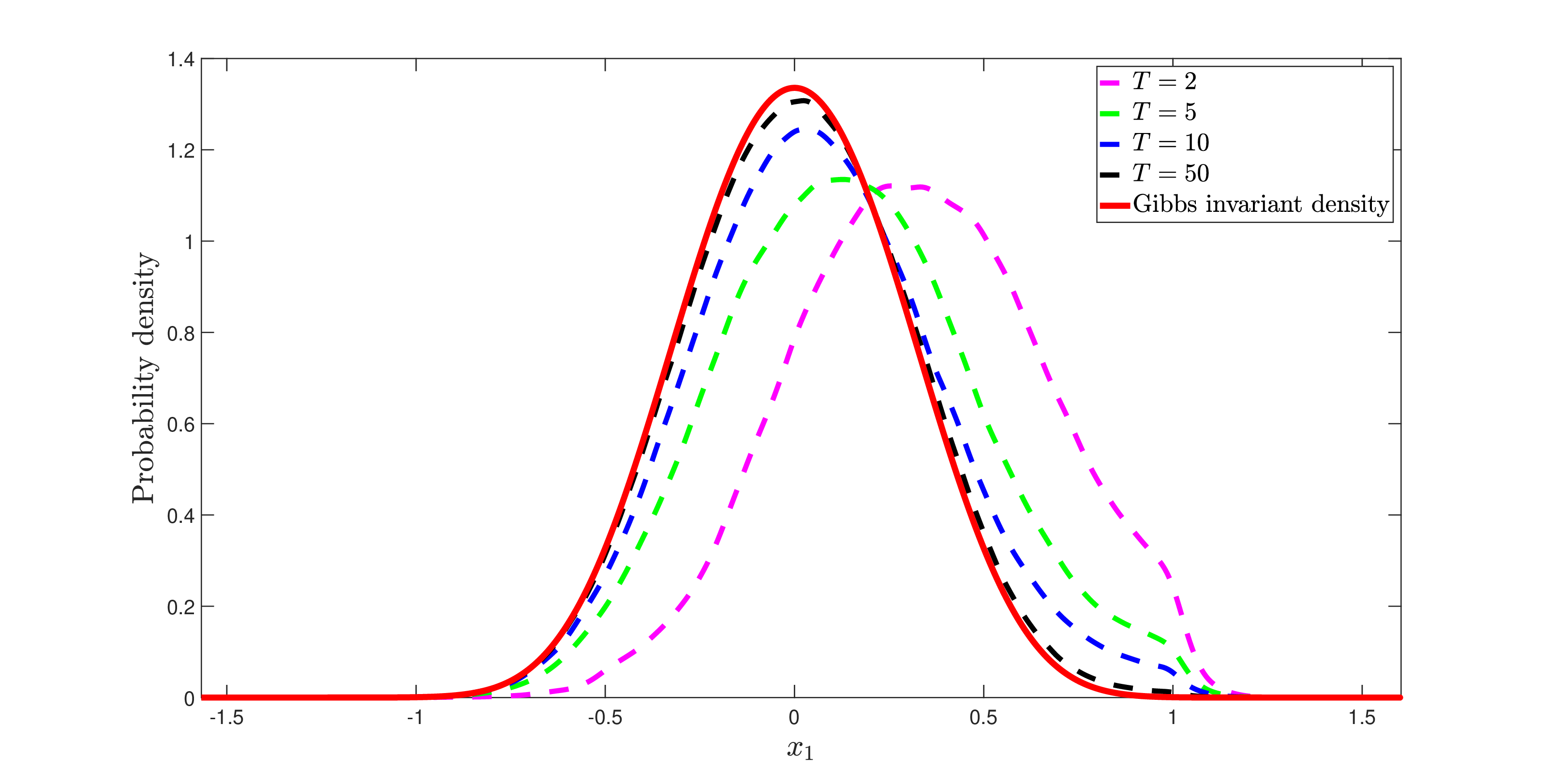}
	\end{subfigure}
	\hspace{0.01\textwidth}
	\begin{subfigure}{0.48\textwidth}
		\centering
		\includegraphics[width=\textwidth]{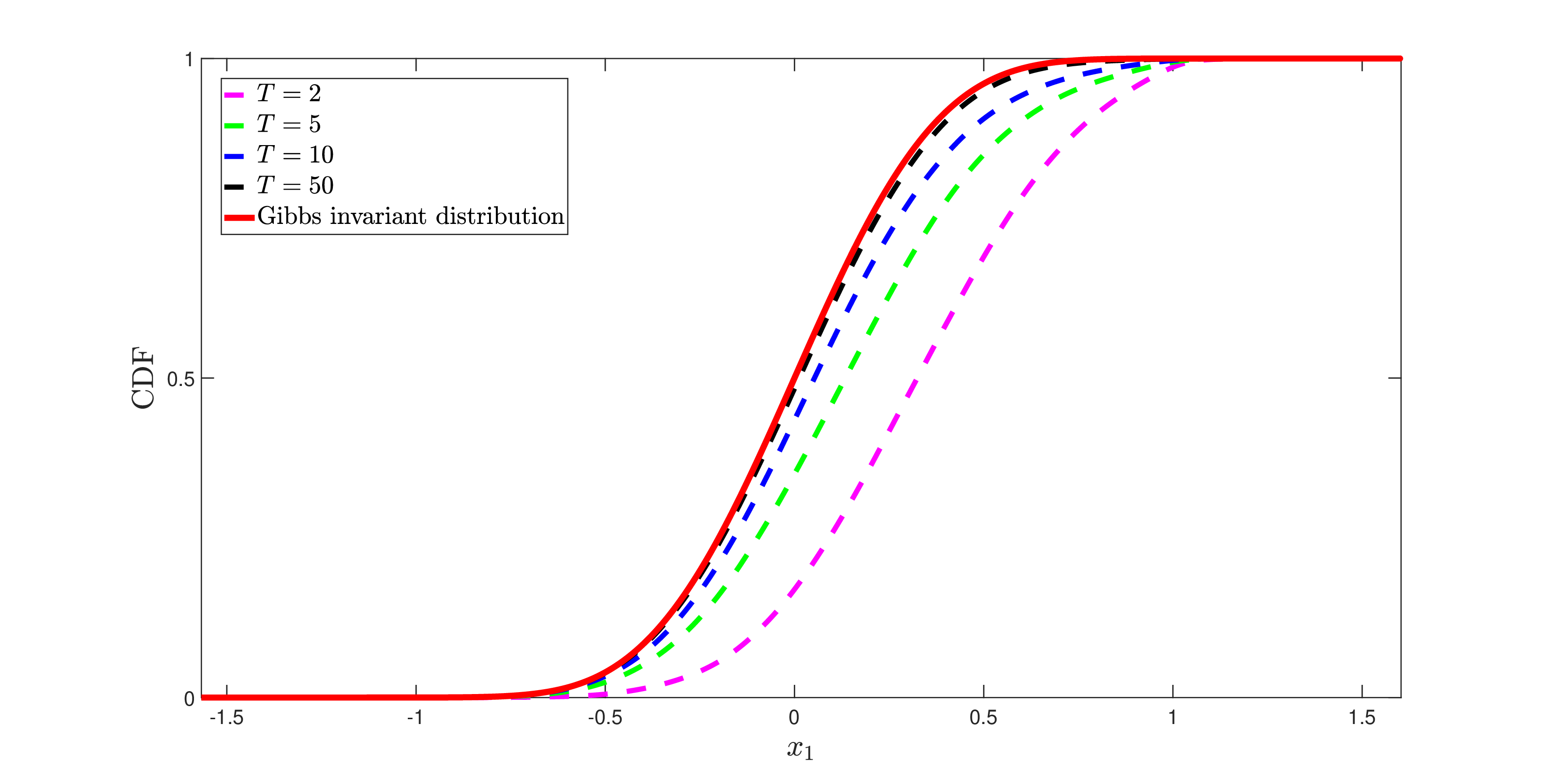}
	\end{subfigure}
	\caption{Weighted KDEs and CDFs of the first component of the AEM solutions to SDE \eqref{gradient}, based on $M = 1000$ sample paths and terminal time $T  \in \{ 2,5, 10, 50 \} $.}
	\label{x1} 
\end{figure}

\begin{figure}[H] 
	\centerline{\includegraphics[width=0.7\textwidth]{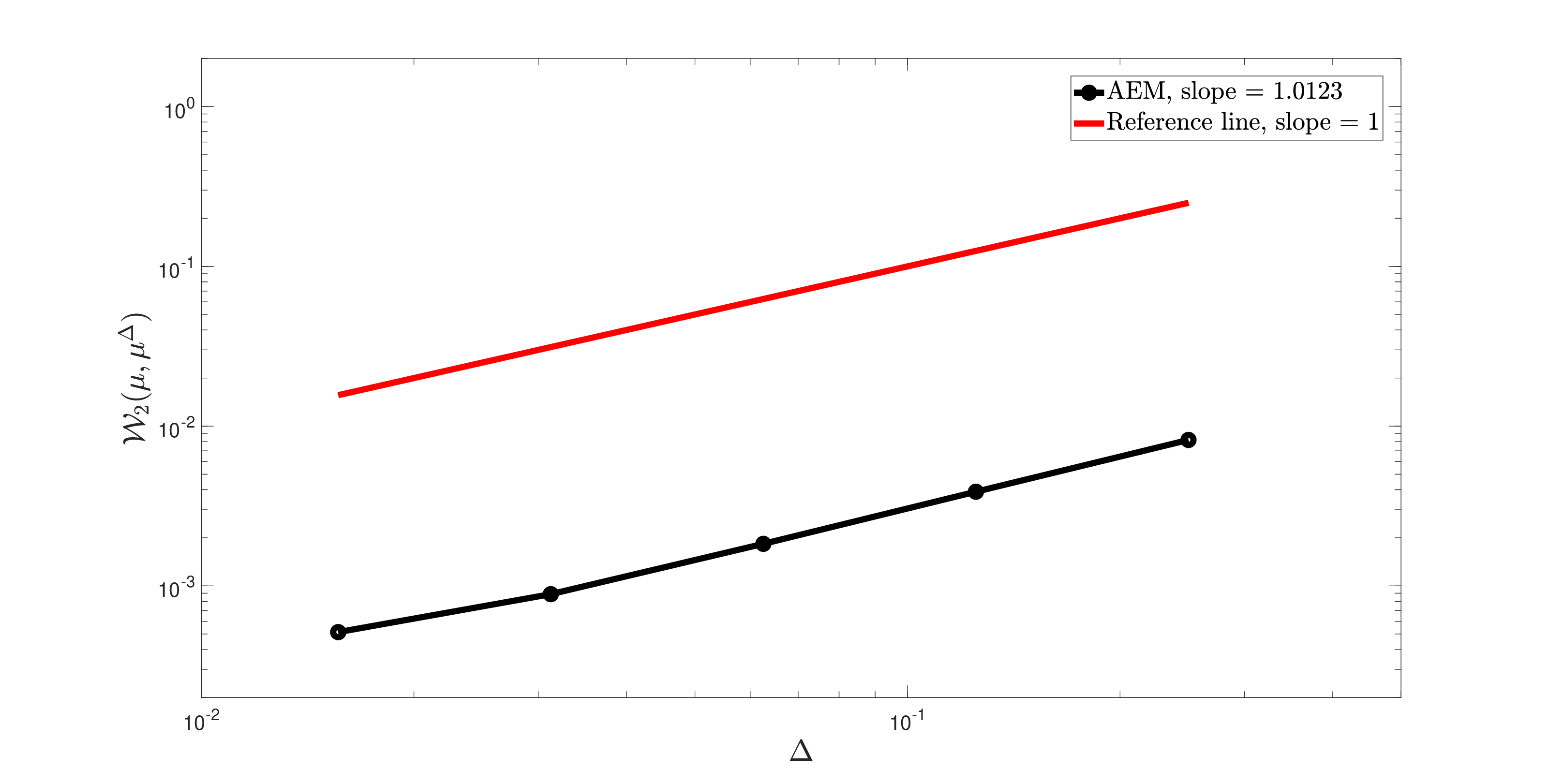}}
	\caption{The $L_2$-Wasserstein error between the numerical invariant measure generated by AEM and the theoretical Gibbs invariant measure, based on $M = 1000$ sample paths, $T = 200$ and $\Delta  \in \{ {2^{ - 6}},{2^{ - 5}},{2^{ - 4}},{2^{ - 3}},{2^{ - 2}}\} $.}
	\label{inv_convergence1}
\end{figure}

To evaluate performance, we compare AEM with the LMS algorithm and Lemaire's method. Figure \ref{x3} plots the weighted KDEs and CDFs of the first component generated by these three schemes alongside the theoretical Gibbs invariant density. One observes that all three numerical schemes accurately reconstruct the target distribution, exhibiting high consistency with the theoretical Gibbs invariant density.
Furthermore, we list the CPU time, effective sample size and $L^2$-Wasserstein errors of these schemes in Table \ref{table1}. 
Here, the effective sample size account for the post-burn-in trajectory points.
Notably, AEM achieves higher accuracy while requiring substantially less CPU time than the competing schemes, thereby demonstrating superior computational efficiency.

\begin{figure}[H]
	\centering
	\begin{subfigure}{0.48\textwidth}
		\centering
		\includegraphics[width=\textwidth]{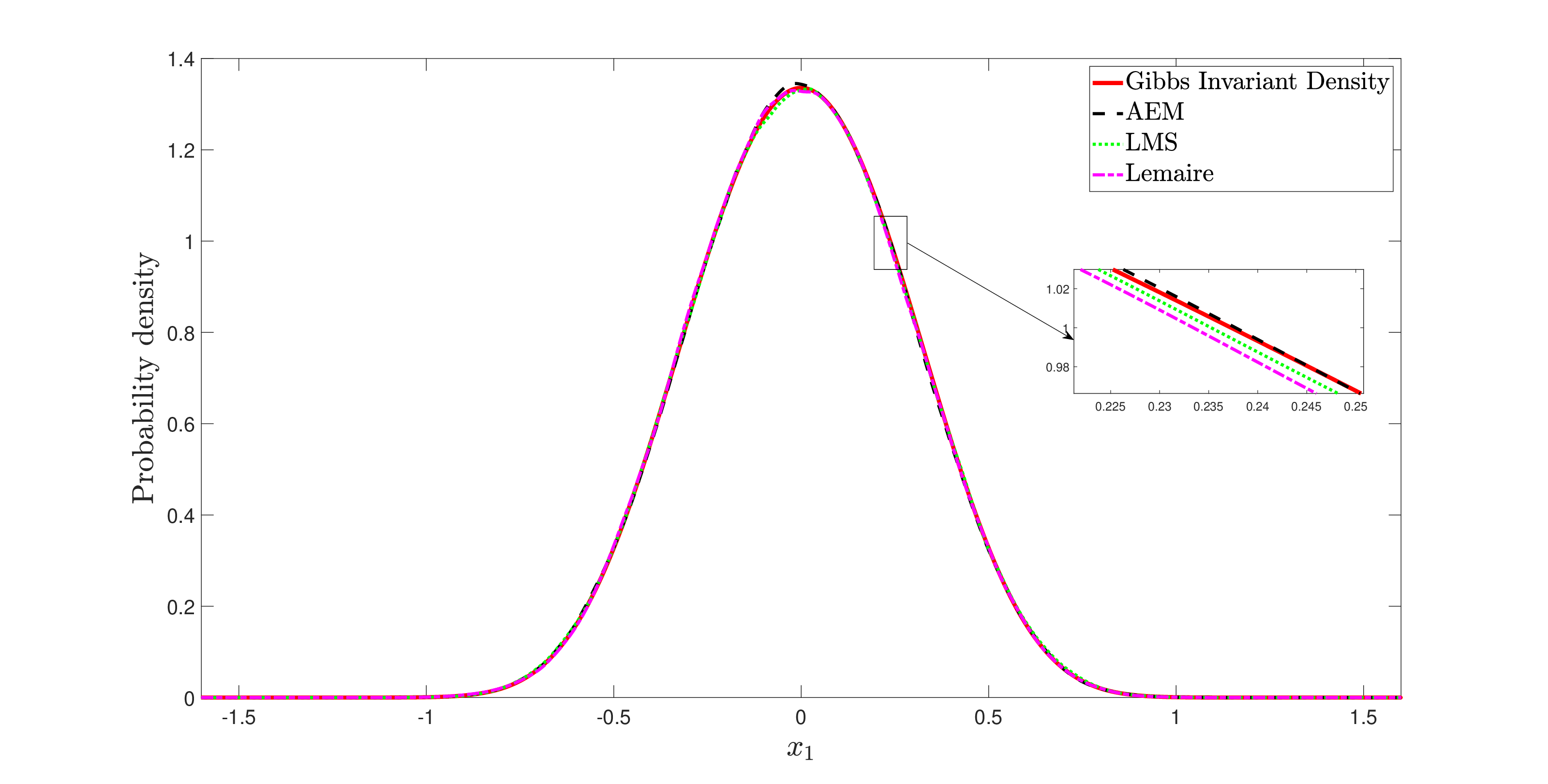}
	\end{subfigure}
	\hspace{0.01\textwidth}
	\begin{subfigure}{0.48\textwidth}
		\centering
		\includegraphics[width=\textwidth]{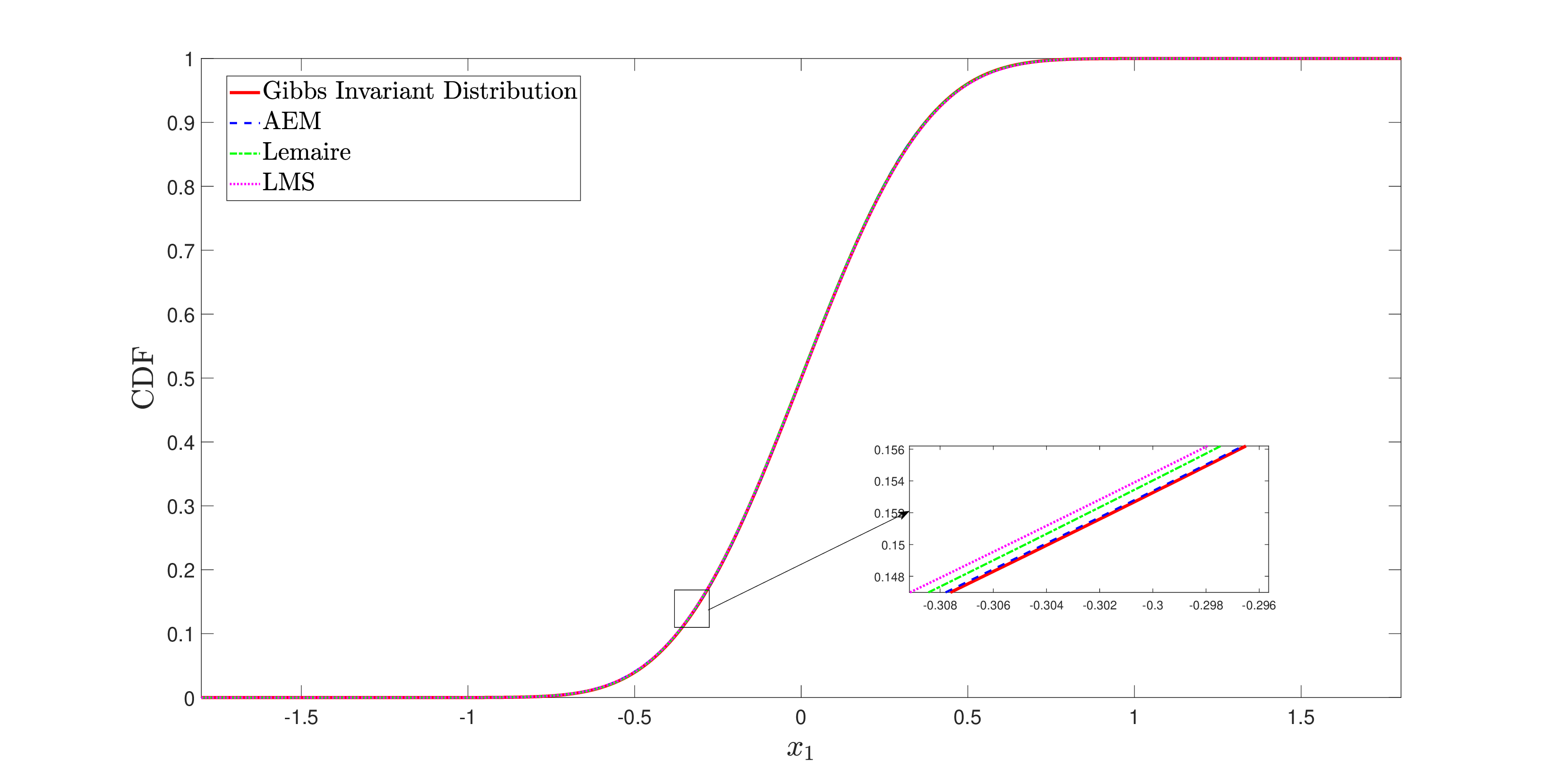}
	\end{subfigure}
	\caption{Weighted KDEs and CDFs of the first component of the numerical schemes for SDE \eqref{gradient}, based on $M = 1000$ sample paths and $T = 50$.}
	\label{x3} 
\end{figure}

\begin{table}[htbp]
	\centering
	\caption{Performance comparisons of numerical schemes for invariant measure estimation.}  
	\label{table1}
	\begin{tabular}{c c c c}
		\toprule  
		Scheme & CPU Time (s) & Effective Sample Size & $\mathcal{W}_2$ Error \\
		\midrule  
		AEM     &  4.35 &  6,852,439 & $1.3091 \times 10^{-3}$ \\
		Lemaire & 10.95 & 25,145,469 & $1.3400 \times 10^{-3}$ \\
		LMS     &  8.26 &  7,415,858 & $1.6737 \times 10^{-3}$ \\
		\bottomrule  
	\end{tabular}
\end{table}
}
\end{example}

\begin{example} {\rm 
Bayesian inference often requires sampling from posterior distributions \cite{gamerman2006markov}.
We consider a Bayesian inference problem introduced in \cite{leroy2024adaptive}. Let $\mathbf{y} = (y_1, ..., y_k)$ be a collection of independent observations satisfying ${y_i} \sim N(\theta ,1)$, $1 \le i \le k$, where $\theta \in \mathbb{R}$ is an unknown parameter. We assume a priori that the true value of $\theta$ lies in the interval $[1, 3]$. To encode this prior knowledge, we introduce a smoothly bounded prior density of the form $\pi_0 (\theta) \propto \exp ( - {(\theta  - a)^{2K}})$ with fixed $a \in [1,3]$ and $K\ge1$ controls the steepness of the confinement. The Gaussian observation model gives rise to the likelihood $\pi (\mathbf{y}| \theta) \propto \exp \left( { - \frac{1}{2}\sum\nolimits_{i = 1}^k {{{\left| {{y_i} - \theta } \right|}^2}} } \right)$. By Bayes' theorem, the resulting posterior density takes the form
\begin{equation}  \label{bey1}
	\pi (\theta |\mathbf{y}) \propto \exp \left( { - {{(\theta  - a)}^{2K}}} \right)\prod\limits_{i = 1}^k {\exp \left( { - \frac{1}{2}{{\left| {{y_i} - \theta } \right|}^2}} \right)}.
\end{equation}

To sample from this posterior distribution, we construct an overdamped Langevin process whose unique invariant measure coincides with $\pi (\theta |\mathbf{y})$. Defining the potential function 
\begin{equation*}
	V(\theta ) = \frac{1}{2}\sum\limits_{i = 1}^k {{{\left| {{y_i} - \theta } \right|}^2}}  + {(\theta  - a)^{2K}}.
\end{equation*}
The associated overdamped Langevin equation is then
\begin{equation}  \label{bay2}
	d\theta(t) =  - V'(\theta(t))dt + \sqrt 2 dW(t)
\end{equation}
with 
\begin{equation*}
	V'(x ) =  - \left( {\sum\limits_{i = 1}^k {{y_i}}  - kx  - 2K{{(x  - a)}^{2K - 1}}} \right),  \quad x\in \mathbb{R}.
\end{equation*}
By solving the corresponding Fokker-Planck equation of SDE \eqref{bay2}, 
one can easily verify that the posterior distribution \eqref{bey1} is the invariant distribution of SDE \eqref{bay2}.

Set $K = 2$, $a = 2$, $M = 1000$, $\Delta = 2^{-12}$, $T = 100$. The data are simulated using ${y_i} \sim N(1.7 ,1)$ with $k  = 10$. The timestep function is specified as in \eqref{stepfunction}.
Figure \ref{a4_1} presents the weighted KDE and CDF generated by the AEM scheme alongside the target posterior density and CDF. It is evident from Figure \ref{a4_1} that AEM provides a highly accurate approximation to the theoretical posterior distribution. 

\begin{figure}[H]
	\centering
	\begin{subfigure}{0.48\textwidth}
		\centering
		\includegraphics[width=\textwidth]{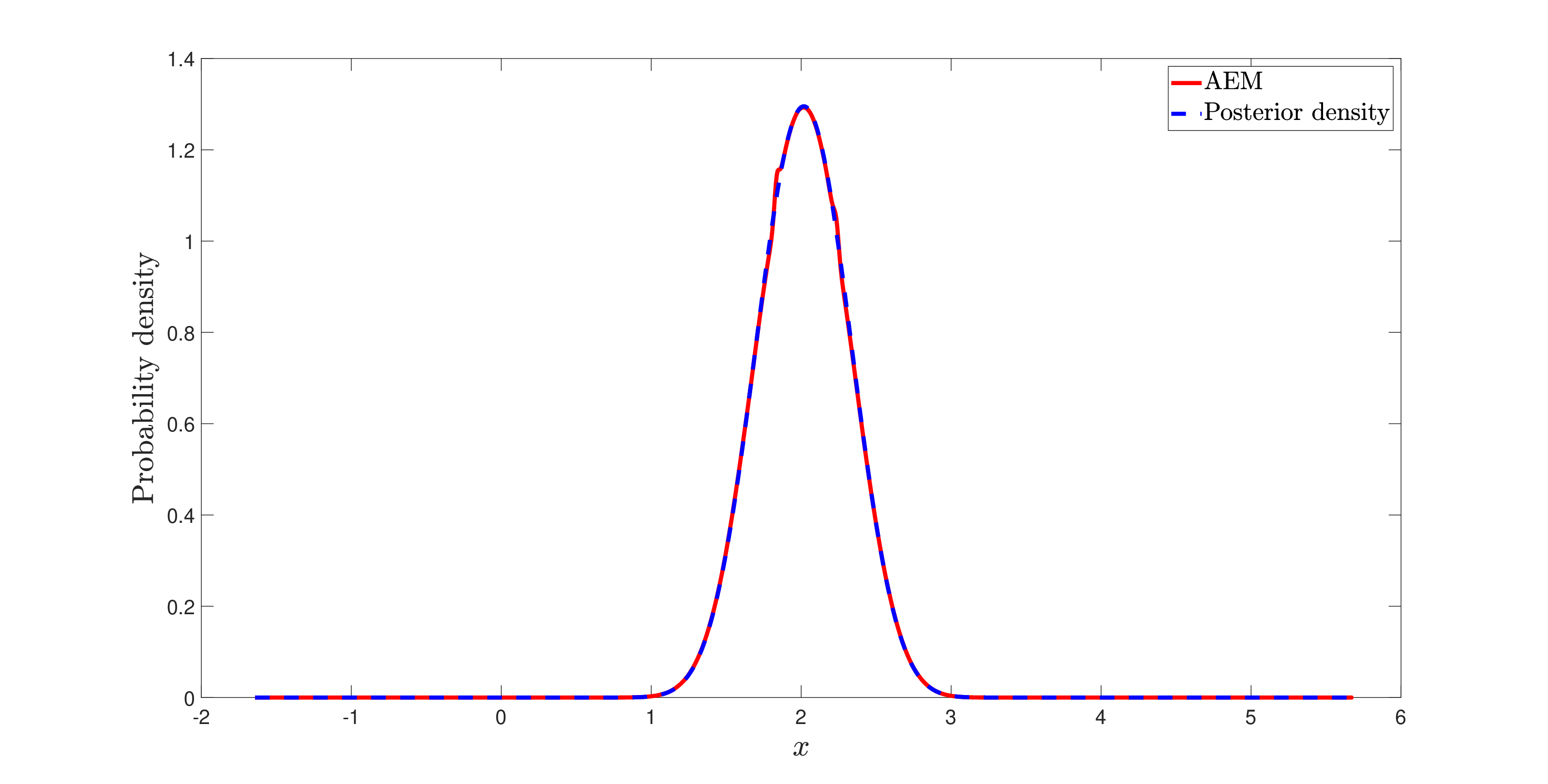}
	\end{subfigure}
	\hspace{0.01\textwidth}
	\begin{subfigure}{0.48\textwidth}
		\centering
		\includegraphics[width=\textwidth]{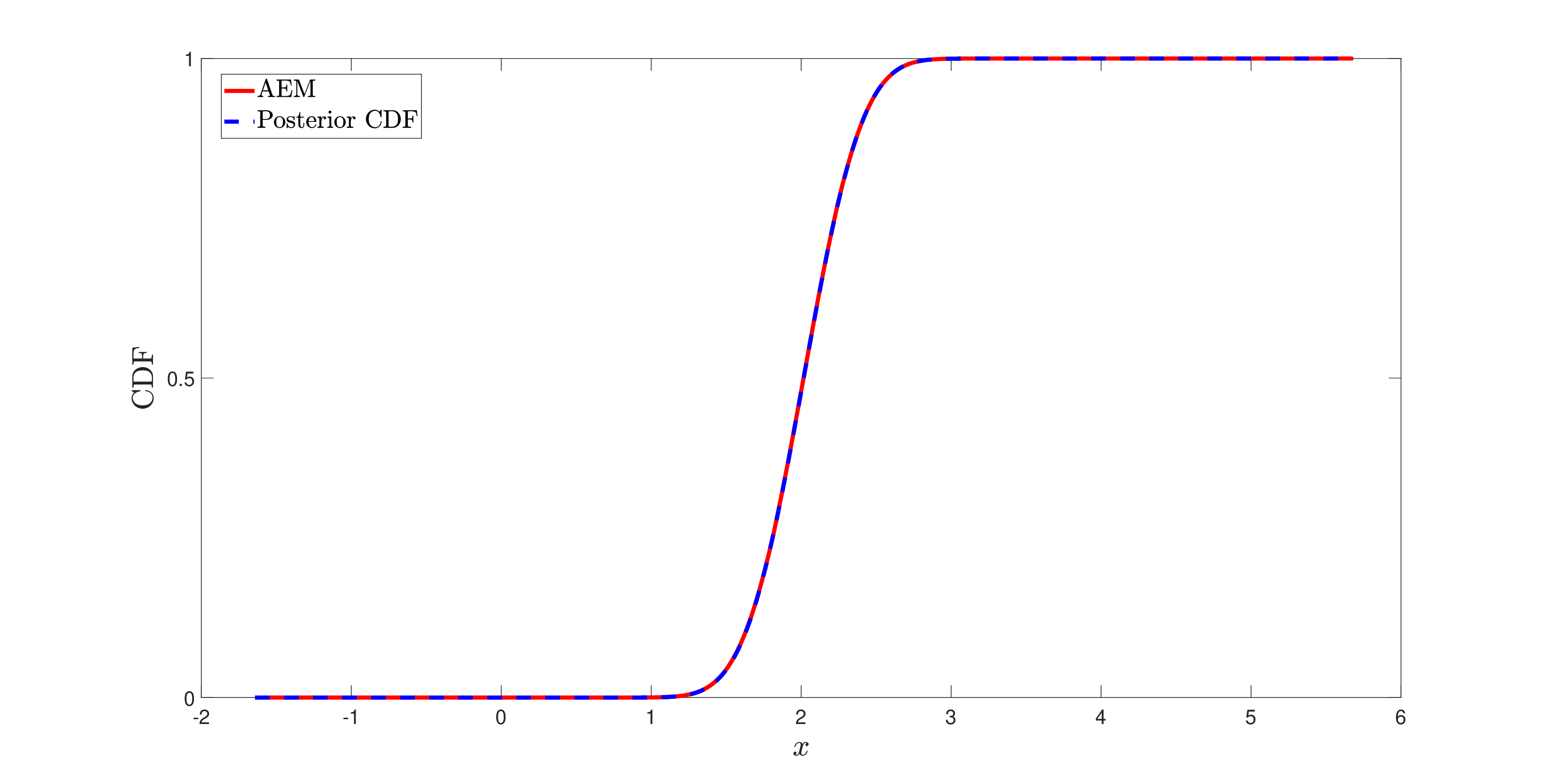}
	\end{subfigure}
	\caption{Comparisons of the weighted AEM approximations with the theoretical posterior density and CDFs for SDE \eqref{bay2}, with $T = 100$ and $M = 1000$.}
	\label{a4_1} 
\end{figure}
}
\end{example}

\section{Conclusions}	\label{sec10}
This paper proposes an adaptive time-stepping EM scheme for SDEs with non-globally Lipschitz coefficients that exhibit polynomial growth. 
We establish the moment boundedness and strong convergence of the proposed scheme on both finite time intervals and uniformly in time. Furthermore, we show that the scheme preserves the exponential stability of the underlying SDE. For long-time ergodic dynamics, we demonstrate that the numerical invariant measure of the proposed scheme converges to that of the underlying SDE at a rate of $1/2$ of convergence rate in the $L^q$-Wasserstein distance, validating its theoretical reliability for sampling target distributions in overdamped Langevin framework.
Finally, we compared the proposed scheme with several widely used fixed-step and existing adaptive methods on stiff, non-stiff and overdamped Langevin systems.
Numerical results show that the proposed scheme provides an accurate approximation of the target distribution and achieves superior accuracy and performance compared to competing methods.

\vskip 25pt
\bibliographystyle{elsarticle-num}

\bibliography{ref}

\end{document}